\documentclass[reqno]{amsart}
\usepackage[margin=1.3in]{geometry}
\usepackage{amssymb}
\usepackage{graphicx}

\usepackage[usenames, dvipsnames]{color}
\usepackage{verbatim}
\usepackage{mathrsfs}

\numberwithin{equation}{section}

\newtheorem{theorem}{Theorem}[section]

\newtheorem{lemma}[theorem]{Lemma}
\newtheorem{proposition}[theorem]{Proposition}

\theoremstyle{definition}
\newtheorem{remark}[theorem]{Remark}

\theoremstyle{definition}

\theoremstyle{definition}

\makeatletter
\def\dashint{\operatorname%
{\,\,\text{\bf-}\kern-.98em\DOTSI\intop\ilimits@\!\!}}
\makeatother

\def\\det{\text{det}}

\def\.5{\frac{1}{2}}

\def\bR{\mathbb{R}}

\def\fp{\mathfrak{p}}

\def\cA{\mathcal{A}}

\def\cP{\mathcal{P}}

\newcommand{\RN}[1]{%
  \textup{\uppercase\expandafter{\romannumeral#1}}%
}

\renewcommand{\epsilon}{\varepsilon}

\begin{document}
\title[Nonlocal parabolic equations]{Dini and Schauder estimates for nonlocal fully nonlinear parabolic equations with drifts}
\author[H. Dong]{Hongjie Dong}
\address[H. Dong]{Division of Applied Mathematics, Brown University,
182 George Street, Providence, RI 02912, USA}
\email{Hongjie\_Dong@brown.edu}
\thanks{H. Dong and H. Zhang were partially supported by the NSF under agreements DMS-1056737 and DMS-1600593.}

\author[T. Jin]{Tianling Jin}
\address[T. Jin]{Department of Mathematics, The Hong Kong University of Science and Technology,
Clear Water Bay, Kowloon, Hong Kong}
\email{tianlingjin@ust.hk}
\thanks{T. Jin was partially supported by Hong Kong RGC grant ECS 26300716 and HKUST initiation grant IGN16SC04.}

\author[H. Zhang]{Hong Zhang}
\address[H. Zhang]{Division of Applied Mathematics, Brown University,
182 George Street, Providence, RI 02912, USA}
\email{Hong\_Zhang@brown.edu}

\begin{abstract}
We obtain Dini and Schauder type estimates for concave fully nonlinear nonlocal parabolic equations of order $\sigma\in (0,2)$ with rough and non-symmetric kernels, and drift terms. We also study such linear equations with only measurable coefficients in the time variable, and obtain Dini type estimates in the spacial variable. This is a continuation of the work \cite{DZ16, DZ162} by the first and last authors.
\end{abstract}
\maketitle

\section{Introduction and main result}
The paper is a continuation of the work \cite{DZ16, DZ162} by the first and last authors, where they obtained Schauder type estimates for concave fully nonlinear nonlocal parabolic equations and Dini type estimates for concave fully nonlinear nonlocal elliptic equations. Here, we consider concave fully nonlinear nonlocal parabolic equations with Dini continuous coefficients, drifts and nonhomogeneous terms, and establish a $C^\sigma$ estimate under these assumptions.

The study of second-order equations with Dini continuous coefficients and data can date back to at least 1970s, when Burch \cite{Burch78} first considered divergence type linear elliptic equations with Dini continuous coefficients and data, and estimated the modulus of continuity of the derivatives of solutions. Later work for second-order linear or concave fully nonlinear elliptic and parabolic equations with Dini data includes, for example, Sperner \cite{Sp81}, Lieberman \cite{Lieb87}, Safonov \cite{Saf88}, Kovats \cite{Kovats97}, Bao \cite{Bao02}, Duzaar-Gastel \cite{DG02}, Wang \cite{Wang06}, Maz'ya-McOwen \cite{MM11}, Li \cite{Y.Li2016}, and many others.

The regularity theory for nonlocal elliptic and parabolic equations has been developed extensively in recent years. For example, $C^\alpha$ estimates, $C^{1,\alpha}$ estimates, Evans-Krylov type theorem, and Schauder estimates were established in the past decade. See, for instance, \cite{CS09,CS11,DK11,DK13,KL13,CD14,MP, CK15,JX152, JX15,JS15,Mou16,IJS}, and the references therein.  In particular, Mou \cite{Mou16} investigated a class of concave fully nonlinear nonlocal elliptic equations with smooth symmetric kernels, and obtained the $C^{\sigma}$ estimate under a slightly stronger assumption than the usual Dini continuity on the coefficients and data.  The author implemented a recursive Evans-Krylov theorem, which was first studied by Jin and Xiong \cite{JX15}, as well as a perturbation type argument.  By using a novel perturbation type argument, the first and last authors proved the $C^{\sigma}$ estimate for concave fully nonlinear elliptic equations in \cite{DZ162}, which relaxed the regularity assumption to simply Dini continuity and also removed the symmetry and smoothness assumptions on the kernels.

In this paper, we extend the results in \cite{DZ162} from elliptic equations to parabolic equations with drifts, that is, we study fully nonlinear nonlocal parabolic equations in the form
\begin{equation}
                            \label{eq1.1}
\partial_t u = \inf_{\beta\in \mathcal{A}}\big(L_\beta u+b_\beta Du+f_\beta\big),
\end{equation}
where $\mathcal{A}$ is an index set and for each $\beta\in \mathcal{A}$,
\begin{align*}
L_\beta u = \int_{\bR^d}\delta u(t,x,y)K_\beta(t,x,y)\,dy,
\end{align*}
\begin{equation*}
\delta u(t,x,y) =
\begin{cases}
u(t,x+y)-u(t,x)\quad& \text{for}\,\,\sigma\in (0,1);\\
u(t,x+y)-u(t,x)-y\cdot Du(t,x)\chi_{B_1}\quad& \text{for}\,\,\sigma= 1;\\
u(t,x+y)-u(t,x)-y\cdot Du(t,x)\quad& \text{for}\,\,\sigma\in(1,2),
\end{cases}
\end{equation*}
and
\begin{equation*}
K_{\beta}(t,x,y)=a_\beta(t,x,y)|y|^{-d-\sigma}.
\end{equation*}
This type of nonlocal operators was first investigated by Komatsu  \cite{Komatsu84}, Mikulevi$\check{\text{c}}$ius and Pragarauskas \cite{MP92,MP}, and later by Dong and Kim \cite{DK11,DK13}, and Schwab and Silvestre \cite{SS16}, etc.

We assume that
$$
(2-\sigma)\lambda\le a_\beta(\cdot,\cdot,\cdot)\le (2-\sigma)\Lambda\quad\forall\ \beta\in\mathcal A
$$ for some ellipticity constants $0<\lambda\le \Lambda$, and is merely measurable with respect to the $y$ variable. When $\sigma=1$, we additionally assume that
\begin{equation}
\int_{S_r}yK_{\beta}(t,x,y)\,ds_y=0,\label{eq10.58}
\end{equation}
for any $r>0$, where $S_r$ is the sphere of radius $r$ centered at the origin.

We also assume that $b_\beta\equiv 0$ when $\sigma<1$ and $b_\beta=b(t,x)$ is independent of $\beta$ when $\sigma=1$.

We say that a function $f$ is Dini continuous if its modulus of continuity $\omega_f$ is a Dini function, i.e.,
$$
\int_0^1 \omega_f(r)/r\,dr<\infty.
$$
We need the Dini continuity assumptions on the coefficients of \eqref{eq1.1}:
\begin{equation}\label{eq:modulusofcontinuity}
\begin{cases}
&\sup_{\beta\in \mathcal{A}}\displaystyle\int_{B_{2r}\setminus B_r}\big|a_{\beta}(t,x,y)-a_{\beta}(t',x',y)\big|\, dy\le \Lambda r^{d}\omega_a(\max\{|x-x'|,|t-t'|^{1/\sigma}\})\quad \forall~r>0,\\
&\sup_{\beta\in \mathcal{A}}\|f_{\beta}\|_{L_\infty(Q_1)}<\infty,\quad
\sup_{\beta\in \mathcal{A}}|f_{\beta}(t,x)-f_{\beta}(t',x')|\le \omega_f(\max\{|x-x'|,|t-t'|^{1/\sigma}\}),\\
&\sup_{\beta\in \mathcal{A}} \|b_\beta\|_{L_\infty(Q_1)}\le N_0,\quad \sup_{\beta\in \mathcal{A}}|b_\beta(t,x)-b_\beta(t',x')|\le \omega_b(\max\{|x-x'|,|t-t'|^{1/\sigma}\}),\\
&\mbox{where }N_0>0, \mbox{and }\omega_a,\omega_b, \omega_f \mbox{ are all Dini functions.}
\end{cases}
\end{equation}

In Theorem \ref{thm 1} below, $\omega_u$ denotes the modulus of continuity of $u$ in $(-1,0)\times \bR^d$, that is.
\[
|u(t,x)-u(t',x')|\le \omega_u(\max\{|x-x'|,|t-t'|^{1/\sigma}\})\ \ \mbox{for all }(t,x),(t,x')\in (-1,0)\times \bR^d.
\]
We also use the notation $C^{1,\sigma^+}(Q_1)$ to denote $C_{t,x}^{1,\sigma+\varepsilon}(Q_1)$ for some arbitrarily small  $\varepsilon>0$. This condition is only needed for $L_\beta u$ to be well defined, and may be replaced by other weaker conditions.

\begin{theorem}\label{thm 1}
Let $\sigma\in (0,2)$, $0<\lambda\le \Lambda<\infty$, and $\cA$ be an index set. Assume for each $\beta\in \cA$, $K_{\beta}$ satisfies \eqref{eq10.58} when $\sigma=1$, and the Dini continuity assumption \eqref{eq:modulusofcontinuity} holds for all $(t,x), (t',x')\in Q_1$.
Suppose $u\in C^{1,\sigma^+}(Q_1)$ is a solution of \eqref{eq1.1} in $Q_1$ and is Dini continuous in $(-1,0)\times \bR^d$.
Then we have that $\partial_t u$ is uniform continuous  and the a priori estimate:
\begin{align}
\label{eq12.17}
\|\partial_t u\|_{L_\infty(Q_{1/2})}+[u]^x_{\sigma;Q_{1/2}}\le
C\sum_{j=0}^\infty\big(2^{-j\sigma}\omega_u(2^j)+\omega_u(2^{-j})+\omega_f(2^{-j})\big),
\end{align}
where $C>0$ is a constant depending only on $d$, $\sigma$, $\lambda$, $\Lambda, N_0$, $\omega_b$, and $\omega_a$. Moreover, when $\sigma\neq 1$, we have
$$
\sup_{(t_0,x_0)\in Q_{1/2}}[u]^x_{\sigma;Q_r(t_0,x_0)}\to 0 \quad\text{as}\quad r\to 0
$$
with a decay rate depending only on $d$, $\sigma$, $\lambda$, $\Lambda$, $\omega_a$, $\omega_f$, $\omega_u$, $N_0$, and $\omega_b$. When $\sigma=1$, $Du$ is uniformly continuous in $Q_{1/2}$ with a modulus of continuity controlled by the quantities before.
\end{theorem}
This theorem improves Theorem 1.1 in \cite{DZ162} in the following two ways. First, the equation \eqref{eq1.1} is parabolic and has drift terms. Second, the right-hand side of the estimate \eqref{eq12.17} depends only on the semi-norms of $u$ and $f$, in particular, not on $\sup_{\beta\in \mathcal{A}}\|f_{\beta}\|_{L_\infty(Q_1)}$.

\begin{remark}
When $\sigma\in(1,2)$ in Theorem \ref{thm 1}, by interpolation inequalities we have that
\[
\begin{split}
[Du]^t_{\frac{\sigma-1}{\sigma}; Q_{1/2}}&\le C(\|\partial_t u\|_{L_\infty(Q_{1/2})}+[u]^x_{\sigma;Q_{1/2}})\le C\sum_{j=0}^\infty\big(2^{-j\sigma}\omega_u(2^j)+\omega_u(2^{-j})+\omega_f(2^{-j})\big).
\end{split}
\]
\end{remark}

The same proof of Theorem \ref{thm 1} can be used to prove Schauder estimates for concave fully nonlinear nonlocal parabolic equations with drifts.  To this end, we need the H\"older continuity assumptions on the coefficients of \eqref{eq1.1}:
\begin{equation}\label{eq:modulusofcontinuity2}
\begin{cases}
&\sup_{\beta\in \mathcal{A}}\displaystyle\int_{B_{2r}\setminus B_r}\big|a_{\beta}(t,x,y)-a_{\beta}(t',x',y)\big|\, dy\le
\Lambda r^{d}\max\{|x-x'|^\gamma,|t-t'|^{\gamma/\sigma}\}\quad \forall ~ r>0,\\
&\sup_{\beta\in \mathcal{A}}\|f_{\beta}\|_{L_\infty(Q_1)}<\infty,\quad
\sup_{\beta\in \mathcal{A}}|f_{\beta}(t,x)-f_{\beta}(t',x')|\le C_f\max\{|x-x'|^\gamma,|t-t'|^{\gamma/\sigma}\},\\
&\sup_{\beta\in \mathcal{A}} \|b_\beta\|_{L_\infty(Q_1)}\le N_0,\quad \sup_{\beta\in \mathcal{A}}|b_\beta(t,x)-b_\beta(t',x')|\le C_b\max\{|x-x'|^\gamma,|t-t'|^{\gamma/\sigma}\},\\
&\mbox{where }N_0, C_f, C_b>0, \mbox{and }\gamma\in (0,1).
\end{cases}
\end{equation}

Recall that we assume that $b_\beta\equiv 0$ when $\sigma<1$, and $b_\beta=b(t,x)$ is independent of $\beta$ when $\sigma=1$.

\begin{theorem}\label{thm:schauder}
Let $\sigma\in (0,2)$,  $0<\lambda\le \Lambda<\infty$, and $\cA$ be an index set. There exists $\hat\alpha$ depending only on $d,\lambda, \Lambda$ and $\sigma$ (uniform as $\sigma\to 2^-$) such that the following holds. Let $\gamma\in (0,\hat\alpha)$ such that $\sigma+\gamma<2$ is not an integer. Assume for each $\beta\in \cA$, $K_{\beta}$ satisfies \eqref{eq10.58} when $\sigma=1$, and the H\"older continuity assumptions \eqref{eq:modulusofcontinuity2} hold for all $(t,x), (t',x')\in Q_1$. Suppose $u\in C^{1+\gamma/\sigma,\sigma+\gamma}(Q_1)\cap C^{\gamma/\sigma,\gamma}((-1,0)\times \bR^d)$ is a solution of \eqref{eq1.1} in $Q_1$,
then we have the a priori estimate:
\begin{align}
\label{eq:schauderestimates}
[u]_{1+\gamma/\sigma,\sigma+\gamma;Q_{1/2}}\le C\|u\|_{\gamma/\sigma,\gamma;(-1,0)\times \bR^d}+CC_f,
\end{align}
where $C>0$ is a constant depending only on $d$, $\sigma,\gamma$, $\lambda$, $\Lambda$, $N_0$, and $C_b$.
\end{theorem}

The essential new part of Theorem 1.3 is for the case $\sigma=1$. For $\sigma<1$, Theorem 1.3 is just Theorem 1.1 in \cite{DZ16}. Even though the H\"older continuity assumption appeared slightly differently, the proof in \cite{DZ16} can be carried out with minimum modifications. For $\sigma>1$, the drift is a lower-order perturbation and the conclusion can be proved without assuming $\sigma+\gamma<2$ by using Theorem 1.1 in \cite{DZ16} and interpolation inequalities.

In the case of the linear equation:
\begin{equation}\label{eq:linear}
\partial_t u= L u+bDu+f,
\end{equation}
the estimate \eqref{eq:schauderestimates} holds for all $\gamma\in(0,\sigma)$. Again, we assume that $b\equiv 0$ when $\sigma<1$.

\begin{theorem}\label{thm:linearschauder}
Let $\sigma\in (0,2)$,  $0<\lambda\le \Lambda<\infty$, $\gamma\in (0,\sigma)$ such that $\sigma+\gamma$ is not an integer. Assume $K$ satisfies \eqref{eq10.58} when $\sigma=1$, and the H\"older continuity assumptions \eqref{eq:modulusofcontinuity2} hold for all $(t,x), (t',x')\in Q_1$. Suppose $u\in C^{1+\gamma/\sigma,\sigma+\gamma}(Q_1)\cap C^{\gamma/\sigma,\gamma}((-1,0)\times \bR^d)$ is a solution of \eqref{eq:linear} in $Q_1$, then we have the a priori estimate:
\begin{align}
\label{eq:linearschauderestimates}
[u]_{1+\gamma/\sigma,\sigma+\gamma;Q_{1/2}}\le C\|u\|_{\gamma/\sigma,\gamma;(-1,0)\times \bR^d)}+CC_f,
\end{align}
where $C>0$ is a constant depending only on $d$, $\sigma,\gamma$, $\lambda$, $\Lambda$, $N_0$, and $C_b$.
\end{theorem}

It is natural to assume that $\gamma<\sigma$ in Theorem \ref{thm:linearschauder}, since \eqref{eq:modulusofcontinuity2} will imply that $f$ is independent of $t$ if $0<\sigma<\gamma$. In many applications, $a$ will be independent of $t$ as well under the assumptions of \eqref{eq:modulusofcontinuity2} and $\sigma<\gamma$. Then, we  can always differentiate the equation \eqref{eq:linear} in $t$, and to obtain higher-order regularity in $t$ by applying the result of Theorem \ref{thm:linearschauder} above.

We are also interested in the linear equation \eqref{eq:linear} when $K$, $b$, and $f$ are Dini continuous in $x$ but only measurable in the time variable $t$, that is, they satisfy:

\begin{equation}\label{eq:modulusofcontinuity3}
\begin{cases}
&\displaystyle\int_{B_{2r}\setminus B_r}\big|a(t,x,y)-a(t,x',y)\big|\, dy
\le \Lambda r^{d}\omega_a(|x-x'|) \quad\forall~r>0,\\
&\|f\|_{L_\infty(Q_1)}<\infty,\quad
|f(t,x)-f(t,x')|\le \omega_f(|x-x'|),\\
&\|b\|_{L_\infty(Q_1)}\le N_0,\quad |b(t,x)-b(t,x')|\le \omega_b(|x-x'|),\\
&\mbox{where }N_0>0, \mbox{and }\omega_a,\omega_b, \omega_f \mbox{ are all Dini functions.}
\end{cases}
\end{equation}

In Theorem \ref{thm 2} below, $\omega_u$ denotes the modulus of continuity of $u$ in $x$ uniform for all $t$, that is,
\[
|u(t,x)-u(t,x')|\le \omega_u(|x-x'|),\quad \forall ~ (t,x), (t,x')\in (-1,0)\times \bR^d.
\]
\begin{theorem}\label{thm 2}
Let $\sigma\in (0,2)$, $0<\lambda\le \Lambda<\infty$. Assume that $K$ satisfies \eqref{eq10.58} when $\sigma=1$, and the Dini continuity assumption \eqref{eq:modulusofcontinuity3} holds
for all $(t,x), (t,x')\in Q_1$.
Suppose $u\in C^{1,\sigma^+}(Q_1)$ is a solution of \eqref{eq:linear}  in $Q_1$ and is Dini continuous in $x$ in $(-1,0)\times \bR^d$.
Then we have the a priori estimate: for $\sigma\in(0,2)$,
\begin{align}
\label{eq12.17-2}
\|\partial_t u\|_{L_\infty(Q_{1/2})}+[u]^x_{\sigma;Q_{1/2}}\le C\sum_{j=0}^\infty\big(2^{-j\sigma}\omega_u(2^j) +\omega_u(2^{-j})+\omega_f(2^{-j})\big),
\end{align}
where $C>0$ is a constant depending only on $d$, $\sigma$, $\lambda$, $\Lambda, N_0$, $\omega_b$, and $\omega_a$. Moreover, when $\sigma\neq 1$, we have
$$
\sup_{(t_0,x_0)\in Q_{1/2}}[u]^x_{\sigma;Q_r(t_0,x_0)}\to 0 \quad\text{as}\quad r\to 0
$$
with a decay rate depending only on $d$, $\sigma$, $\lambda$, $\Lambda$, $\omega_a$, $\omega_f$, $\omega_u$, $N_0$, and $\omega_b$. When $\sigma=1$, $Du$ is uniformly continuous in $x$ in $Q_{1/2}$ with a modulus of continuity controlled by the quantities before.
Also, $\partial_t u$ is uniformly continuous in $x$ in $Q_{1/2}$ with a modulus of continuity controlled by $d$, $\sigma$, $\lambda$, $\Lambda$, $\omega_a$, $\omega_f$, $\omega_u$, $N_0$, $\omega_b$, and $\|u\|_{L_\infty}$.
\end{theorem}

If $K$, $b,$ and $f$ in \eqref{eq:linear} are H\"older continuous in $x$ locally but only measurable in the time variable $t$, that is, they satisfy:
\begin{equation}\label{eq:modulusofcontinuity4}
\begin{cases}
&\displaystyle\int_{B_{2r}\setminus B_r}\big|a(t,x,y)-a(t,x',y)\big|\, dy
\le \Lambda r^{d}|x-x'|^\gamma \quad\forall~r>0,\\
&\|f\|_{L_\infty(Q_1)}<\infty,\quad
|f(t,x)-f(t,x')|\le C_f|x-x'|^\gamma,\\
&\|b\|_{L_\infty(Q_1)}\le N_0,\quad |b(t,x)-b(t,x')|\le C_b|x-x'|^\gamma,\\
&\mbox{where }N_0, C_a, C_b>0, \mbox{and } \gamma\in(0,1),
\end{cases}
\end{equation}
then we have
\begin{theorem}\label{thm:linearschauder2}
Let $\sigma\in (0,2)$,  $0<\lambda\le \Lambda<\infty$, $\gamma\in (0,1)$ such that $\sigma+\gamma$ is not an integer. Assume $K$ satisfies \eqref{eq10.58} when $\sigma=1$, and the H\"older continuity assumptions \eqref{eq:modulusofcontinuity4} hold for all $(t,x), (t,x')\in Q_1$. Suppose $u\in C^{1,\sigma+\gamma}(Q_1)\cap C^{\gamma}_x((-1,0)\times \bR^d)$ is a solution of \eqref{eq:linear} in $Q_1$, then we have the a priori estimate:
\begin{align}
\label{eq:linearschauderestimates2}
[\partial_t u]^x_{\gamma;Q_{1/2}}+[u]^x_{\sigma+\gamma,;Q_{1/2}}\le C\|u\|^x_{\gamma;(-1,0)\times \bR^d)}+CC_f,
\end{align}
where $C>0$ is a constant depending only on $d$, $\sigma,\gamma$, $\lambda$, $\Lambda$, $N_0$, and $C_b$.
\end{theorem}
Note that here we assume $\gamma\in(0,1)$ for all $\sigma\in(0,2)$, since all the estimates only involve $x$.  This theorem improves Theorem 1.1 in \cite{JX152} which does not include drifts and requires the H\"older continuity of $a$ and $f$ in the time variable $t$ as well. In the second-order case, similar results were obtained long time ago by Knerr \cite{Knerr} and Lieberman \cite{Lieb92}.

A few remarks are in order.

\begin{remark}
It is evident that all Theorems \ref{thm 1}, \ref{thm:schauder}, \ref{thm:linearschauder}, \ref{thm 2}, and \ref{thm:linearschauder2} hold for corresponding elliptic equations as well.
\end{remark}

\begin{remark}
Our proof does not tell whether the a priori estimates in Theorem \ref{thm 1} and Theorem \ref{thm 2} can be made uniformly bounded as $\sigma\to 2^-$, even if we replace $\Lambda$ by $(2-\sigma)\Lambda$ in both \eqref{eq:modulusofcontinuity2} and \eqref{eq:modulusofcontinuity3}.

\end{remark}


The ideas of our proofs are in the spirit of Campanato's approach first developed in \cite{Ca66}, which have been used in \cite{DZ162} for nonlocal fully nonlinear elliptic equations. Similar idea was also used in the literature to derive Cordes-Nirenberg type estimates, see e.g., \cite{Nirenberg}. Here, we adapt the methods in \cite{DZ162} from elliptic settings to parabolic settings, with extra efforts to deal with the drift term especially when $\sigma=1$ and some simplification of the proofs.

The key idea is that instead of estimating $C^\sigma$ semi-norm of the solution, we construct and bound certain semi-norms of the solution, see Lemma \ref{lem2.1}. When $\sigma<1$, we define such semi-norm as a series of lower-order H\"older semi-norms of $u$. In order for the nonlocal operator to be well defined, the solution needs to be smoother than $C^\sigma$. This motivates us to divide the integral domain into annuli, and use a lower-order semi-norm to estimate the integral in each annulus. 
The proof of the case when $\sigma\ge 1$ is more involved mainly due to the fact that the series of lower-order H\"older semi-norms of the solution itself is no longer sufficient to estimate the $C^\sigma$ norm. Therefore, we need to subtract a polynomial from the solution in the construction of the semi-norm. In some sense, the polynomial should be chosen to minimize the series.
It turns out that when $\sigma\ge 1$, we can make use of the first-order Taylor's expansion of the mollification of the solution.

The organization of this paper is as follows. In the next section, we introduce some notations and preliminary results that are necessary in the proof of our main theorems.  In Section \ref{sec:Dini}, we show the Dini estimates for nonlocal nonlinear parabolic equations in Theorem \ref{thm 1}.
In Section \ref{sec:schauder}, we are going to prove the Schauder estimates for equations with a drift in Theorems \ref{thm:schauder} and \ref{thm:linearschauder}. The last section is devoted to linear parabolic  equations with measurable coefficients in the time variable $t$, where Theorems \ref{thm 2} and \ref{thm:linearschauder2} are proved.
\medskip

\section*{Acknowledgments}
Part of this work was completed while H. Dong was visiting the Department of Mathematics at the Hong Kong University of Science and Technology under the HKUST-ICERM visiting fellow program, to which he is grateful   for providing  the very stimulating research environment and supports. T. Jin would like to thank Roman Shvydkoy for inspiring discussions. The authors thank the referees for their helpful comments and Luis Silvestre for pointing out an omission in the original manuscript.

\section{preliminary}
We will use the following notation:
\begin{itemize}
\item For $r>0$, $Q_r(t_0,x_0)=(t_0-r^\sigma,t_0]\times B_r(x_0)$ and $\widehat Q_r(t_0,x_0)=(t_0-r^\sigma,t_0+r^\sigma)\times B_r(x_0)$, where $B_r(x_0)\subset \bR^d$ is the ball of radius $r$ centered at $x_0$. We write $Q_r=Q_r(0,0)$ for brevity;

\item $\mathcal{P}_t$ (or $\mathcal{P}_x$) is the set of first-order polynomial in $t$ (or $x$), respectively;

\item $\mathcal{P}_1$ is the set of first-order polynomial in both $t$ and $x$.

\item $\widetilde{\mathcal{P}}$ is the set of functions in the form $a(t)+b\cdot x$, where $b$ is a constant and $a$ is a measurable function.

\item For $\alpha,\beta>0$,
\[
\begin{split}
[u]_{\alpha,\beta;Q_{r}(t_0,x_0)}&=[u]_{C^{\alpha,\beta}_{t,x}(Q_r(t_0,x_0))}.\\
[u]^x_{\beta;Q_{r}(t_0,x_0)}&=\sup_{t\in (t_0-r^\sigma,t_0)}[u(t,\cdot)]_{C^\beta(B_r(x_0))}.\\
[u]^t_{\alpha;Q_{r}(t_0,x_0)}&=\sup_{x\in B_r(x_0)}[u(\cdot,x)]_{C^\alpha((t_0-r^\sigma,t_0))}.
\end{split}
\]
If $\beta$ (or $\alpha$) is an integer, the above semi-norms mean the Lipschitz norm of $D^{|\beta|-1}$ (or $\partial_t^{|\alpha|-1}$).
If there is no subscript about the region where the norm is taken, then it means the whole domain where the function is defined (e.g., $\bR^d$ or $(-t_0,0]\times \bR^d$ for some $t_0>0$).

\item We say $u\in C^{1,\sigma^+}(Q_1)$ if $u\in C^{1,\sigma+\varepsilon}(Q_1)$ for some small  $\varepsilon>0$.

\item We will also use the following Lipschitz-Zygmund semi-norms. Let $\Omega\subset\bR^d$ be a domain, $r>0$, and $Q=(t_0-r,t_0]\times\Omega$. For $\alpha,\beta\in(0,2)$, we denote
\[
\begin{split}
[u]^x_{\mathbf{\Lambda}^\beta(Q)}&=\sup_{t\in (t_0-r^\sigma, t_0]}[u(t,\cdot)]_{\mathbf{\Lambda}^\beta(\Omega)}=\sup_{t\in (t_0-r^\sigma, t_0]}\sup_{\substack{x_1,x_2,x_3\in \Omega\\ x_1\not=x_3,\ x_1+x_3=2x_2 }} \frac{|u(t,x_1)+u(t,x_3)-2u(t,x_2)|}{|x_1-x_2|^\alpha},\\
[u]^t_{\mathbf{\Lambda}^\alpha(Q)}&=\sup_{x\in \Omega}[u(\cdot,x)]_{\mathbf{\Lambda}^\alpha((t_0-r^\sigma,t_0))}=\sup_{x\in \Omega}\sup_{\substack{t_1,t_2,t_3\in (t_0-r^\sigma,t_0]\\ t_1\not=t_3,\ t_1+t_3=2t_2 }} \frac{|u(t_1,x)+u(t_3,x)-2u(t_2,x)|}{|t_1-t_2|^\alpha},\\
[u]_{\mathbf{\Lambda}^{\alpha,\beta}(Q)}
&=\sup_{\substack{(t_1,x_1),(t_2,x_2),
(t_3,x_3)\in Q\\ (t_1,x_1)\not=(t_3,x_3),\ (t_1,x_1)+(t_3,x_3)=2(t_2,x_2) }} \frac{|u(t_1,x_1)+u(t_3,x_3)-2u(t_2,x_2)|}{|t_1-t_2|^\alpha+|x_1-x_2|^\beta}.
\end{split}
\]
\end{itemize}

We will  frequently use the following identities:
\begin{align}
	\nonumber
	&2^{j}(u(t,x+2^{-j}l)-u(t,x)) -(u(t,x+l)-u(t,x))\\
	\label{eq 2.1}
	&= \sum_{k=1}^j2^{k-1}(2u(t,x+2^{-k}l)-u(t,x+2^{-k+1}l)-u(t,x))
\end{align}
and
\begin{align}
	\nonumber
	&2^j(u(t-2^{-j},x)-u(t,x))-(u(t-1,x)-u(t,x))\\
	\label{eq 2.111}
	& = \sum_{k=1}^j2^{k-1}(2u(t-2^{-k},x)-u(t-2^{-k+1},x)-u(t,x)),
\end{align}
which hold for any  unit vector $l\in \bR^d$ and $j\in \mathbb{N}$.

\begin{lemma}
	\label{lem2.1}
Let $\alpha\in(0,\sigma)$ be a constant. Let $Q$ be a convex cylinder such that $Q_{1/2}\subset Q\subset Q_1$.

(i) When $\sigma\in(0,1)$, we have
\begin{align}
	\nonumber
&	[u]^x_{\sigma;Q}+\|\partial_t u\|_{L_\infty(Q)} \\
&\le C\sum_{k=0}^\infty 2^{k(\sigma-\alpha)}\sup_{(t_0,x_0)\in Q}\inf_{p\in \mathcal{P}_t}[u-p]_{\alpha/\sigma,\alpha;Q_{2^{-k}}(t_0,x_0)}
	\label{eq 2.2}
	+C\|u\|_{L_\infty(Q_{2^{1/\sigma}})},
\end{align}
where $C$ is a constant depending only on $d$, $\sigma$, and $\alpha$. Moreover, the modulus of continuity of $\partial_t u$ is bounded by the tail of the summation on the right-hand side of \eqref{eq 2.2}.

(ii) When $\sigma\in(1,2)$, we have
\begin{align}
	\nonumber
&[u]^x_{\sigma;Q}+\|\partial_t u\|_{L_\infty(Q)}+[Du]^t_{\frac{\sigma-1}{\sigma};Q}\\
&\le C\sum_{k=0}^\infty 2^{k(\sigma-\alpha)}\sup_{(t_0,x_0)\in Q}\inf_{p\in \mathcal{P}_1}[u-p]_{\alpha/\sigma,\alpha;Q_{2^{-k}}(t_0,x_0)}
	\label{eq 2.3a}
+ C\|u\|_{L_\infty(Q_2)},
\end{align}
where $C$ is a constant depending on $d$, $\alpha$, and $\sigma$. The modulus of continuity of $\partial_t u$ is bounded by the tail of the summation above.

(iii) When $\sigma = 1$,  we have
\begin{align}
	\nonumber
	\label{eq 2.242}
	\|Du\|_{L_\infty(Q)}+\|\partial_t u\|_{L_\infty(Q)}&\le C\sum_{k=0}^\infty 2^{k(1-\alpha)}\sup_{(t_0,x_0)\in Q}\inf_{p\in \mathcal{P}_1}[u-p]_{\alpha,\alpha;Q_{2^{-k}}(t_0,x_0)}\\
&\quad	+C\sup_{\substack{(t,x),(t',x')\in Q_2\\\max\{|t-t'|,|x-x'|\}=1}}|u(t,x)-u(t',x')|,
\end{align}
where $C$ is a constant depending on $d$, $\alpha$, and $\sigma$. The modulus of continuity of $\partial_t u$ and $Du$ are bounded by the tail of the summation above.
\end{lemma}

\begin{proof}
We first prove the estimate of $\partial_t u$ for $\sigma\in(0,2)$  by showing that
\begin{align}
	\nonumber
	\label{eq 2.241}
\|\partial_t u\|_{L_\infty(Q)}
&\le C\sum_{k = 0}^\infty2^{k(\sigma-\alpha)}\sup_{(t_0,x_0)\in Q}\inf_{p\in \mathcal{P}_t}[u-p]^t_{\alpha/\sigma;Q_{2^{-k}}(t_0,x_0)}\\
&\quad +2\sup_{(t_0,x_0)\in Q}|u(t_0-1,x_0)-u(t_0,x_0)|.
\end{align}
Indeed, from \eqref{eq 2.111},
\begin{align}
	\nonumber
	&2^j|u(t-2^{-j},x)-u(t,x)|\\
	\nonumber
	&\le |u(t-1,x)-u(t,x)| + \sum_{k=1}^\infty2^{k-1}|2u(t-2^{-k},x)-u(t-2^{-k+1},x)-u(t,x)|\\
	\label{eq 1.181}
	&\le|u(t-1,x)-u(t,x)|+C\sum_{k=1}^\infty2^{k(1-\alpha/\sigma)}
[u]^t_{\mathbf{\Lambda}^{\alpha/\sigma}(Q_{2^{-k^\ast}}(t,x))},
\end{align}
where $C$ only depends on $\sigma$ and
$k^\ast = [\frac{k-1}{\sigma}]$,
i.e., the largest integer which is smaller than $(k-1)/\sigma$. The right-hand side of the above inequality  is less than
\begin{align*}
	&|u(t-1,x)-u(t,x)|+C\sum_{k=1}^\infty2^{(k^\ast\sigma+\sigma)(1-\alpha/\sigma)}
[u]^t_{\mathbf{\Lambda}^{\alpha/\sigma}(Q_{2^{-k^\ast}}(t,x))}\\
	&\le |u(t-1,x)-u(t,x)|+C\sum_{k=1}^\infty2^{k^\ast(\sigma-\alpha)}\inf_{p\in \mathcal{P}_t}[u-p]^t_{\mathbf{\Lambda}^{\alpha/\sigma}(Q_{2^{-k^\ast}}(t,x))}.
\end{align*}
By using the definition of $k^\ast$, it is easy to see the second term on the right-hand side of the above inequality  is bounded by
\begin{align*}
C\sum_{k = 0}^\infty2^{k(\sigma-\alpha)}\sup_{(t_0,x_0)\in Q}\inf_{p\in \mathcal{P}_t}[u-p]^t_{\alpha/\sigma;Q_{2^{-k}}(t_0,x_0)}.
\end{align*}
Therefore, by sending $j\rightarrow \infty$ in  \eqref{eq 1.181}, we prove that $\|\partial_t u\|_{L_\infty(Q)}$ is bounded by the right-hand side of \eqref{eq 2.241}.
Since
\begin{align*}
&\inf_{p\in \mathcal{P}_t}[u-p]^t_{\alpha/\sigma;Q_{2^{-k}}(t_0,x_0)} \le \inf_{p\in \mathcal{P}_t}[u-p]_{\alpha/\sigma,\alpha;Q_{2^{-k}}(t_0,x_0)},\\
&\inf_{p\in \mathcal{P}_t}[u-p]^t_{\alpha/\sigma;Q_{2^{-k}}(t_0,x_0)} \le \inf_{p\in \mathcal{P}_1}[u-p]_{\alpha/\sigma,\alpha;Q_{2^{-k}}(t_0,x_0)},
\end{align*}
the right-hand side of \eqref{eq 2.241} is bounded by that of \eqref{eq 2.2}, \eqref{eq 2.3a}, and \eqref{eq 2.242}. We obtain the bound of $\|\partial_t u\|_{L_\infty(Q)}$.

Next,  we bound the modulus of continuity of $\partial_t u$ in $Q$. Assume that
$$|t-t'| + |x-x'|^\sigma\in [2^{-(i+1)},2^{-i})\quad\mbox{for some }i\ge 1.$$
From \eqref{eq 2.111}, for any $j\ge i+1$,
\begin{align*}
	&2^j(u(t-2^{-j},x)-u(t,x))-2^i(u(t-2^{-i},x)-u(t,x))\\
	&= \sum_{k = i+1}^j2^{k-1}(2u(t-2^{-k},x)-u(t-2^{-k+1},x)-u(t,x)),
\end{align*}
and the same identity holds with $(t',x')$ in place of $(t,x)$. Then we have
\begin{align*}
	&|\partial_t u(t,x)-\partial_t u(t',x')| \\
	&= \lim_{j\rightarrow \infty}|2^j(u(t-2^{-j},x)-u(t,x))-2^j(u(t'-2^{-j},x')-u(t',x'))|\\
	&\le |2^i(u(t-2^{-i},x)-u(t,x))-2^i(u(t'-2^{-i},x')-u(t',x'))|\\
	&\quad+ C\sum_{k=i+1}^\infty\sup_{(t_0,x_0)\in Q}2^{k(1-\alpha/\sigma)}
[u]^t_{\mathbf{\Lambda}^{\alpha/\sigma}(Q_{2^{-k^\ast}}(t_0,x_0))},
\end{align*}
where $k^\ast$ is defined above. By the triangle inequality,  the first term on the right-hand side is bounded by
\begin{align*}
2^i|u(t-2^{-i},x)+u(t',x')-2u(\bar{t},\bar{x})|
+2^i|u(t'-2^{-i},x')-2u(\bar{t},\bar{x})+u(t,x)|,
\end{align*}
where $\bar{t} = (t+t'-2^{-i})/2$ and $\bar{x} = (x+x')/2$. This is further bounded by
\begin{equation*}
2^{i(1-\alpha/\sigma)}\sup_{(t_0,x_0)\in Q}[u]_{\mathbf{\Lambda}^{\alpha/\sigma,\alpha}(Q_{2^{-i^\ast}}(t_0,x_0))},
\end{equation*}
where $i^\ast = [\frac{i-1}{\sigma}]$.
Therefore,
\begin{align*}
	&|\partial_t u(t,x)-\partial_t u(t',x')|\\
	&\le C\sum_{k = i}^\infty\sup_{(t_0,x_0)\in Q}2^{k(1-\alpha/\sigma)}[u]_{\mathbf{\Lambda}^{\alpha/\sigma,\alpha}
(Q_{2^{-i^\ast}}(t_0,x_0))}\\
	&\le C\sum_{k = i}^\infty\sup_{(t_0,x_0)\in Q}2^{k(1-\alpha/\sigma)}\inf_{p\in \mathcal{P}_1}[u-p]_{\alpha/\sigma,\alpha;Q_{2^{-i^*}}(t_0,x_0)},
\end{align*}
which, from the definition of $i^\ast$,  converges to  $0$ as $i\rightarrow \infty$.

In the rest of the proof, we consider the three cases separately.

{\em Case 1. $\sigma\in (0,1)$}.
The estimates of $[u]_\sigma^x$ are the same as \cite[Lemma 2.1]{DZ162} and we only provide a sketch here.
 Let $(t,x),(t,x')\in Q$ be  two different points. Suppose that $h:=|x-x'|\in (0,1)$. Since
\begin{equation*}
h^{-\sigma}|u(t,x')-u(t,x)|\le \sup_{x\in Q} h^{\alpha-\sigma}[u(t,\cdot)]_{\alpha;B_h(x)},
\end{equation*}
by taking the supremum with respect to $t,x$, and $x'$ for $h<1$ on both sides, we get
\begin{align*}
[u]^x_{\sigma;Q} &\le \sup_{(t_0,x_0)\in Q}\sup_{0<h<1} h^{\alpha-\sigma}[u]^x_{\alpha;Q_h(t_0,x_0)}\\
&\le C\sum_{k=0}^\infty 2^{k(\sigma-\alpha)}\sup_{(t_0,x_0)\in Q}[u]^x_{\alpha;Q_{2^{-k}}(t_0,x_0)}
\end{align*}
Notice that
\begin{align*}
[u]^x_{\alpha;Q_{2^{-k}}(t_0,x_0)} = \inf_{p\in \mathcal{P}_t}[u-p]^x_{\alpha;Q_{2^{-k}}(t_0,x_0)}\le \inf_{p\in \mathcal{P}_t}[u-p]_{\alpha/\sigma,\alpha;Q_{2^{-k}}(t_0,x_0)}.
\end{align*}
The proof of Case 1 is completed.

{\em Case 2. $\sigma\in(1,2)$.}
Similar to the previous case, we only provide the sketch of the proof following  that of  \cite[Lemma 2.1]{DZ162}. Let $\ell\in\bR^d$ be a unit vector and $\varepsilon\in (0,1/16)$ be a small constant to be specified later. For any two distinct points $(t,x),(t,x')\in Q$ such that $h=|x-x'|<1/2$, there exist $\bar x, \bar x'\in Q$ such that  $|x-\bar x|<\varepsilon h$, $\bar x+\varepsilon hl\in Q$, and  $|x'-\bar x'|<\varepsilon h$, $\bar x'+\varepsilon hl\in Q$. By the triangle inequality,
\begin{equation}
                        \label{eq8.43}
h^{1-\sigma}|D_\ell u(t,x)-D_\ell u(t,x')|
\le I_1+I_2+I_3,
\end{equation}
where
\begin{align*}
I_1&:=h^{1-\sigma}|D_\ell u(t,x)-(\varepsilon h)^{-1}(u(t,\bar x+\varepsilon h\ell)-u(t,\bar x))|,\\
I_2&:=h^{1-\sigma}|D_\ell u(t,x')-(\varepsilon h)^{-1}(u(t,\bar x'+\varepsilon h\ell)-u(t,\bar x'))|,\\
I_3&:=h^{1-\sigma}(\varepsilon h)^{-1}|(u(t,\bar x+\varepsilon h\ell)-u(t,\bar x))-(u(t,\bar x'+\varepsilon h\ell)-u(t,\bar x'))|.
\end{align*}
By the mean value theorem,
\begin{equation}
                                    \label{eq3.02b}
I_1+I_2\le 2^\sigma\varepsilon^{\sigma-1} [u]^x_{\sigma;Q}.
\end{equation}
Now we choose and fix an $\varepsilon$ sufficiently small depending only on $\sigma$ such that $2^\sigma\varepsilon^{\sigma-1}\le 1/2$.
Using the triangle inequality, we have
\begin{align*}
I_3 &\le Ch^{-\sigma}\big(|u(t,\bar x+\varepsilon h\ell)+u(t,\bar x')-2u(t,\tilde x)|\\
&\quad+|u(t,\bar x'+\varepsilon h\ell)+u(t,\bar x)-2u(t,\tilde x)|\big),
\end{align*}
where $\tilde x=(\bar x+\varepsilon h\ell+\bar x')/2$.
Thus,
\begin{equation}
                                        \label{eq3.03b}
I_3\le Ch^{\alpha-\sigma}[u(t,\cdot)]^x_{\mathbf{\Lambda}^\alpha(Q_h(t,\tilde x))}.
\end{equation}
Combining \eqref{eq8.43}, \eqref{eq3.02b}, and \eqref{eq3.03b}, we get
\begin{align*}
[u]^x_{\sigma;Q} \le C\sum_{k=0}^\infty 2^{k(\sigma-\alpha)}\sup_{(t_0,x_0)\in Q}\inf_{p\in \mathcal{P}_x}[u-p]^x_{\alpha;Q_{2^{-k}}(t_0,x_0)}.
\end{align*}
Because
\begin{align*}
\inf_{p\in \mathcal{P}_x}[u-p]^x_{\alpha;Q_{2^{-k}}(t_0,x_0)}\le \inf_{p\in \mathcal{P}_1}[u-p]_{\alpha/\sigma,\alpha;Q_{2^{-k}}(t_0,x_0)},
\end{align*}
we bound $[u]^x_{\sigma;Q}$ by the right-hand side of \eqref{eq 2.3a}.

It follows from \cite[Section 3.3]{Krylov97} that $[Du]^t_{\frac{\sigma-1}{\sigma};Q}$ is bounded by $\|\partial_t u\|_{L_\infty(Q)}+[u]^x_{\sigma;Q}$. Therefore, \eqref{eq 2.3a} is proved.

{\em Case 3. $\sigma=1$.}
We give the estimate of $\|Du\|_{L_\infty}$. It follows from \eqref{eq 2.1} that
\begin{align*}
&2^j\big|u(t,x+2^{-j}\ell)-u(t,x)\big|\\
&\le |u(t,x+\ell)-u(t,x)|+\sum_{k=1}^j 2^{k(1-\alpha)}[u(t,\cdot)]^x_{\mathbf{\Lambda}^\alpha(B_{2^{-k}}(x+2^{-k}\ell))}.
\end{align*}
Taking $j\to \infty$, we obtain that
\begin{align*}
\|Du\|_{L_\infty}
&\le C\sum_{k=1}^\infty2^{k(1-\alpha)}\sup_{(t_0,x_0)\in Q}\inf_{p\in \mathcal{P}_x}[u-p]^x_{\alpha;Q_{2^{-k}}(t_0,x_0)}\\
&\quad +\sup_{\substack{(t,x),(t,x') \in Q_2\\|x-x'|=1}}|u(t,x)-u(t,x')|.
\end{align*}
The estimate of the continuity of $Du$ is the same as $\partial_t u$, and thus omitted.
\end{proof}

Let $\eta$ be a smooth nonnegative function in $\bR$ with unit integral and vanishing outside $(0,1)$. For $R>0$ and $\sigma \in(0, 1)$, we define the mollification of $u$ with respect to $t$ as
$$
u^{(R)}(t,x) = \int_{\bR}u(t-R^\sigma s,x)\eta(s)\,ds.
$$
For the case $\sigma \in[1,2)$, we define $u^{(R)}$ differently by mollifying the $x$ variable as well. Let $\zeta\in C_0^\infty(B_1)$ be a radial nonnegative function with unit integral. For $R>0$, we define
\begin{align*}
u^{(R)}(t,x) = \int_{\bR^{d+1}}u(t-R^\sigma s,x-Ry)\eta(s)\zeta(y)\,dy\,ds.
\end{align*}
The following lemma is for the case $\sigma \in (0,1)$.

\begin{lemma}
	\label{lem2.4}
Let $\sigma\in (0,1)$, $\alpha\in (0,\sigma)$,  and $R>0$ be constants. Let $p_0= p_0(t)$ be the first-order Taylor expansion of $u^{(R)}$ at the origin in $t$ and $\tilde{u} = u-p_0$. Then for any integer $j\ge 0$, we have
\begin{align}
                            \label{eq11.28}
[\tilde u]_{\alpha/\sigma,\alpha;(-R^\sigma,0)\times B_{2^j R}}
\le C\inf_{p\in \mathcal{P}_t}[u-p]_{\alpha/\sigma,\alpha;(-R^\sigma,0)\times B_{2^jR}},
 \end{align}
where $C$ is a constant only depending on $d$ and $\alpha$.
\end{lemma}
\begin{proof}
It is easily seen that $\tilde u$ is invariant up to a constant if we replace $u$ by $u-p$ for any $p\in \cP_t$. Thus to prove the lemma, we only need to bound the left-hand side of \eqref{eq11.28} by

$$
C[u]_{\alpha/\sigma,\alpha;(-R^\sigma,0)\times B_{2^jR}}.
$$
Since $\tilde u=u-p(t)$, it suffices to observe that
\begin{equation*}
[p]^t_{\alpha/\sigma;(-R^\sigma,0)}=R^{\sigma-\alpha}|\partial_t u^{(R)}(0,0)|
\le C[u]^t_{\alpha/\sigma;Q_{R}}.
\end{equation*}
The lemma is proved.
\end{proof}

The following lemma is useful in dealing with the case $\sigma\in (1,2)$.
\begin{lemma}
	\label{lem2.2}
Let $\alpha\in (0,1)$ and $\sigma\in(1,2)$ be constant. Then for any $u\in C^1$ and any cylinder $Q$, we have
\begin{align}
	\nonumber
	&\sum_{k=0}^\infty 2^{k(\sigma-\alpha)}\sup_{(t_0,x_0) \in Q}[u-p_0]^x_{\alpha;Q_{2^{-k}}(t_0,x_0)}\\
	\label{eq 2.9}
	&\le C\sum_{k=0}^\infty2^{k(\sigma-\alpha)}\sup_{(t_0,x_0)\in Q}\inf_{p\in \mathcal{P}_x}[u-p]^x_{\alpha;Q_{2^{-k}}(t_0,x_0)},
\end{align}
where $p_0$ is the first-order Taylor's expansion of $u$ in the $x$ variable at $(t_0,x_0)$,  and $C>0$ is a constant depending only on $d$, $\alpha$, and $\sigma$.
\end{lemma}
\begin{proof}
Denote
$$b_k:= 2^{k(\sigma-\alpha)}\sup_{(t_0,x_0)\in Q}\inf_{p\in \mathcal{P}_x}[u-p]^x_{\alpha;Q_{2^{-k}}(t_0,x_0)}.$$
Then for any $(t_0,x_0)\in Q$ and each $k=0,1,\ldots$, there exists $\fp_k\in \mathcal{P}_x$ such that
\begin{equation*}
[u-\fp_k]^x_{\alpha;Q_{2^{-k}}(t_0,x_0)}\le 2b_k2^{-k(\sigma-\alpha)}.
\end{equation*}
By the triangle inequality, for $k\ge 1$  we have
\begin{equation}
	\label{eq 2.10}
	[\fp_{k-1}-\fp_k]^x_{\alpha;Q_{2^{-k}}(t_0,x_0)}\le 2b_k2^{-k(\sigma-\alpha)}+2b_{k-1}2^{-(k-1)(\sigma-\alpha)}.
\end{equation}
It is easily seen that
\begin{equation*}
[\fp_{k-1}-\fp_{k}]^x_{\alpha;Q_{2^{-k}}(t_0,x_0)} = |\nabla \fp_{k-1}-\nabla \fp_{k}|2^{-(k-1)(1-\alpha)},
\end{equation*}
which together with \eqref{eq 2.10} implies that
\begin{equation}
	\label{eq 2.11}
	|\nabla \fp_{k-1}-\nabla \fp_k|\le C2^{-k(\sigma-1)}(b_{k-1}+b_{k}).
\end{equation}
Since $\sum_{k}b_k<\infty$, from \eqref{eq 2.11} we see that $\{\nabla \fp_k\}$ is a Cauchy sequence in $\bR^{d}$. Let $q = q(t_0,x_0)\in \bR^d$ be the limit, which clearly satisfies for each $k\ge 0$,
\begin{equation}
	\nonumber
	|q-\nabla \fp_k|\le C\sum_{j=k}^\infty 2^{-j(\sigma-1)}b_j.
\end{equation}
By the triangle inequality, we get
\begin{align}
	\nonumber
	&[u-q\cdot x]^x_{\alpha;Q_{2^{-k}}(t_0,x_0)}\le [u-\fp_k]^x_{\alpha;Q_{2^{-k}}(t_0,x_0)} + [\fp_k-q\cdot x]^x_{\alpha;Q_{2^{-k}}(t_0,x_0)}\\
	\label{eq 2.12}
	&\le C2^{-k(1-\alpha)}\sum_{j=k}^\infty2^{-j(\sigma-1)}b_j\le C2^{-k(\sigma-\alpha)},
\end{align}
which implies that
\begin{equation}
	\nonumber
	\|u(t_0,\cdot)-u(t_0,x_0)-q\cdot(x-x_0)\|_{L_{\infty}(B_{2^{-k}}(x_0))}
\le C2^{-k}\sum_{j=k}^\infty2^{-j(\sigma-1)}b_j\le C2^{-k\sigma},
\end{equation}
and thus $q = \nabla u(t_0,x_0)$. It then follows from \eqref{eq 2.12} that
\begin{align}
	\nonumber
	&\sum_{k=0}^\infty 2^{k(\sigma-\alpha)}\sup_{(t_0,x_0)\in Q}[u-p_0]^x_{\alpha;Q_{2^{-k}}(t_0,x_0)}\\
	\nonumber	
	&\le C\sum_{k=0}^\infty 2^{k(\sigma-1)}\sum_{j=k}^\infty2^{-j(\sigma-1)}b_j\\
	\nonumber
	&= C\sum_{j=0}^\infty 2^{-j(\sigma-1)}b_j\sum_{k=0}^j 2^{k(\sigma-1)}
\le C\sum_{j=0}^\infty b_j.
\end{align}
This completes the proof of \eqref{eq 2.9}.
\end{proof}

The last lemma in this section is for the case when $\sigma \in[1,2)$.

\begin{lemma}
	\label{lem2.7}
Let $\alpha\in(0,1)$, $\sigma\in [1,2)$, and $R>0$ be constants. Let $p_0 = p_0(t,x)$ be the first-order Taylor's expansion of $u^{(R)}$ at the origin and $\tilde{u} = u-p_0$. Then for any integer $j\ge 0$, we have
\begin{align}
	\nonumber\label{eq 2.19a}
	&\sup_{\substack{(t,x),(t',x')\in (-R^\sigma,0)\times B_{2^j R}\\{(t,x)\neq (t',x'),\,0\le}|x-x'|<2R}}
\frac{|\tilde{u}(t,x)-\tilde{u}(t',x')|}
{|x-x'|^\alpha+|t-t'|^{\alpha/\sigma}}\\
	&\le C \inf_{p\in \mathcal{P}_1}[u-p]_{\alpha/\sigma,\alpha;(-R^\sigma,0)\times B_{2^{j}R}},
\end{align}
where $C>0$ is a constant depending only on $d$, $\alpha$, and $\sigma$.
\end{lemma}
\begin{proof}
It is easily seen that $\tilde u$ is invariant up to a constant if we replace $u$ by $u-p$ for any $p\in \cP_1$. Thus to show \eqref{eq 2.19a}, we only bound the left-hand side of \eqref{eq 2.19a} by
$$
C[u]_{\alpha/\sigma,\alpha;(-R^\sigma,0)\times B_{2^{j}R}}.
$$
Since $\tilde u=u-p_0$, it suffices to observe that for any two distinct $(t,x),(t',x')\in (-R^\sigma,0)\times B_{2^j R}$ such that $0\le|x-x'|<2R$,
\begin{align*}
&|p_0(t,x)-p_0(t',x')|\\
&\le |x-x'||Du^{(R)}(0,0)|+|t-t'||\partial_t u^{(R)}(0,0)|\\
&\le C|x-x'|R^{\alpha-1}[u]^x_{\alpha;Q_{R}}
+C|t-t'|R^{\sigma(\alpha/\sigma-1)}[u]^t_{\alpha/\sigma;Q_{R}}\\
&\le C\big(|x-x'|^\alpha+|t-t'|^{\alpha/\sigma}\big)
[u]_{\alpha/\sigma,\alpha;Q_{R}}.
\end{align*}
The lemma is proved.
\end{proof}

\section{Dini estimates for nonlocal nonlinear parabolic equations}\label{sec:Dini}
The following proposition is a further refinement of \cite[Corollary 4.6]{DZ16}.
\begin{proposition}\label{prop3.1}
Let $\sigma\in(0,2)$ and $0< \lambda \le \Lambda$. Assume that for any $\beta\in \mathcal{A}, K_\beta$ only depends on $y$. There is a constant $\hat{\alpha}$ depending on $d$, $\sigma$, $\lambda$, and $\Lambda$ (uniformly as $\sigma\to 2^-$) so that the following holds. Let $\alpha\in (0,\hat{\alpha})$ such that $\sigma+\alpha$ is not an integer. Suppose $u\in C^{1+\alpha/\sigma,\sigma+\alpha}(Q_1)\cap C^{\alpha/\sigma,\alpha}((-1,0)\times\bR^d)$ is a solution of
\begin{equation*}
\partial_t u = \inf_{\beta\in\mathcal{A}}(L_\beta u +f_\beta)\quad \text{in}\quad Q_1.
\end{equation*}
Then,
\begin{equation*}
[u]_{1+\alpha/\sigma,\alpha+\sigma;Q_{1/2}}\le C\sum_{j=1}^\infty2^{-j\sigma}M_j+C\sup_\beta[f_\beta]_{\alpha/\sigma,\alpha;Q_1},
\end{equation*}
where
\begin{equation*}
M_j = \sup_{^{(t,x),(t',x')\in (-1,0)\times B_{2^j},}_{(t,x)\neq (t',x'),\, 0\le|x-x'|<2}}\frac{|u(t,x)-u(t',x')|}{|x-x'|^\alpha+|t-t'|^{\alpha/\sigma}},
\end{equation*}
and $C > 0$ depends only on $d,\lambda,\Lambda, \alpha$ and $\sigma$, and is uniformly bounded as $\sigma\to 2^-$.
\end{proposition}
\begin{proof}
This follows from the proof of \cite[Corollary 4.6]{DZ16} by observing that in the estimate of $[h_\beta]_{\alpha/\sigma,\alpha;Q_1}$, the term $[u]_{\alpha/\sigma,\alpha;(-1,0)\times B_{2^j}}$ can be replaced by $M_j$. Moreover, by replacing $u$ by $u-u(0,0)$, we see that
\begin{align*}
\|u\|_{\alpha/\sigma,\alpha;(-1,0)\times B_2}\le C[u]_{\alpha/\sigma,\alpha;(-1,0)\times B_2}.
\end{align*}
The lemma is proved.
\end{proof}

In the rest of this section, we consider three cases separately.
\subsection{The case $\sigma \in(0,1)$}
\begin{proposition}
	\label{prop3.21}
Suppose that \eqref{eq1.1} is satisfied in $Q_{2^{1/\sigma}}$. Then under the conditions of Theorem \ref{thm 1}, we have
\begin{equation}
	\label{eq 8.16}
[u]_{\sigma;Q_{1/2}}^x+\|\partial_t u\|_{L_\infty;Q_{1/2}} \le C\|u\|_{\alpha/\sigma,\alpha}+C\sum_{k=1}^\infty\omega_f(2^{-k}),
\end{equation}
where $C>0$ is a constant depending only on $d$, $\lambda$, $\Lambda$, $\omega_a$, and $\sigma$. 
\end{proposition}
\begin{proof}
For $k\in \mathbb{N}$, let $v$ be the solution of
\begin{equation*}
\begin{cases}
\partial_t v = \inf_{\beta\in\mathcal{A}}(L_\beta(0,0)v+f_\beta(0,0)-\partial_t p_0)\quad& \text{in}\,\,Q_{2^{-k}},\\
v = u-p_0(t)\quad & \text{in}\,\, \big((-2^{-k\sigma},0)\times B_{2^{-k}}^c\big)\cup \big(\{t=-2^{-k\sigma}\}\times B_{2^{-k}}\big),
\end{cases}
\end{equation*}
where $L_\beta(0,0)$  is the operator with kernel $K_\beta(0,0,y)$, and $p_0(t)$ is the Taylor's expansion of $u^{(2^{-k})}$ in $t$ at the origin. Then by Proposition \ref{prop3.1} with scaling, we have
\begin{align}
[v]_{1+\alpha/\sigma,\alpha+\sigma;Q_{2^{-k-1}}} \le C\sum_{j=1}^\infty 2^{(k-j)\sigma}M_j+ C2^{k\sigma}[v]_{\alpha/\sigma,\alpha;Q_{2^{-k}}},
	\label{eq 1.121}
\end{align}
where $\alpha\in(0,\hat{\alpha})$ satisfying $\sigma+\alpha<1$,
\begin{equation*}
M_j = \sup_{\substack{(t,x),(t',x')\in (-2^{-k\sigma},0)\times B_{2^{j-k}},\\(t,x)\neq(t',x'),\,0\le |x-x'|<2^{-k+1}}}\frac{|\tilde{u}(t,x)-\tilde{u}(t',x')|}
{|x-x'|^\alpha+|t-t'|^{\alpha/\sigma}},
\end{equation*}
and $\tilde u=u-p_0$.

Let $k_0\ge 1$ be an integer to be specified and  $p_1 = p_1(t)$ be the Taylor's expansion of $v$ in $t$ at the origin. By the mean value formula,
\begin{align*}
\|v-p_1\|_{L_\infty(Q_{2^{-k-k_0}})}\le 2^{-(k+k_0)(\sigma+\alpha)}[v]_{1+\alpha/\sigma,\sigma+\alpha;Q_{2^{-k-k_0}}},
\end{align*}
and the interpolation inequality
\begin{align*}
&[v-p_1]_{\alpha/\sigma,\alpha;Q_{2^{-k-k_0}}}\\
&\le C\big(2^{(k+k_0)\alpha}\|v-p_1\|_{L_\infty(Q_{2^{-k-k_0}})}
+2^{-(k+k_0)\sigma}[v-p_1]_{1+\alpha/\sigma,\alpha+\sigma;Q_{2^{-k-k_0}}}\big),
\end{align*}
we obtain
\begin{align*}
[v-p_1]_{\alpha/\sigma,\alpha;Q_{2^{-k-k_0}}}\le C2^{-(k+k_0)\sigma}[v]_{1+\alpha/\sigma,\alpha+\sigma;Q_{2^{-k-k_0}}}.
\end{align*}
From Lemma \ref{lem2.4}, we have
\begin{align}
	 \label{eq 1.281}
	M_j\le C\inf_{p\in \mathcal{P}_t}[u-p]_{\alpha/\sigma,\alpha;(-2^{-k\sigma},0)\times B_{2^{j-k}}}
\le C[u]_{\alpha/\sigma,\alpha;(-2^{-k\sigma},0)\times \bR^d}.
\end{align}
These and \eqref{eq 1.121} give
\begin{align}
	\nonumber
&[v-p_1]_{\alpha/\sigma,\alpha;Q_{2^{-k-k_0}}}\le C2^{-(k+k_0)\sigma}\sum_{j=1}^k2^{(k-j)\sigma}M_j\\
	\label{eq 3.3}
&\quad +C2^{-(k+k_0)\sigma}[u]_{\alpha/\sigma,\alpha;(-2^{-k\sigma},0)\times \bR^d}+C2^{-k_0\sigma}[v]_{\alpha/\sigma,\alpha;Q_{2^{-k}}}\nonumber\\
&\le C2^{-(k+k_0)\sigma}\sum_{j=1}^k2^{(k-j)\sigma}M_j
+C2^{-(k+k_0)\sigma}[u]_{\alpha/\sigma,\alpha}+C2^{-k_0\sigma}[v]_{\alpha/\sigma,\alpha;Q_{2^{-k}}}.
\end{align}

Next, $w: =u-p_0-v$ satisfies
\begin{equation}
\label{eq 3.4}
\begin{cases}
w_t-\mathcal{M}^+w\le C_k\quad &\text{in}\,\, Q_{2^{-k}},\\
w_t-\mathcal{M}^-w \ge -C_k\quad &\text{in}\,\, Q_{2^{-k}},\\
w = 0\quad &\text{in}\,\, \big((-2^{-k\sigma},0)\times B_{2^{-k}}^c\big)\cup \big(\{t=-2^{-k\sigma}\}\times B_{2^{-k}}\big),
\end{cases}
\end{equation}
where $\mathcal{M}^+$ and $\mathcal{M}^-$ are the Pucci extremal operators (see, e.g., \cite{DZ16}), and
$$
C_k = \sup_{\beta\in \mathcal{A}}\|f_\beta-f_\beta(0,0)
+(L_\beta-L_\beta(0,0))u\|_{L_\infty(Q_{2^{-k}})}.
$$
It is easily seen that
\begin{align*}
	C_k
	&\le \omega_f(2^{-k}) + C\omega_a(2^{-k})\Big(\sup_{(t_0,x_0)\in Q_{2^{-k}}}\sum_{j=0}^\infty 2^{j(\sigma-\alpha)}[u]^x_{\alpha;Q_{2^{-j}}(t_0,x_0)}+\|u\|_{L_\infty}\Big).
\end{align*}
Then by the H\"older estimate \cite[Lemma 2.5]{DZ16}, we have
\begin{align}
	\nonumber
&[w]_{\alpha/\sigma,\alpha;Q_{2^{-k}}}\le C2^{-k(\sigma-\alpha)}C_k
\le C2^{-k(\sigma-\alpha)}\\
                    	\label{eq 3.5}
&\cdot\Big[\omega_f(2^{-k})+\omega_a(2^{-k})\Big(\sup_{(t_0,x_0)\in Q_{2^{-k}}}\sum_{j=0}^\infty 2^{j(\sigma-\alpha)}[u]^x_{\alpha;Q_{2^{-j}}(t_0,x_0)}+\|u\|_{L_\infty}\Big)\Big]
\end{align}
for some $\alpha>0$. This $\alpha$ can be the same as the one in \eqref{eq 1.121} since $\alpha$ is always small. 
%
By the triangle inequality and Lemma \ref{lem2.4} with $j=0$
\begin{align}
                                    \label{eq11.34}
	&[v]_{\alpha/\sigma,\alpha;Q_{2^{-k}}} \le [w]_{\alpha/\sigma,\alpha;Q_{2^{-k}}} + [u-p_0]_{\alpha/\sigma,\alpha;Q_{2^{-k}}}\nonumber\\
	&\le [w]_{\alpha/\sigma,\alpha;Q_{2^{-k}}} +C \inf_{p\in \mathcal{P}_t}[u-p]_{\alpha/\sigma,\alpha;Q_{2^{-k}}}.
\end{align}
Combining \eqref{eq 3.3}, \eqref{eq 3.5}, \eqref{eq 1.281}, and \eqref{eq11.34} yields
\begin{align}
	\nonumber	
&2^{(k+k_0)(\sigma-\alpha)}[u-p_0-p_1]_{\alpha/\sigma,\alpha;Q_{2^{-k-k_0}}}\\
	\nonumber	
&=2^{(k+k_0)(\sigma-\alpha)}[w+v-p_1]_{\alpha/\sigma,\alpha;Q_{2^{-k-k_0}}}\\
	\nonumber
	&\le C2^{-(k+k_0)\alpha}\sum_{j=1}^k2^{(k-j)\sigma}
\inf_{p\in \mathcal{P}_t}[u-p]_{\alpha/\sigma,\alpha;(-2^{-k\sigma},0)\times B_{2^{j-k}}}\\
	\nonumber
&\,\,+C2^{-(k+k_0)\alpha}[u]_{\alpha/\sigma,\alpha}+C2^{-k_0\alpha+k(\sigma-\alpha)} \inf_{p\in \mathcal{P}_t}[u-p]_{\alpha/\sigma,\alpha;Q_{2^{-k}}}+C2^{k_0(\sigma-\alpha)}\omega_f(2^{-k})\nonumber\\
	\label{eq 3.7}	
&\,\,
+C2^{k_0(\sigma-\alpha)}\omega_a(2^{-k})\Big(\sup_{(t_0,x_0)\in Q_{2^{-k}}}\sum_{j=0}^\infty 2^{j(\sigma-\alpha)}[u]^x_{\alpha;Q_{2^{-j}(t_0,x_0)}}
+\|u\|_{L_\infty}\Big).
\end{align}
Let $\ell_0\ge 1$ be an integer  such that
\[
\frac{1}{2^\sigma}+\sum_{l=\ell_0+1}^\infty\frac{1}{2^{l\sigma}}\le 1.
\]
Denote $Q^{\ell_0}=Q_{1/2}$ and for  $l=\ell_0+1,\ell_0+2, \ldots$, we denote
$$Q^l := (-\frac{1}{2^\sigma}-\sum_{j=\ell_0+1}^{l}\frac{1}{2^{j\sigma}},0]\times \Big\{x: |x|<\frac{1}{2}+\sum_{j=\ell_0+1}^{l}\frac{1}{2^{j}}\Big\}.$$
The choice of $\ell_0$ will ensure that $Q^l\subset Q_1$ for all $l\ge\ell_0$, and the definition of $Q^l$ will ensure that for $l\ge\ell_0$, $k\ge l+1$, there holds
\[
Q^l+ Q_{2^{-k}}(t_0,x_0)\subset Q^{l+1}\mbox{ for all }(t_0,x_0)\in Q^l.
\]

By translation of the coordinates, from \eqref{eq 3.7} we have for any $l\ge \ell_0$ and $k\ge l+1$,
\begin{align}
	\nonumber
&2^{(k+k_0)(\sigma-\alpha)}\sup_{(t_0,x_0)\in Q^l}[u-p_0-p_1]_{\alpha/\sigma,\alpha;Q_{2^{-k-k_0}}(t_0,x_0)}\\
	\nonumber
&\le  C2^{-(k+k_0)\alpha}\sup_{(t_0,x_0)\in Q^l}\sum_{j=0}^k2^{(k-j)\sigma}\inf_{p\in \mathcal{P}_t}[u-p]_{\alpha/\sigma,\alpha;(t_0-2^{-k\sigma},t_0)\times B_{2^{j-k}}(x_0)}+C2^{-(k+k_0)\alpha}[u]_{\alpha/\sigma,\alpha}\\\nonumber
&\quad
+C2^{k_0(\sigma-\alpha)}\Big[ \omega_f(2^{-k})+\omega_a(2^{-k})\\
&\quad\quad\cdot\Big(\sup_{(t_0,x_0)\in Q^{l+1}}\sum_{j=0}^\infty 2^{j(\sigma-\alpha)}\inf_{p\in \mathcal{P}_t}[u-p]_{\alpha/\sigma,\alpha;Q_{2^{-j}}(t_0,x_0)}	+\|u\|_{L_\infty}\Big)\Big].
	\label{eq 3.8}
\end{align}
Then we take the sum \eqref{eq 3.8} in $k = l+1,l+2, \ldots$ to obtain
\begin{align*}
	\nonumber
	&\sum_{k=l+1}^\infty 2^{(k+k_0)(\sigma-\alpha)}\sup_{(t_0,x_0)\in Q^l}\inf_{p\in \mathcal{P}_t}[u-p]_{\alpha/\sigma,\alpha;Q_{2^{-k-k_0}}(t_0,x_0)}\\
	\nonumber
	&\le C\sum_{k=l+1}^\infty 2^{-(k+k_0)\alpha} \sup_{(t_0,x_0)\in Q^l}\sum_{j=0}^k2^{(k-j)\sigma}\inf_{p\in \mathcal{P}_t}[u-p]_{\alpha/\sigma,\alpha;(t_0-2^{-k\sigma},t_0)\times B_{2^{j-k}}(x_0)} \\
&\quad +C2^{-(l+k_0)\alpha}[u]_{\alpha/\sigma,\alpha}+C2^{k_0(\sigma-\alpha)}\sum_{k=l+1}^\infty\omega_f(2^{-k})
\\
	\nonumber &\quad +C2^{k_0(\sigma-\alpha)} \cdot
\sum_{k=l+1}^\infty\omega_a(2^{-k})\Big(\sum_{j=0}^\infty 2^{j(\sigma-\alpha)}\sup_{(t_0,x_0)\in Q^{l+1}}\inf_{p\in \mathcal{P}_t}[u-p]_{\alpha/\sigma,\alpha;Q_{2^{-j}}(t_0,x_0)}+ \|u\|_{L_\infty}\Big).
\end{align*}
By switching the order of summations and then replacing $k$ by $k+j$, the first term on the right-hand side is bounded by
\begin{align*}
&C2^{-k_0\alpha}\sum_{j=0}^\infty 2^{-j\sigma}\sum_{k=j}^\infty2^{k(\sigma-\alpha)}\sup_{(t_0,x_0)\in Q^l}\inf_{p\in\mathcal{P}_t}[u-p]_{\alpha/\sigma,\alpha;
(t_0-2^{-k\sigma},t_0)\times B_{2^{j-k}}(x_0)}\\
&\le C2^{-k_0\alpha}\sum_{j=0}^\infty 2^{-j\alpha}\sum_{k=0}^\infty2^{k(\sigma-\alpha)}\sup_{(t_0,x_0)\in Q^l}\inf_{p\in\mathcal{P}_t}[u-p]_{\alpha/\sigma,\alpha;
(t_0-2^{-k\sigma},t_0)\times B_{2^{-k}}(x_0)}\\
&\le C2^{-k_0\alpha}\sum_{k=0}^\infty2^{k(\sigma-\alpha)}\sup_{(t_0,x_0)\in Q^l}\inf_{p\in \mathcal{P}_t}[u-p]_{\alpha/\sigma,\alpha;Q_{2^{-k}}(t_0,x_0)}.
\end{align*}
With the above inequality, we have
\begin{align*}
	&\sum_{k=l+1}^\infty 2^{(k+k_0)(\sigma-\alpha)}\sup_{(t_0,x_0)\in Q^l}\inf_{p\in \mathcal{P}_t}[u-p]_{\alpha/\sigma,\alpha;Q_{2^{-k-k_0}}(t_0,x_0)}\\
	&\le C2^{-k_0\alpha}\sum_{j=0}^\infty 2^{j(\sigma-\alpha)}\sup_{(t_0,x_0)\in Q^l}\inf_{p\in \mathcal{P}_t}[u-p]_{\alpha/\sigma,\alpha;Q_{2^{-j}}(t_0,x_0)}\\
	&\quad+C2^{-(l+k_0)\alpha}[u]_{\alpha/\sigma,\alpha} + C2^{k_0(\sigma-\alpha)}\sum_{k=l+1}^\infty\omega_f(2^{-k})\\ &\quad+C2^{k_0(\sigma-\alpha)}\sum_{k=l+1}^\infty\omega_a(2^{-k})\cdot
\Big(\sum_{j=0}^\infty 2^{j(\sigma-\alpha)}\sup_{(t_0,x_0)\in Q^{l+1}}\inf_{p\in \mathcal{P}_t}[u-p]_{\alpha/\sigma,\alpha;Q_{2^{-j}}(t_0,x_0)}+\|u\|_{L_\infty}\Big).
\end{align*}
The bound above together with the obvious inequality
\begin{equation*}
\sum_{j=0}^{l+k_0}2^{j(\sigma-\alpha)}\sup_{(t_0,x_0)\in Q^l}\inf_{p\in \mathcal{P}_t}[u-p]_{\alpha/\sigma,\alpha;Q_{2^{-j}}(t_0,x_0)} \le C2^{(l+k_0)(\sigma-\alpha)}[u]_{\alpha/\sigma,\alpha},
\end{equation*}
implies that
\begin{align*}
	&\sum_{j=0}^\infty 2^{j(\sigma-\alpha)}\sup_{(t_0,x_0)\in Q^l}\inf_{p\in \mathcal{P}_t}[u-p]_{\alpha/\sigma,\alpha;Q_{2^{-j}}(t_0,x_0)}\\
	&\le C2^{-k_0\alpha}\sum_{j=0}^\infty 2^{j(\sigma-\alpha)}\sup_{(t_0,x_0)\in Q^{l+1}}\inf_{p\in \mathcal{P}_t}[u-p]_{\alpha/\sigma,\alpha;Q_{2^{-j}}(t_0,x_0)}\\
	&\,\,+C2^{(l+k_0)(\sigma-\alpha)}[u]_{\alpha/\sigma,\alpha}
+C2^{k_0(\sigma-\alpha)}\sum_{k=l}^\infty
\omega_f(2^{-k})\\
	&\,\,+C2^{k_0(\sigma-\alpha)}\sum_{k=l}^\infty \omega_a(2^{-k})\cdot\Big(\sum_{j=0}^\infty 2^{j(\sigma-\alpha)}\sup_{(t_0,x_0)\in Q^{l+1}}\inf_{p\in \mathcal{P}_t}[u-p]_{\alpha/\sigma,\alpha;Q_{2^{-j}}(t_0,x_0)}
+\|u\|_{L_\infty}\Big).
\end{align*}
By first choosing $k_0$ sufficiently large, and then $\ell_0$ sufficiently large (recalling $l\ge\ell_0$), we get
\begin{align*}
&\sum_{j=0}^\infty 2^{j(\sigma-\alpha)}\sup_{(t_0,x_0)\in Q^l}\inf_{p\in \mathcal{P}_t}[u-p]_{\alpha/\sigma,\alpha;Q_{2^{-k}}(t_0,x_0)}\\
&\le \frac 14 \sum_{j=0}^\infty 2^{j(\sigma-\alpha)}\sup_{(t_0,x_0)\in Q^{l+1}}\inf_{p\in \mathcal{P}_t}[u-p]_{\alpha/\sigma,\alpha;Q_{2^{-k}}(t_0,x_0)}+C2^{(l+k_0)(\sigma-\alpha)} \|u\|_{\alpha/\sigma,\alpha}+ C\sum_{k=1}^\infty \omega_f(2^{-k}).
\end{align*}
Multiplying both sides by $4^{-l}$, taking the sum in $l$, we have
\begin{equation}\label{eq:thetrick}
4^{-l} \sum_{j=0}^\infty 2^{j(\sigma-\alpha)}\sup_{(t_0,x_0)\in Q^{l}}\inf_{p\in \mathcal{P}_t}[u-p]_{\alpha/\sigma,\alpha;Q_{2^{-k}}(t_0,x_0)}\le C \|u\|_{\alpha/\sigma,\alpha}+ C\sum_{k=1}^\infty \omega_f(2^{-k}).
\end{equation}
This, together with Lemma \ref{lem2.1} $(i)$ and the fact that $Q^{\ell_0}=Q_{1/2}$,  gives \eqref{eq 8.16} and the continuity of $\partial_t u$.
\end{proof}

\subsection{The case when $\sigma\in(1,2)$}\label{subsec:2}
\begin{proposition}
	\label{prop 3.22}
Suppose that \eqref{eq1.1} is satisfied in $Q_2$. Then under the conditions of Theorem \ref{thm 1}, we have  for $\sigma\in(1,2)$
\begin{equation}
	\label{eq 8.18}
[u]^x_{\sigma;Q_{1/2}}+[Du]^t_{\frac{\sigma-1}{\sigma};Q_{1/2}}+\|\partial_t u\|_{L_\infty(Q_{1/2})} \le C\|u\|_{\alpha/\sigma,\alpha}+C\sum_{k=1}^\infty\omega_f(2^{-k}),
\end{equation}
where $C>0$ is a constant depending only on $d$, $\lambda$, $\Lambda$, $\omega_a$, $\omega_b$, $N_0$, and $\sigma$.
\end{proposition}
\begin{proof}
For $k\in \mathbb{N}$, let $v_M$ be the solution of
\begin{equation*}
\begin{cases}
	\partial_tv_{M} = \inf_{\beta\in \mathcal{A}}\big(L_\beta(0,0)v_M+f_\beta(0,0)+b_\beta(0,0)Du(0,0)-\partial_t p_0\big)\quad \,\,\text{in}\,\,Q_{2^{-k}}\\
	v_M = g_M\qquad\qquad\qquad\qquad \,\,\text{in}\,\,\big((-2^{-k\sigma},0)\times B_{2^{-k}}^c\big)\cup \big(\{t=-2^{-k\sigma}\}\times B_{2^{-k}}\big),
\end{cases}
\end{equation*}
where $M\ge 2\|u-p_0\|_{L_\infty(Q_{2^{-k}})}$ is a constant to be specified later,
\begin{equation*}
g_M = \max(\min(u-p_0,M),-M),
\end{equation*}
and $p_0=p_0(t,x)$ is the first-order Taylor's expansion of $u^{(2^{-k})}$ at the origin.
By Proposition \ref{prop3.1}, we have
\begin{align}
	\nonumber
	&[v_M]_{1+\alpha/\sigma,\alpha+\sigma;Q_{2^{-k-1}}} \le C\sum_{j=1}^\infty 2^{(k-j)\sigma}M_j + C2^{k\sigma}[v_M]_{\alpha/\sigma,\alpha;Q_{2^{-k}}},
\end{align}
where $\alpha\in(0,\min\{\hat{\alpha},(\sigma-1)/2,2-\sigma\})$ and
\begin{align*}
M_j = \sup_{\substack{(t,x),(t',x')\in (-2^{-k\sigma},0)\times B_{2^{j-k}}\\(t,x)\neq(t',x'),\,0\le |x-x'|<2^{-k+1}}}\frac{|u(t,x)-p_0(t,x)-u(t',x^\prime)
+p_0(t',x')|}{|t-t'|^{\alpha/\sigma}+|x-x'|^\alpha}.
\end{align*}
From Lemma \ref{lem2.7} with $\sigma\in(1,2)$, it follows
\begin{align}
	\label{eq 1.311}
M_j \le C \inf_{p\in \mathcal{P}_1}[u-p]_{\alpha/\sigma,\alpha;(-2^{-k\sigma},0)\times B_{2^{j-k}}}.
\end{align}
In particular, for $j>k$, we have
\[
M_j \le C [u]_{\alpha/\sigma,\alpha;(-2^{-k\sigma},0)\times \bR^d},
\]
and thus,
\begin{align}
	\nonumber
	&[v_M]_{1+\alpha/\sigma,\alpha+\sigma;Q_{2^{-k-1}}} \le C\sum_{j=1}^\infty 2^{(k-j)\sigma}M_j + C2^{k\sigma}[v_M]_{\alpha/\sigma,\alpha;Q_{2^{-k}}}\\
	\label{eq 3.9}
	&\le C\sum_{j=1}^k 2^{(k-j)\sigma}M_j +  C [u]_{\alpha/\sigma,\alpha;(-2^{-k\sigma},0)\times \bR^d} + C2^{k\sigma}[v_M]_{\alpha/\sigma,\alpha;Q_{2^{-k}}}.
\end{align}
From \eqref{eq 3.9},  and the mean value formula (recalling $\alpha<2-\sigma$),
\begin{align*}
	&\|v_M-p_1\|_{L_\infty(Q_{2^{-k-k_0}})}\le C2^{-(k+k_0)(\sigma+\alpha)}\sum_{j=1}^k2^{(k-j)\sigma}M_j\\
	&\quad+C2^{-(k+k_0)(\sigma+\alpha)}[u]_{\alpha/\sigma,\alpha;(-2^{-k\sigma},0)\times \bR^d}+ C2^{-k\alpha-k_0(\sigma+\alpha)}[v_M]_{\alpha/\sigma,\alpha;Q_{2^{-k}}},
\end{align*}
where $p_1$ is the first-order Taylor's expansion of $v_M$ at the origin. The above inequality, \eqref{eq 3.9}, and the interpolation inequality imply
\begin{align}
	\nonumber
	&[v_M-p_1]_{\alpha/\sigma,\alpha;Q_{2^{-k-k_0}}} \le C2^{-(k+k_0)\sigma}\sum_{j=1}^k2^{(k-j)\sigma}M_j\\
	\label{eq 3.10} &\quad +C2^{-(k+k_0)\sigma}[u]_{\alpha/\sigma,\alpha;(-2^{-k\sigma},0)\times \bR^d}+C2^{-k_0\sigma}[v_M]_{\alpha/\sigma,\alpha;Q_{2^{-k}}}.
\end{align}
Next $w_M: = g_M-v_M$ satisfies
\begin{equation*}
\begin{cases}
\partial_t w_M\le \mathcal{M}^+w_M + h_M +C_k\quad &\text{in}\,\,Q_{2^{-k}}\\
\partial_t w_M\ge \mathcal{M}^- w_M + \hat{h}_M-C_k\quad &\text{in}\,\, Q_{2^{-k}}\\
w_M = 0\quad &\text{in}\,\, \big((-2^{-k\sigma},0)\times B_{2^{-k}}^c\big)\cup \big(\{t=-2^{-k\sigma}\}\times B_{2^{-k}}\big),
\end{cases}
\end{equation*}
where
\begin{equation*}
h_M: =\mathcal{M}^+\big(u-p_0-g_M\big),\quad \hat{h}_M: = \mathcal{M}^-\big(u-p_0-g_M\big).
\end{equation*}
Here
\begin{align*}
C_k = \sup_{\beta\in\mathcal{A}}\big\|f_\beta-f_\beta(0,0)+b_\beta Du-b_\beta(0,0)Du(0,0)+(L_\beta-L_\beta(0,0))u\big\|_{L_\infty(Q_{2^{-k}})}.
\end{align*}
It follows easily that
\begin{align*}
	C_k&\le \omega_f(2^{-k})+\omega_b(2^{-k})\|Du\|_{L_\infty(Q_{2^{-k}})}
+\sup_\beta\|b_\beta\|_{L_\infty}2^{-k\alpha}
[Du]_{\alpha/\sigma,\alpha;Q_{2^{-k}}}\\
	&\quad+C\omega_a(2^{-k})\Big(\sup_{(t_0,x_0)\in Q_{2^{-k}}}\sum_{j=0}^\infty 2^{j(\sigma-\alpha)}[u-p_{t_0,x_0}]^x_{\alpha;Q_{2^{-j}}(t_0,x_0)}+\|Du\|_{L_\infty(Q_{2^{-k}})}+\|u\|_{L_\infty}\Big),
\end{align*}
where $p_{t_0,x_0}=p_{t_0,x_0}(x)$ is the first-order Taylor's expansion of $u$ with respect to $x$ at $(t_0,x_0)$. From Lemma \ref{lem2.2}, we obtain
\begin{align*}
	C_k&\le \omega_f(2^{-k})+\omega_b(2^{-k})\|Du\|_{L_\infty(Q_{2^{-k}})}
+\sup_\beta\|b_\beta\|_{L_\infty}2^{-k\alpha}
[Du]_{\alpha/\sigma,\alpha;Q_{2^{-k}}}\\
	&\quad +C\omega_a(2^{-k})\Big(\sum_{j=0}^\infty 2^{j(\sigma-\alpha)}\sup_{(t_0,x_0)\in Q_{2^{-k}}}\inf_{p\in \mathcal{P}_x}[u-p]^x_{\alpha;Q_{2^{-j}}(t_0,x_0)}+\|Du\|_{L_\infty(Q_{2^{-k}})}+\|u\|_{L_\infty}\Big).
\end{align*}
By the dominated convergence theorem, it is easy to see that
\begin{equation*}
\|h_M\|_{L_\infty(Q_{2^{-k}})},\,\,\|\hat{h}_M\|_{L_\infty(Q_{2^{-k}})}\rightarrow 0\quad \text{as}\quad M\rightarrow \infty.
\end{equation*}
Thus similar to \eqref{eq 3.5}, choosing $M$ sufficiently large so that
\begin{equation*}
\|h_M\|_{L_\infty(Q_{2^{-k}})},\,\,\|\hat{h}_M\|_{L_\infty(Q_{2^{-k}})}\le C_k/2,
\end{equation*}
we have
\begin{align}
	\nonumber	\label{eq 3.11}
	&[w_M]_{\alpha/\sigma,\alpha;Q_{2^{-k}}}\\
	\nonumber
	&\le C2^{-k(\sigma-\alpha)}\Big(\omega_f(2^{-k})
+\big(\omega_b(2^{-k})+\omega_a(2^{-k})\big)
\|Du\|_{L_\infty(Q_{2^{-k}})}+2^{-k\alpha}[Du]_{\alpha/\sigma,\alpha;Q_{2^{-k}}}
	\\
	&\quad+\omega_a(2^{-k})\big(\sum_{j=0}^\infty2^{j(\sigma-\alpha)}\sup_{(t_0,x_0)\in Q_{2^{-k}}}\inf_{p\in \mathcal{P}_x}[u-p]^x_{\alpha;Q_{2^{-j}}(t_0,x_0)} + \|u\|_{L_\infty}\big)\Big).
\end{align}
Clearly,
\begin{align}
                                \label{eq1.48}
\inf_{p\in \mathcal{P}_x}[u-p]^x_{\alpha;Q_{2^{-j}}(t_0,x_0)}\le \inf_{p\in \mathcal{P}_1}[u-p]_{\alpha/\sigma,\alpha;Q_{2^{-j}}(t_0,x_0)}.
\end{align}
From the triangle inequality and Lemma \ref{lem2.7} with $j=0$,
\begin{align*}
	[v_M]_{\alpha/\sigma,\alpha;Q_{2^{-k}}} \le [w_M]_{\alpha/\sigma,\alpha;Q_{2^{-k}}} + [u-p_0]_{\alpha/\sigma,\alpha;Q_{2^{-k}}}\le[w_M]_{\alpha/\sigma,\alpha;Q_{2^{-k}}} + C\inf_{p\in \mathcal{P}_1}[u-p]_{\alpha/\sigma,\alpha;Q_{2^{-k}}}.
\end{align*}
We define for all $l= 1,2,\cdots$, that $Q^l=Q_{1-2^{-l}}$.
Combining \eqref{eq 3.10}, \eqref{eq 3.11} with \eqref{eq1.48}, and \eqref{eq 1.311}, similar to \eqref{eq 3.8}, we get that for all $l\ge 1$ and $k\ge l+1$,
\begin{align}
	&2^{(k+k_0)(\sigma-\alpha)}\sup_{(t_0,x_0)\in Q^l}\inf_{p\in \mathcal{P}_1}[u-p]_{\alpha/\sigma,\alpha;Q_{2^{-(k_0+k)}}(t_0,x_0)} \nonumber\\
	&\le C2^{-(k+k_0)\alpha}\sup_{(t_0,x_0)\in Q^{l}}\sum_{j=0}^k 2^{(k-j)\sigma}\inf_{p\in \mathcal{P}_1}[u-p]_{\alpha/\sigma,\alpha;(t_0-2^{-k\sigma},t_0)\times B_{2^{j-k}(x_0)}} \nonumber\\
	&\quad+ C2^{-(k+k_0)\alpha}[u]_{\alpha/\sigma,\alpha}+C 2^{-k\alpha+k_0(\sigma-\alpha)}[Du]_{\alpha/\sigma,\alpha;Q^{l+1}}\nonumber\\	
	&\quad + C2^{k_0(\sigma-\alpha)}\Big[\omega_f(2^{-k})+
\big(\omega_b(2^{-k})+\omega_a(2^{-k})\big)\|Du\|_{L_\infty(Q^{l+1})}\nonumber\\
	&\quad+\omega_a(2^{-k})\big(\sum_{j=0}^\infty 2^{j(\sigma-\alpha)}\sup_{(t_0,x_0)\in Q^{l+1}}\inf_{p\in \mathcal{P}_1}[u-p]_{\alpha/\sigma,\alpha;Q_{2^{-j}}(t_0,x_0)}+\|u\|_{L_\infty}\big)\label{eq 3.88}
\Big].
\end{align}
Summing the above inequality in $k = l+1,l+2,\ldots$ as before, we obtain
\begin{align}
	\nonumber
	&\sum_{k=l+1}^\infty 2^{(k+k_0)(\sigma-\alpha)}\sup_{(t_0,x_0)\in Q^{l}}\inf_{p\in \mathcal{P}_1}[u-p]_{\alpha/\sigma,\alpha;Q_{2^{-k-k_0}}(t_0,x_0)}\\
	\nonumber
	&\le C2^{-k_0\alpha}\sum_{j=0}^\infty 2^{j(\sigma-\alpha)}\sup_{(t_0,x_0)\in Q^{l+1}}\inf_{p\in \mathcal{P}_1}[u-p]_{\alpha/\sigma,\alpha;Q_{2^{-j}}(t_0,x_0)}\\
	\nonumber
	&\,\,+ C2^{-(k_0+l)\alpha}[u]_{\alpha/\sigma,\alpha}+C2^{k_0(\sigma-\alpha)}\sum_{k=l+1}^\infty2^{-k\alpha}[Du]_{\alpha/\sigma,\alpha;Q^{l+1}}\\
	\nonumber &\,\,+C2^{k_0(\sigma-\alpha)}\sum_{k=l+1}^\infty
\Big(\omega_f(2^{-k})+\big(\omega_b(2^{-k})
+\omega_a(2^{-k})\big)\|Du\|_{L_\infty(Q^{l+1})}\Big)\\
	\label{eq 3.13}
	&\,\,+C2^{k_0(\sigma-\alpha)}\sum_{k=l+1}^\infty\omega_a(2^{-k}) \big(\sum_{j=0}^\infty 2^{j(\sigma-\alpha)}\sup_{(t_0,x_0)\in Q^{l+1}}\inf_{p\in \mathcal{P}_1}[u-p]^x_{\alpha;Q_{2^{-j}}(t_0,x_0)}+\|u\|_{L_\infty}\big),
\end{align}
and
\begin{align}
	\nonumber
	&\sum_{j=0}^\infty 2^{j(\sigma-\alpha)}\sup_{(t_0,x_0)\in Q^{l}}\inf_{p\in \mathcal{P}_1}[u-p]_{\alpha/\sigma,\alpha;Q_{2^{-j}}(t_0,x_0)}\\
	\nonumber
	&\le C2^{(k_0+l)(\sigma-\alpha)}[u]_{\alpha/\sigma,\alpha}+ C2^{-k_0\alpha}\sum_{j=0}^\infty 2^{j(\sigma-\alpha)}\sup_{(t_0,x_0)\in Q^{l+1}}\inf_{p\in \mathcal{P}_1}[u-p]_{\alpha/\sigma,\alpha;Q_{2^{-j}}(t_0,x_0)}\\
	\nonumber &\,\,+C2^{k_0(\sigma-\alpha)-l\alpha}[Du]_{\alpha/\sigma,\alpha;Q^{l+1}}\\
	\nonumber &\,\,+C2^{k_0(\sigma-\alpha)}\sum_{k=l+1}^\infty
\Big(\omega_f(2^{-k})+\big(\omega_b(2^{-k})
+\omega_a(2^{-k})\big)\|Du\|_{L_\infty(Q^{l+1})}\Big)\\
	\nonumber
	&\,\,+C2^{k_0(\sigma-\alpha)}\sum_{k=l+1}^\infty\omega_a(2^{-k}) \big(\sum_{j=0}^\infty 2^{j(\sigma-\alpha)}\sup_{(t_0,x_0)\in Q^{l+1}}\inf_{p\in \mathcal{P}_1}[u-p]^x_{\alpha;Q_{2^{-j}}(t_0,x_0)}+\|u\|_{L_\infty}\big).
\end{align}
By choosing $k_0$ and $l$ sufficiently large, and using \eqref{eq 2.3a} and interpolation inequalities (recalling that $\alpha<(\sigma-1)/2$),  we obtain
\begin{align*}
&\sum_{j=0}^\infty 2^{j(\sigma-\alpha)}\sup_{(t_0,x_0)\in Q^l}\inf_{p\in \mathcal{P}_1}[u-p]_{\alpha/\sigma,\alpha;Q_{2^{-j}}(t_0,x_0)}\\
&\le \frac 14 \sum_{j=0}^\infty 2^{j(\sigma-\alpha)}\sup_{(t_0,x_0)\in Q^{l+1}}\inf_{p\in \mathcal{P}_1}[u-p]_{\alpha/\sigma,\alpha;Q_{2^{-j}}(t_0,x_0)}+C2^{(k_0+l)(\sigma-\alpha)}\|u\|_{\alpha/\sigma,\alpha}+C\sum_{k=1}^\infty \omega_f(2^{-k}).
\end{align*}
Therefore,
\begin{align}\label{eq 3.14}
\frac{1}{4^l}\sum_{j=0}^\infty 2^{j(\sigma-\alpha)}\sup_{(t_0,x_0)\in Q^l}\inf_{p\in \mathcal{P}_1}[u-p]_{\alpha/\sigma,\alpha;Q_{2^{-j}}(t_0,x_0)}\le C\|u\|_{\alpha/\sigma,\alpha}+C\sum_{k=1}^\infty \omega_f(2^{-k}),
\end{align}
which together with Lemma \ref{lem2.1} $(ii)$ gives \eqref{eq 8.18} and the continuity of $\partial_t u$.
\end{proof}

\subsection{The case when $\sigma = 1$}
\begin{proposition}
	\label{prop3.23}
Suppose that \eqref{eq1.1} is satisfied in $Q_{2}$. Then under the conditions of Theorem \ref{thm 1},
\begin{equation}
	\label{eq 8.17}
\|Du\|_{L_\infty(Q_{1/2})}+\|\partial_t u\|_{L_\infty(Q_{1/2})} \le C\|u\|_{\alpha,\alpha}+C\sum_{k=1}^\infty\omega_f(2^{-k}),
\end{equation}
where $C>0$ is a constant depending only on $d$, $\lambda$, $\Lambda$, $N_0$, $\omega_a$, and $\omega_b$. 
\end{proposition}
\begin{proof}
Set $b_0 = b(0,0)$ and we define
\begin{align*}
\hat{u}(t,x) = u(t,x-b_0t),\quad\quad \hat{f_\beta}(t,x) = f_\beta(t,x-b_0t),\quad \text{and}\quad \hat{b}(t,x) = b(t,x-b_0t).
\end{align*}
It is easy to see that in  $Q_\delta$ for some $\delta>0$,
\begin{align*}
\partial_t\hat{u}(t,x) = \partial_tu(t,x-b_0t)-b_0\nabla u(t,x-b_0t),
\end{align*}
and for $(t,x)\in Q_{2^{-k}}$,
\begin{align*}
|\hat{f}_\beta(t,x)-\hat{f}_\beta(0,0)|\le \omega_f((1+N_0)2^{-k}),\\
|\hat b-b_0|\le \omega_b((1+N_0)2^{-k}).
\end{align*}
It follows immediately that
\begin{equation}
	\label{eq 2.221}
\hat{u}_t = \inf_\beta(\hat L_\beta \hat{u} + \hat{f}_\beta+ (\hat{b}-b_0)\nabla \hat{u}),
\end{equation}
where $\hat L$ is the operator with kernel $a(t,x-b_0t,y)|y|^{-d-\sigma}$. Furthermore,
\begin{align*}
\|Du\|_{L_\infty} + \|\partial_t u\|_{L_\infty} \le (1+N_0)\big(\|D\hat{u}\|_{L_\infty} + \|\partial_t\hat{u}\|_{L_\infty}\big).
\end{align*}
Therefore, it is sufficient to bound $\hat{u}$. In the rest of the proof,  we estimate
the solution to \eqref{eq 2.221} and abuse the notation to use $u$ instead of $\hat{u}$ for simplicity. By scaling, translation and covering arguments, we also assume $u$ satisfies the equation in $Q_2$.

The proof is similar to the case $\sigma  \in(1,2)$ and we indeed proceed as in the previous case.  Take $p_0$ to be the first-order Taylor's expansion of $u^{(2^{-k})}$ at the origin. 
We also assume that the solution $v$ to the following equation
\begin{equation*}
\begin{cases}
	\partial_tv = \inf_{\beta\in \mathcal{A}}(L_\beta(0,0)v+f_\beta(0,0)-\partial_t p_0)\quad \,\,&\text{in}\,\,Q_{2^{-k}}\\
	v = u-p_0\quad \,\,&\text{in}\,\,\big((-2^{-k\sigma},0)\times B_{2^{-k}}^c\big)\cup \big(\{t=-2^{-k\sigma}\}\times B_{2^{-k}}\big),
\end{cases}
\end{equation*}
 exists without carrying out another approximation argument.
By Proposition \ref{prop3.1} and Lemma \ref{lem2.7} with $\sigma= 1$,
\begin{align}
	\nonumber
	&[v]_{1+\alpha,1+\alpha;Q_{2^{-k-1}}}\le C\sum_{j=1}^\infty 2^{k-j}M_j+C2^k[v]_{\alpha,\alpha;Q_{2^{-k}}}\\
	\nonumber
	&\le C\sum_{j=1}^\infty 2^{k-j}\inf_{p\in \mathcal{P}_1}[u-p]_{\alpha,\alpha;(-2^{-k},0)\times B_{2^{j-k}}} + C2^k[v]_{\alpha,\alpha;Q_{2^{-k}}}\\
        \label{eq 3.16}
	&\le C\sum_{j=1}^k 2^{k-j}\inf_{p\in \mathcal{P}_1}[u-p]_{\alpha,\alpha;(-2^{-k},0)\times B_{2^{j-k}}}+ C[u]_{\alpha,\alpha}+C2^{k}[v]_{\alpha,\alpha;Q_{2^{-k}}}.
\end{align}
From \eqref{eq 3.16} and the interpolation inequality, we obtain
\begin{align}
	\nonumber
	&[v-p_1]_{\alpha,\alpha;Q_{2^{-k-k_0}}}\\
	\nonumber
	&\le C2^{-(k+k_0)}\sum_{j=1}^k 2^{k-j}\inf_{p\in \mathcal{P}_1}[u-p]_{\alpha,\alpha;(-2^{-k},0)\times B_{2^{j-k}}} \\
	\label{eq 3.16a}
	&\quad+C2^{-k_0}[v]_{\alpha,\alpha;Q_{2^{-k}}}+ C2^{-(k+k_0)}[u]_{\alpha,\alpha},
\end{align}
where $p_1$ is the first-order Taylor's expansion of $v$ at the origin. Next $w :=u-p_0-v$ satisfies \eqref{eq 3.4}, where by the cancellation property,
\begin{align*}
C_k \le &\omega_f((1+N_0)2^{-k})+\omega_{b}((1+N_0)2^{-k})\|Du\|_{L_\infty(Q_{2^{-k}})} + C\omega_a((1+N_0)2^{-k}) \\
&\cdot\Big(\sup_{(t_0,x_0)\in Q_{2^{-k}}}\sum_{j=0}^\infty 2^{j(1-\alpha)}\inf_{p\in \mathcal{P}_x}[u-p]^x_{\alpha; Q_{2^{-j}}(t_0,x_0)}+\|u\|_{L_\infty}\Big).
\end{align*}
Clearly, for any $r\ge 0$,
$$
\omega_\bullet ((1+N_0)r)\le (2+N_0)\omega_\bullet (r).
$$
Therefore, similar to \eqref{eq 3.5},  we have
\begin{align}
	\nonumber
	&[w]_{\alpha,\alpha;Q_{2^{-k}}}\le C2^{-k(1-\alpha)}\big( \omega_f(2^{-k})+\omega_b(2^{-k})\|Du\|_{L_\infty(Q_{2^{-k}})}\\
	\label{eq 3.16b}
	&+\omega_a(2^{-k})\Big(\sum_{j=0}^\infty 2^{j(1-\alpha)}\sup_{(t_0,x_0)\in Q_{2^{-k}}}\inf_{p\in \mathcal{P}_x}[u-p]^x_{\alpha; Q_{2^{-j}}(t_0,x_0)}+\|u\|_{L_\infty}\Big).
\end{align}
From \eqref{eq 2.19a} and the triangle inequality,
\begin{align*}
	[v]_{\alpha,\alpha;Q_{2^{-k}}}\le [w]_{\alpha,\alpha;Q_{2^{-k}}}+[u-p_0]_{\alpha,\alpha;Q_{2^{-k}}} \le [w]_{\alpha,\alpha;Q_{2^{-k}}}+ C\inf_{p\in \mathcal{P}_1}[u-p]_{\alpha,\alpha;Q_{2^{-k}}}.
\end{align*}
For all $l= 1,2,\cdots$, we define $Q^l=Q_{1-2^{-l}}$. Similar to \eqref{eq 3.8}, by combining \eqref{eq 3.16a} and \eqref{eq 3.16b}, shifting the coordinates, and using the above inequality, we obtain for all $l\ge 1$ and $k\ge l+1$,
\begin{equation}\label{eq:startingpoint}
\begin{split}
	&2^{(k+k_0)(1-\alpha)}\sup_{(t_0,x_0)\in Q^l}\inf_{p\in \mathcal{P}_1}[u-p]_{\alpha,\alpha;Q_{2^{-k-k_0}(t_0,x_0)}}\\
	&\le C2^{-(k+k_0)\alpha}\sup_{(t_0,x_0)\in Q^{l}}\sum_{j=0}^k 2^{k-j}\inf_{p\in \mathcal{P}_1}[u-p]_{\alpha,\alpha;(t_0-2^{-k},t_0)\times B_{2^{j-k}}(x_0)}\\
&\quad+C2^{k_0(1-\alpha)}\Big[\omega_f(2^{-k})
+\omega_b(2^{-k})\|Du\|_{L_\infty(Q^{l+1})}\\
	&\quad+\omega_a(2^{-k})\Big(\sum_{j=0}^\infty 2^{j(1-\alpha)}\sup_{(t_0,x_0)\in Q^{l+1}}\inf_{p\in \mathcal{P}_x}[u-p]^x_{\alpha, Q_{2^{-j}}(t_0,x_0)}+\|u\|_{L_\infty}\Big)\Big]\\
	&\quad+ C2^{-(k+k_0)\alpha}[u]_{\alpha,\alpha},
\end{split}
\end{equation}
which by summing in $k = l+1,l+2,\ldots$, implies that
\begin{align*}
	&\sum_{k=l+1}^\infty 2^{(k+k_0)(1-\alpha)}\sup_{(t_0,x_0)\in Q^l}\inf_{p\in \mathcal{P}_1}[u-p]_{\alpha,\alpha;Q_{2^{-k-k_0}}(t_0,x_0)}\\
	&\le C2^{-k_0\alpha}\sum_{j=0}^\infty 2^{j(1-\alpha)}\sup_{(t_0,x_0)\in Q^{l+1}}\inf_{p\in \mathcal{P}_1}[u-p]_{\alpha,\alpha;Q_{2^{-j}}(t_0,x_0)}\\
	&\quad+C2^{-(k_0+l)\alpha}[u]_{\alpha,\alpha}+C2^{k_0(1-\alpha)}
\sum_{k=l+1}^\infty\omega_f(2^{-k}) \\
	&\quad+ C2^{k_0(1-\alpha)}\sum_{k=l+1}^\infty
\Big[\omega_b(2^{-k})\|Du\|_{L_\infty(Q^{l+1})}+\omega_a(2^{-k})\\
	&\quad\cdot\Big(\sum_{j=0}^\infty 2^{j(1-\alpha)}\sup_{(t_0,x_0)\in Q^{l+1}}\inf_{p\in \mathcal{P}_x}[u-p]^x_{\alpha; Q_{2^{-j}}(t_0,x_0)}+\|u\|_{L_\infty}\Big)\Big],
\end{align*}
where for the first term on the right-hand side, we replaced $j$ by $k-j$, switched the order of the summation, and bounded it by
\begin{align*}
	&\sum_{k=0}^\infty 2^{-(k+k_0)\alpha}\sum_{j=0}^k 2^{j}\sup_{(t_0,x_0)\in Q^{l+1}}\inf_{p\in \mathcal{P}_1}[u-p]_{\alpha,\alpha;Q_{2^{-j}}(t_0,x_0)}\\
	&= 2^{-k_0\alpha}\sum_{j=0}^\infty 2^{j}\sup_{(t_0,x_0)\in Q^{l+1}}\inf_{p\in \mathcal{P}_1}[u-p]_{\alpha,\alpha;Q_{2^{-j}}(t_0,x_0)}\sum_{k=j}^\infty 2^{-k\alpha}\\
	&\le C2^{-k_0\alpha}\sum_{j=0}^\infty 2^{j(1-\alpha)}\sup_{(t_0,x_0)\in Q^{l+1}}\inf_{p\in \mathcal{P}_1}[u-p]_{\alpha,\alpha;Q_{2^{-j}}(t_0,x_0)}.
\end{align*}
Therefore,
\begin{equation}\label{eq:auxequal1}
\begin{split}
	&\sum_{j=0}^\infty 2^{j(1-\alpha)}\sup_{(t_0,x_0)\in Q^{l}}\inf_{p\in \mathcal{P}_1}[u-p]_{\alpha,\alpha;Q_{2^{-j}}(t_0,x_0)}\\
	&\le C2^{-k_0\alpha}\sum_{j=0}^\infty 2^{j(1-\alpha)}\sup_{(t_0,x_0)\in Q^{l+1}}\inf_{p\in \mathcal{P}_1}[u-p]_{\alpha,\alpha;Q_{2^{-j}}(t_0,x_0)}\\
	&\quad+C2^{(l+k_0)(1-\alpha)}[u]_{\alpha,\alpha}
+C2^{k_0(1-\alpha)}\sum_{k=l+1}^\infty \omega_f(2^{-k})\\
	&\quad +C2^{k_0(1-\alpha)}\sum_{k=l+1}^\infty\omega_b(2^{-k})\|Du\|_{L_\infty(Q^{l+1})}
+C2^{k_0(1-\alpha)}\sum_{k=l+1}^\infty \omega_a(2^{-k})\\
	&\quad\cdot\Big(\sum_{j=0}^\infty 2^{j(1-\alpha)}\sup_{(t_0,x_0)\in Q^{l+1}}\inf_{p\in \mathcal{P}_1}[u-p]_{\alpha,\alpha;Q_{2^{-j}}(t_0,x_0)}+\|u\|_{L_\infty}\Big).
\end{split}
\end{equation}
Then we choose $k_0$ and $l$ sufficiently large, and apply Lemma \ref{lem2.1} $(iii)$ to obtain
\begin{align*}
	&\sum_{j=0}^\infty 2^{j(1-\alpha)}\sup_{(t_0,x_0)\in Q^{l}}\inf_{p\in \mathcal{P}_1}[u-p]_{\alpha,\alpha;Q_{2^{-j}}(t_0,x_0)}\\
	&\le 	\frac 14 \sum_{j=0}^\infty 2^{j(1-\alpha)}\sup_{(t_0,x_0)\in Q^{l+1}}\inf_{p\in \mathcal{P}_1}[u-p]_{\alpha,\alpha;Q_{2^{-j}}(t_0,x_0)}+C2^{(l+k_0)(1-\alpha)}\|u\|_{\alpha/\sigma,\alpha}+C\sum_{k=1}^\infty \omega_f(2^{-k}),
\end{align*}
and thus,
\begin{align}\label{eq:controlnormlarge}
\frac{1}{4^l}\sum_{j=0}^\infty 2^{j(1-\alpha)}\sup_{(t_0,x_0)\in Q^{l}}\inf_{p\in \mathcal{P}_1}[u-p]_{\alpha,\alpha;Q_{2^{-j}}(t_0,x_0)}\le C\|u\|_{\alpha/\sigma,\alpha}+C\sum_{k=1}^\infty \omega_f(2^{-k}),
\end{align}
from which \eqref{eq 8.17} follows.
The proposition is proved.
\end{proof}

\subsection{Proof of Theorem \ref{thm 1}}
We use the localization argument to prove Theorem \ref{thm 1}.
\begin{proof}[Proof of Theorem \ref{thm 1}]
Without loss of generality, we assume the equation holds in $Q_3$. We divide the proof into three steps.

{\em Step 1.} For $k=1,2,\ldots$, denote $Q^k := Q_{1-2^{-k}}$.
Let $\eta_k\in C_0^\infty(\widehat Q^{k+3})$ be a sequence of nonnegative smooth cutoff functions satisfying
$\eta_k\equiv 1$ in $Q^{k+2}$, $|\eta_k|\le 1$ in $Q^{k+3}$,  $\|\partial_t^jD^i\eta_k\|_{L_\infty}\le C2^{k(i+j)}$ for each $i,j\ge 0$.  Set $v_k := u\eta_k\in C^{1,\sigma+}$ and notice that in $Q^{k+1}$,
\begin{align*}
\partial_tv_k&=\eta_k \partial_tu+\partial_t\eta_k u
=\inf_{\beta\in \mathcal{A}}(\eta_k L_\beta u+\eta_kb_\beta Du+\eta_k f_\beta+\partial_t\eta_k u)\\
&=\inf_{\beta\in \mathcal{A}}(L_\beta v_k+b_\beta Dv_k-b_\beta uD\eta_k+h_{k\beta}+\eta_k f_{\beta}+\partial_t\eta_k u),
\end{align*}
where
\begin{align*}
h_{k\beta}=\eta_k L_\beta u-L_\beta v_k=\int_{\bR^d}\frac{
\xi_k(t,x,y)a_\beta(t,x,y)}{|y|^{d+\sigma}}\,dy,
\end{align*}
and
\begin{equation*}
\begin{split}
\xi_k(t,x,y) &= u(t,x+y)(\eta_k(t,x+y)-\eta_k(t,x))-y\cdot D\eta_k(t,x)u(t,x)(\chi_{\sigma=1}\chi_{B_1}+\chi_{\sigma>1})\\
& =u(t,x+y)(\eta_k(t,x+y)-\eta_k(t,x))\quad\mbox{since }D\eta_k\equiv 0\mbox{ in } Q^{k+1}.
\end{split}
\end{equation*}

We will apply Proposition \ref{prop 3.22} to the equation of $v_k$ in $Q^{k+1}$ and obtain corresponding estimates for $v_k$ in $Q^k$.

Obviously, in $Q^{k+1}$ we have $\eta_k f_\beta\equiv f_\beta, b_\beta uD\eta_k\equiv 0$, and $\partial_t\eta_k u\equiv 0$. Thus, we only need to estimate the modulus of continuity of $h_{k\beta}$ in $Q^{k+1}$.

{\em Step 2.} For $(t,x)\in Q^{k+1}$ and $|y|\le 2^{-k-3}$, we have
\[
\xi_k(t,x,y) = 0.
\]
Also,
\begin{equation*}
\begin{split}
|\xi_k(t,x,y)| &= |u(t,x+y)(\eta_k(t,x+y)-\eta_k(t,x))|\\
&\le
\begin{cases}
2\omega_u(|y|)+2|u(t,x)|\quad\mbox{when }|y|\ge 1,\\
C2^k|u(t,x+y)| |y|\quad\mbox{when }2^{-k-3}<|y|< 1.\\
\end{cases}
\end{split}
\end{equation*}

For $(t,x),(t',x')\in Q^{k+1}$, by the triangle inequality,
\begin{align}
                                    \label{eq11.39}
&|h_{k\beta}(t,x)-h_{k\beta}(t',x')|\nonumber \\
&\le \int_{\bR^d}\frac{|(\xi_k(t,x,y)-\xi_k(t',x',y))
a_\beta(t,x,y)|}{|y|^{d+\sigma}}\nonumber\\
&\quad+\frac{|\xi_k(t',x',y)(a_\beta(t,x,y)-a_\beta(t',x',y))|}
{|y|^{d+\sigma}}\,dy
:= {\rm I}+{\rm II}.
\end{align}
By the estimates of $|\xi_k(t,x,y)|$ above, we have
\begin{equation}
                                \label{eq11.40}
{\rm II}\le C\Big(2^{k(\sigma+1)}\|u\|_{L_\infty(Q_2)}+\sum_{j=0}^\infty2^{-j\sigma}\omega_u(2^j)\Big)\omega_a(\max\{|x-x'|,|t-t'|^{1/\sigma}\}),
\end{equation}
where $C$ depends on $d$, $\sigma$, and  $\Lambda$. For ${\rm I}$, by the fundamental theorem of calculus,
\begin{align*}
\xi_k(t,x,y)-\xi_k(t',x',y) = y\cdot\int_0^1\Big(u(t,x+y)D\eta_k(t,x+sy) -u(t',x'+y)D\eta_k(t',x'+sy)\Big)\,ds.
\end{align*}
When $2^{-k-3}\le |y|< 2$, similar to the estimate of $\xi_k(t,x,y)$, it follows that
\begin{align}
             \nonumber
|\xi_k(t,x,y)-\xi_k(t',x',y)|&\le C|y|\big(2^{k}\omega_u(\max\{|x-x'|,|t-t'|^{1/\sigma}\})\\
	\label{eq11.41}
&\quad +2^{2k}\|u\|_{L_\infty(Q_3)}(|x-x'|+|t-t'|)\big).
\end{align}
When $|y|\ge 2$, we have
\[
|\xi_k(t,x,y)-\xi_k(t',x',y)|=|u(t,x+y)-u(t',x'+y)|\le \omega_u(\max\{|x-x'|,|t-t'|^{1/\sigma}\}),
\]
which implies
\[
{\rm I}\le C2^{k(\sigma+1)}\omega_u(\max\{|x-x'|,|t-t'|^{1/\sigma}\})
+2^{k(\sigma+2)}\|u\|_{L_\infty(Q_3)}(|x-x'|+|t-t'|)\big).
\]
Therefore,
\begin{align*}
|h_{k\beta}(t,x)-h_{k\beta}(t',x')|&\le \omega_h(\max\{|x-x'|,|t-t'|^{1/\sigma}\}),
\end{align*}
where
\begin{align}
                        \label{eq9.32}
\omega_h(r) &:= C\Big(2^{k(\sigma+1)}\|u\|_{L_\infty(Q_3)}+\sum_{j=0}^\infty2^{-j\sigma}\omega_u(2^j)\Big)\omega_a(r)\nonumber\\
&\quad +C2^{k(\sigma+1)}\omega_u(r)+C2^{k(\sigma+2)}\|u\|_{L_\infty(Q_3)}(r+r^\sigma)
\end{align}
is a Dini function.

{\em Step 3.}  In this last step, we only present the detailed proof for $\sigma\in (1,2)$. We omit the details for the proof of the case $\sigma\in(0,1]$, since it is almost the same as and actually even simpler than the case $\sigma\in(1,2)$.  We apply Proposition \ref{prop 3.22}, together with a scaling and covering argument, to $v_k$ to obtain
\begin{align*}
	&\|\partial_t v_k\|_{L_\infty(Q^k)}+[v_k]^x_{\sigma;Q^k}+[Dv_k]^t_{\frac{\sigma-1}{\sigma};Q^k}\\
	&\le C2^{k\sigma}\|v_k\|_{L_\infty}+C2^{k(\sigma-\alpha)}[v_k]_{\alpha/\sigma,\alpha}+C\sum_{j=1}^\infty \big(\omega_h(2^{-j})+\omega_f(2^{-j})\big)\\
	&\le  C2^{k(\sigma+2)}\|u\|_{L_\infty(Q_3)}+C_02^{k(\sigma-\alpha)}[u]_{\alpha/\sigma,\alpha;Q^{k+3}}+C\sum_{j=0}^\infty2^{-j\sigma}\omega_u(2^j)\\
	&\quad +C\sum_{j=0}^\infty\big(2^{k(\sigma+1)}\omega_u(2^{-j})
+\omega_f(2^{-j})\big),
\end{align*}
where $C$ and $C_0$ depend on $d$, $\lambda$, $\Lambda$, $\sigma$, $N_0$, $\omega_b$, and $\omega_a$, but independent of $k$. Since $\eta_k\equiv 1$ in $Q^k$,  it follows that
\begin{align}
		\nonumber
&\|\partial_t u\|_{L_\infty(Q^k)}+[u]^x_{\sigma;Q^{k}}+[Du]^t_{\frac{\sigma-1}{\sigma};Q^k}\\
&\le C2^{k(\sigma+2)}\|u\|_{L_\infty(Q_3)}+C_02^{k(\sigma-\alpha)}[u]_{\alpha/\sigma,\alpha;Q^{k+3}}+C\sum_{j=0}^\infty2^{-j\sigma}\omega_u(2^j)\nonumber \\
	&\quad +C\sum_{j=0}^\infty\big(2^{k(\sigma+1)}\omega_u(2^{-j})
+\omega_f(2^{-j})\big).\label{eq12.141}
\end{align}
By the interpolation inequality, for any $\epsilon\in(0,1)$,
\begin{align}
	\label{eq12.142}
[u]_{\alpha/\sigma, \alpha;Q^{k+3}} \le \epsilon(\|\partial_t u\|_{L_\infty(Q^{k+3})}+[u]^x_{\sigma;Q^{k+3}}) +C\epsilon^{-\frac{\alpha}{\sigma-\alpha}}\|u\|_{L_\infty(Q_3)}.
\end{align}
Recall that $\alpha \le \frac{\sigma-1}{2}$ and thus,
$$
\frac{\alpha}{\sigma-\alpha} \le \frac{\sigma-1}{\sigma+1}<1/2.
$$
Combining \eqref{eq12.141} and \eqref{eq12.142} with $\epsilon = C_0^{-1}2^{-3k-16}$, we obtain
\begin{align*}
&\|\partial_t u\|_{L_\infty(Q^k)}+[u]^x_{\sigma;Q^k}+[Du]^t_{\frac{\sigma-1}{\sigma};Q^k}\\
&\quad \le 2^{-16}([u]^x_{\sigma;Q^{k+3}}+\|\partial_t u\|_{L_\infty(Q^{k+3})}+[Du]^t_{\frac{\sigma-1}{\sigma};Q^{k+3}})\\
&\quad + C2^{4k}\|u\|_{L_\infty(Q_3)}+C\sum_{j=0}^\infty2^{-j\sigma}\omega_u(2^j) +C\sum_{j=0}^\infty\big(2^{k(\sigma+1)}\omega_u(2^{-j})
+\omega_f(2^{-j})\big).
\end{align*}
Then we multiply $2^{-5k}$ to both sides of the above inequality and get
\begin{align*}
	&2^{-5k}(\|\partial_t u\|_{L_\infty(Q^k)}+[u]^x_{\sigma;Q^k}+[Du]^t_{\frac{\sigma-1}{\sigma};Q^k})\\
	&\le 2^{-5(k+3)-1}(\|\partial_t u\|_{L_\infty(Q^{k+3})}+[u]^x_{\sigma;Q^{k+3}}+[Du]^t_{\frac{\sigma-1}{\sigma};Q^{k+3}})\\
& +C2^{-k}\|u\|_{L_\infty(Q_3)}+C2^{-2k}\sum_{j=0}^\infty\big(2^{-j\sigma}\omega_u(2^j) +\omega_u(2^{-j})+\omega_f(2^{-j})\big).
\end{align*}
We sum up the both sides of the above inequality  and obtain
\begin{align*}
	&\sum_{k=1}^\infty2^{-5k}(\|\partial_t u\|_{L_\infty(Q^k)}+[u]^x_{\sigma;Q^k}+[Du]^t_{\frac{\sigma-1}{\sigma};Q^k})\\
&\quad \le \frac{1}{2}\sum_{k=4}^\infty2^{-5k}(\|\partial_t u\|_{L_\infty(Q^{k})}+[u]^x_{\sigma;Q^{k}}
+[Du]^t_{\frac{\sigma-1}{\sigma};Q^{k}})\\
	&\quad +C\|u\|_{L_\infty(Q_3)}+C\sum_{j=0}^\infty\big(2^{-j\sigma}\omega_u(2^j) +\omega_u(2^{-j})+\omega_f(2^{-j})\big),
\end{align*}
which further implies that
\begin{align*}
&\sum_{k=1}^\infty2^{-5k}(\|\partial_t u\|_{L_\infty(Q^k)}+[u]^x_{\sigma;Q^k}+[Du]^t_{\frac{\sigma-1}{\sigma};Q^k})\\
&\quad \le C\|u\|_{L_\infty(Q_3)}+C\sum_{j=0}^\infty\big(2^{-j\sigma}\omega_u(2^j) +\omega_u(2^{-j})+\omega_f(2^{-j})\big),
\end{align*}
where $C$ depends on $d$, $\lambda$, $\Lambda$, $\sigma$, $\omega_b$, $N_0$, and $\omega_a$.
By applying this estimate to $u-u(0,0)$, we obtain
\begin{align}	\label{eq12.09}
\|\partial_t u\|_{L_\infty(Q^4)}+[u]^x_{\sigma;Q^4}+[Du]^t_{\frac{\sigma-1}{\sigma};Q^4}\le C\sum_{j=0}^\infty\big(2^{-j\sigma}\omega_u(2^j) +\omega_u(2^{-j})+\omega_f(2^{-j})\big).
\end{align}
This proves \eqref{eq12.17}.

Finally, since $\|v_1\|_{\alpha/\sigma,\alpha}$ is bounded by the right-hand side of \eqref{eq12.09}, from \eqref{eq 3.14}, we see that
\begin{align*}
\sum_{j=0}^\infty 2^{j(\sigma-\alpha)}\sup_{(t_0,x_0)\in Q^l}\inf_{p\in \cP_1}[v_1-p]_{\alpha/\sigma,\alpha;Q_{2^{-j}}(t_0,x_0)}\le C
\end{align*}
for some large $l$.
This and \eqref{eq 3.13} with $u$ replaced by $v_1$ and $f_\beta$ replaced by $h_{1\beta}+\eta_1 f_{\beta}+\partial_t\eta_1 u-b_\beta u D\eta_1$ give
\begin{align*}
&\sum_{j=k_1+1}^\infty 2^{(j+k_0)(\sigma-\alpha)}\sup_{(t_0,x_0)\in Q^{k_1}}\inf_{p\in \cP_1}[v_{1}-p]_{\alpha/\sigma,\alpha;Q_{2^{-j-k_0}}(t_0,x_0)}\\
&\le
C2^{-k_0\alpha}+C2^{k_0(\sigma-\alpha)}
\sum_{j=k_1}^\infty \big(\omega_f(2^{-j})+\omega_a(2^{-j})+\omega_u(2^{-j})
+\omega_b(2^{-j})+2^{-j\alpha}\big).
\end{align*}
Here we also used \eqref{eq9.32} with $k=1$.
Therefore, for any small $\varepsilon>0$, we can find $k_0$ sufficiently large then $k_1$ sufficiently large, depending only on $C$, $\sigma$, $N_0$, $\alpha$, $\omega_f$, $\omega_a$, $\omega_f$, $\omega_b$, and $\omega_u$, such that
$$
\sum_{j=k_1+1}^\infty 2^{(j+k_0)(\sigma-\alpha)}\sup_{(t_0,x_0)\in Q^{k_1}}\inf_{p\in \cP_1}[v_1-p]_{\alpha/\sigma,\alpha;Q_{2^{-j-k_0}}(t_0,x_0)}<\varepsilon,
$$
which, together with the fact that $v_1 = u$ in $Q_{1/2}$ and the proof of Lemma \ref{lem2.1} (ii), indicates that
$$
\sup_{(t_0,x_0)\in Q_{1/2}} \Big([u]^x_{\sigma;Q_r(t_0,x_0)}
+[Du]^t_{\frac{\sigma-1}{\sigma};Q_{r}(t_0,x_0)}\Big)\to 0 \quad\text{as}\quad r\to 0
$$
with a decay rate depending  only on $d$, $\lambda$, $N_0$, $\Lambda$, $\omega_a$, $\omega_f$, $\omega_b$, $\omega_u$, and $\sigma$. Hence, the proof of the case when $\sigma\in (1,2)$ is completed.
\end{proof}

\section{Schauder estimates for equations with drifts}\label{sec:schauder}
In this section, we are going to prove Theorems \ref{thm:schauder} and  \ref{thm:linearschauder}. Here, the main difference from the theorems in \cite{DZ16} is that our equation may have a drift, especially for $\sigma=1$. 

We first prove a weaker version of Theorem \ref{thm:schauder}.

\begin{proposition}
	\label{prop:global1}
Suppose that \eqref{eq1.1} is satisfied in $Q_2$. Then under the conditions of Theorem \ref{thm:schauder}, for any $\gamma\in (0,\min\{\hat\alpha,2-\sigma\})$ with $\hat\alpha$ being the one in Proposition \ref{prop3.1}, and any $\alpha\in (\gamma,\min\{\hat\alpha,2-\sigma\})$, we have
\begin{equation*}
[u]_{1+\gamma/\sigma,\sigma+\gamma;Q_{1/2}} \le C(\|u\|_{\alpha/\sigma,\alpha}+C_f),
\end{equation*}
where $C>0$ is a constant depending only on $d,\gamma, \alpha$, $\sigma, \lambda$, $\Lambda$, $N_0$, and $C_b$. 
\end{proposition}
\begin{proof}
The proof is very similar to that of Propositions  \ref{prop3.21}, \ref{prop 3.22}, and \ref{prop3.23}. We fix an $\alpha\in (\gamma,\hat\alpha)$.

\textbf{Case 1:} $\sigma\in(0,1)$. We start from \eqref{eq 3.8}. Let $Q^l$ and $\ell_0$ be as in the proof of Proposition \ref{prop3.21}. Multiplying $2^{(k+k_0)\gamma}$ to both sides of \eqref{eq 3.8} and making use of the H\"older continuity of $a$ and $f$, we have for all $l\ge \ell_0$ and $k\ge l+1$,
\begin{align}
	\nonumber
&2^{(k+k_0)(\sigma+\gamma-\alpha)}\sup_{(t_0,x_0)\in Q^l}\inf_{p\in\mathcal{P}_t}[u-p]_{\alpha/\sigma,\alpha;Q_{2^{-k-k_0}}(t_0,x_0)}\\
	\nonumber
&\le  C2^{(k+k_0)(\gamma-\alpha)}\sup_{(t_0,x_0)\in Q^l}\sum_{j=0}^k2^{(k-j)\sigma}\inf_{p\in \mathcal{P}_t}[u-p]_{\alpha/\sigma,\alpha;(t_0-2^{-k\sigma},t_0)\times B_{2^{j-k}}(x_0)}\\\nonumber
&\quad+C2^{(k+k_0)(\gamma-\alpha)}[u]_{\alpha/\sigma,\alpha}+C2^{k_0(\sigma+\gamma-\alpha)}\Big[ C_f+\\
&\quad\quad \sup_{(t_0,x_0)\in Q^{l+1}}\sum_{j=0}^\infty 2^{j(\sigma-\alpha)}\inf_{p\in \mathcal{P}_t}[u-p]_{\alpha/\sigma,\alpha;Q_{2^{-j}}(t_0,x_0)}	+\|u\|_{L_\infty}\Big].
	\nonumber
\end{align}
Taking the supremum in $k\ge\ell_0+1$ and using the fact that $\gamma<\alpha$, we have
\begin{align}
	\nonumber
&\sup_{k\ge \ell_0+k_0+1}2^{k(\sigma+\gamma-\alpha)}\sup_{(t_0,x_0)\in Q^l}\inf_{p\in\mathcal{P}_t}[u-p]_{\alpha/\sigma,\alpha;Q_{2^{-k}}(t_0,x_0)}\\
	\nonumber
&\le  C2^{k_0(\gamma-\alpha)}\sup_{k\ge 0}2^{k(\sigma+\gamma-\alpha)}\sup_{(t_0,x_0)\in Q^l}\inf_{p\in \mathcal{P}_t}[u-p]_{\alpha/\sigma,\alpha;Q_{2^{-k}}(x_0)}\\\nonumber
&\quad+C2^{(\ell_0+k_0+1)(\gamma-\alpha)}[u]_{\alpha/\sigma,\alpha}+C2^{k_0(\sigma+\gamma-\alpha)}\Big[ C_f+\\
&\quad\quad \sup_{(t_0,x_0)\in Q^{l+1}}\sum_{j=0}^\infty 2^{j(\sigma-\alpha)}\inf_{p\in \mathcal{P}_t}[u-p]_{\alpha/\sigma,\alpha;Q_{2^{-j}}(t_0,x_0)}	+\|u\|_{L_\infty}\Big].
	\nonumber
\end{align}
By taking $k_0$ large, $l=\ell_0$,  using \eqref{eq:thetrick}, and noticing that
\[
\sup_{0\le k\le \ell_0+k_0} 2^{k(\sigma+\gamma-\alpha)}\sup_{(t_0,x_0)\in Q^{l}}\inf_{p\in \mathcal{P}_t}[u-p]_{\alpha/\sigma,\alpha;Q_{2^{-k}(t_0,x_0)}}\le C 2^{(\ell_0+k_0)(\sigma+\gamma-\alpha)}[u]_{\alpha/\sigma,\alpha},
\]
we have
\[
\sup_{k\ge 0}2^{k(\sigma+\gamma-\alpha)}\sup_{(t_0,x_0)\in Q_{1/2}}\inf_{p\in\mathcal{P}_t}[u-p]_{\alpha/\sigma,\alpha;Q_{2^{-k}}(t_0,x_0)}\le C\Big[C_f+\|u\|_{\alpha/\sigma,\alpha}\Big].
\]
Since
\[
[u]_{1+\gamma/\sigma,\sigma+\gamma;Q_{1/2}} \le C\sup_{k\ge 0}2^{k(\sigma+\gamma-\alpha)}\sup_{(t_0,x_0)\in Q_{1/2}}\inf_{p\in\mathcal{P}_t}[u-p]_{\alpha/\sigma,\alpha;Q_{2^{-k}}(t_0,x_0)}+C[u]_{\alpha/\sigma,\alpha},
\]
we obtain
\[
[u]_{1+\gamma/\sigma,\sigma+\gamma;Q_{1/2}} \le C(\|u\|_{\alpha/\sigma,\alpha}+C_f).
\]

\textbf{Case 2:} $\sigma\in(1,2)$. We start from \eqref{eq 3.88}. Let $Q^l$ be as in the proof of Proposition \ref{prop 3.22}. Multiplying $2^{(k+k_0)\gamma}$ to both sides of \eqref{eq 3.88} and making use of the H\"older continuity of $a$, $b$, and $f$, we have for all $l\ge 1$ and $k\ge l+1$,
\begin{align}
	\nonumber
&2^{(k+k_0)(\sigma+\gamma-\alpha)}\sup_{(t_0,x_0)\in Q^l}\inf_{p\in\mathcal{P}_1}[u-p]_{\alpha/\sigma,\alpha;Q_{2^{-k-k_0}}(t_0,x_0)}\\
	\nonumber
&\le  C2^{(k+k_0)(\gamma-\alpha)}\sup_{(t_0,x_0)\in Q^l}\sum_{j=0}^k2^{(k-j)\sigma}\inf_{p\in \mathcal{P}_1}[u-p]_{\alpha/\sigma,\alpha;(t_0-2^{-k\sigma},t_0)\times B_{2^{j-k}}(x_0)}\\\nonumber
&\quad+C2^{(k+k_0)(\gamma-\alpha)}[u]_{\alpha/\sigma,\alpha}+C2^{(k+k_0)(\gamma-\alpha)+k_0\sigma}[Du]_{\alpha/\sigma,\alpha;Q^{l+1}}+C2^{k_0(\sigma+\gamma-\alpha)}\Big[ C_f+\\
&\quad\quad\|Du\|_{L_\infty(Q^{l+1})}+ \sup_{(t_0,x_0)\in Q^{l+1}}\sum_{j=0}^\infty 2^{j(\sigma-\alpha)}\inf_{p\in \mathcal{P}_1}[u-p]_{\alpha/\sigma,\alpha;Q_{2^{-j}}(t_0,x_0)}	+\|u\|_{L_\infty}\Big].
	\nonumber
\end{align}
Note that this $\hat\alpha$ can be chosen very small, at least strictly smaller than $\sigma-1$.  Taking the supremum in $k\ge 2$  and using the fact that $\gamma<\alpha$, we have
\begin{align}
	\nonumber
&\sup_{k\ge k_0+2}2^{k(\sigma+\gamma-\alpha)}\sup_{(t_0,x_0)\in Q^l}\inf_{p\in\mathcal{P}_1}[u-p]_{\alpha/\sigma,\alpha;Q_{2^{-k}}(t_0,x_0)}\\
	\nonumber
&\le  C2^{k_0(\gamma-\alpha)}\sup_{k\ge 0}2^{k(\sigma+\gamma-\alpha)}\sup_{(t_0,x_0)\in Q^l}\inf_{p\in \mathcal{P}_1}[u-p]_{\alpha/\sigma,\alpha;Q_{2^{-k}}(x_0)}\\\nonumber
&\quad+C2^{(k_0+2)(\gamma-\alpha)}[u]_{\alpha/\sigma,\alpha}+C2^{(2+k_0)(\gamma-\alpha)+k_0\sigma}[Du]_{\alpha/\sigma,\alpha;Q^{l+1}}+C2^{k_0(\sigma+\gamma-\alpha)}\Big[ C_f+\\
&\quad\quad\|Du\|_{L_\infty(Q^{l+1})}+ \sup_{(t_0,x_0)\in Q^{l+1}}\sum_{j=0}^\infty 2^{j(\sigma-\alpha)}\inf_{p\in \mathcal{P}_1}[u-p]_{\alpha/\sigma,\alpha;Q_{2^{-j}}(t_0,x_0)}	+\|u\|_{L_\infty}\Big].
	\nonumber
\end{align}
By taking $k_0$ large, $l=1$, using \eqref{eq 3.14} and \eqref{eq 2.3a}, and noticing that
\[
\sup_{0\le k\le 1+k_0} 2^{k(\sigma+\gamma-\alpha)}\sup_{(t_0,x_0)\in Q^{l}}\inf_{p\in \mathcal{P}_1}[u-p]_{\alpha/\sigma,\alpha;Q_{2^{-k}(t_0,x_0)}}\le C 2^{(1+k_0)(\sigma+\gamma-\alpha)}[u]_{\alpha/\sigma,\alpha},
\]
we have
\[
\sup_{k\ge 0}2^{k(\sigma+\gamma-\alpha)}\sup_{(t_0,x_0)\in Q_{1/2}}\inf_{p\in\mathcal{P}_1}[u-p]_{\alpha/\sigma,\alpha;Q_{2^{-k}}(t_0,x_0)}\le C\Big[C_f+\|u\|_{\alpha/\sigma,\alpha}\Big].
\]
Since
\[
[u]_{1+\gamma/\sigma,\sigma+\gamma;Q_{1/2}} \le C\sup_{k\ge 0}2^{k(\sigma+\gamma-\alpha)}\sup_{(t_0,x_0)\in Q_{1/2}}\inf_{p\in\mathcal{P}_1}[u-p]_{\alpha/\sigma,\alpha;Q_{2^{-k}}(t_0,x_0)}+C[u]_{\alpha/\sigma,\alpha},
\]
we obtain
\[
[u]_{1+\gamma/\sigma,\sigma+\gamma;Q_{1/2}} \le C(\|u\|_{\alpha/\sigma,\alpha}+C_f).
\]

\textbf{Case 3:} $\sigma=1$. We start from \eqref{eq:startingpoint}. Multiplying $2^{(k+k_0)\gamma}$ to both sides of \eqref{eq:startingpoint} and making use of the H\"older continuity of $a,b,f$, we have for all $l\ge 1$ and $k\ge l+1$,
\begin{align*}
	&2^{(k+k_0)(1+\gamma-\alpha)}\sup_{(t_0,x_0)\in Q^l}\inf_{p\in \mathcal{P}_1}[u-p]_{\alpha,\alpha;Q_{2^{-k-k_0}(t_0,x_0)}}\\
	&\le C2^{(k+k_0)(\gamma-\alpha)}\sup_{(t_0,x_0)\in Q^{l}}\sum_{j=0}^k 2^{k-j}\inf_{p\in \mathcal{P}_1}[u-p]_{\alpha,\alpha;(t_0-2^{-k},t_0)\times B_{2^{j-k}}(x_0)}\\
&\quad+C2^{k_0(1+\gamma-\alpha)}\Big[C_f
+\|Du\|_{L_\infty(Q^{l+1})}\\
	&\quad+\sum_{j=0}^\infty 2^{j(1-\alpha)}\sup_{(t_0,x_0)\in Q^{l+1}}\inf_{p\in \mathcal{P}_x}[u-p]^x_{\alpha, Q_{2^{-j}}(t_0,x_0)}+\|u\|_{L_\infty}\Big]\\
	&\quad+ C2^{(k+k_0)(\gamma-\alpha)}[u]_{\alpha,\alpha}.
\end{align*}
Taking the supremum in $k\ge 2$ and using the fact that $\gamma<\alpha$, we have
\begin{align*}
	&\sup_{k\ge k_0+2} 2^{k(1+\gamma-\alpha)}\sup_{(t_0,x_0)\in Q^l}\inf_{p\in \mathcal{P}_1}[u-p]_{\alpha,\alpha;Q_{2^{-k}(t_0,x_0)}}\\
	&\le C2^{(k_0-1)(\gamma-\alpha)}\sup_{k\ge 0} 2^{k(1+\gamma-\alpha)}\sup_{(t_0,x_0)\in Q^{l}}\inf_{p\in \mathcal{P}_1}[u-p]_{\alpha,\alpha;Q_{2^{-k}(t_0,x_0)}}\\
&\quad+C2^{k_0(1+\gamma-\alpha)}\Big[C_f
+\|Du\|_{L_\infty(Q^{l+1})}\\
	&\quad+\sum_{j=0}^\infty 2^{j(1-\alpha)}\sup_{(t_0,x_0)\in Q^{l+1}}\inf_{p\in \mathcal{P}_x}[u-p]^x_{\alpha, Q_{2^{-j}}(t_0,x_0)}+\|u\|_{L_\infty}\Big]\\
	&\quad+ C2^{k_0(\gamma-\alpha)}[u]_{\alpha,\alpha}.
\end{align*}
By taking $k_0$ large, $l=1$, using \eqref{eq:controlnormlarge}, and noticing that
\[
\sup_{0\le k\le k_0+1} 2^{k(1+\gamma-\alpha)}\sup_{(t_0,x_0)\in Q^{l}}\inf_{p\in \mathcal{P}_1}[u-p]_{\alpha,\alpha;Q_{2^{-k}(t_0,x_0)}}\le C 2^{k_0(1+\gamma-\alpha)}[u]_{\alpha,\alpha},
\]
we have
\begin{align*}
\sup_{k\ge 0} 2^{k(1+\gamma-\alpha)}\sup_{(t_0,x_0)\in Q^l}\inf_{p\in \mathcal{P}_1}[u-p]_{\alpha,\alpha;Q_{2^{-k}(t_0,x_0)}}\le C\Big[C_f+\|u\|_{\alpha,\alpha}\Big].
\end{align*}
Since
\[
[Du]_{\gamma,\gamma;Q_{1/2}}+[\partial_t u]_{\gamma,\gamma;Q_{1/2}} \le C\sup_{k\ge 0} 2^{k(1+\gamma-\alpha)}\sup_{(t_0,x_0)\in Q_{1/2}}\inf_{p\in \mathcal{P}_1}[u-p]_{\alpha,\alpha;Q_{2^{-k}(t_0,x_0)}}+C[u]_{\alpha,\alpha},
\]
we obtain
\[
[Du]_{\gamma,\gamma;Q_{1/2}}+[\partial_t u]_{\gamma,\gamma;Q_{1/2}} \le C\Big[C_f
+\|u\|_{\alpha,\alpha}\Big].
\]
The proposition is proved.
\end{proof}

\begin{proof}[Proof of Theorem \ref{thm:schauder}]
The proof is the same as that of Theorem \ref{thm 1} using localizations. We sketch the proof here. We use the same notation as in the proof of Theorem \ref{thm 1}. Without loss of generality, we assume the equation \eqref{eq1.1} holds in $Q_{3}$.

Let $\eta_k\in C_0^\infty(\widehat Q^{k+3})$ be a sequence of nonnegative smooth cutoff functions satisfying
$\eta\equiv 1$ in $Q^{k+2}$, $|\eta|\le 1$ in $Q^{k+3}$,  $\|\partial_t^jD^i\eta_k\|_{L_\infty}\le C2^{k(i+j)}$ for each $i,j\ge 0$.  Set $v_k := u\eta_k\in C^{1+\gamma/\sigma,\sigma+\gamma}$ and notice that in $Q_1$,
\begin{align*}
\partial_tv_k=\inf_{\beta\in \mathcal{A}}(L_\beta v_k+b_\beta Dv_k-b_\beta uD\eta_k+h_{k\beta}+\eta_k f_{\beta}+\partial_t\eta_k u),
\end{align*}
where
\begin{equation*}
h_{k\beta}(t,x) = \int_{\bR^d}\frac{\xi_k(t,x,y) a_\beta(t,x,y)}{|y|^{d+1}}\,dy,
\end{equation*}
and
\begin{align*}
\xi_k(t,x,y) &:= u(t,x+y)(\eta_k(t,x+y)-\eta_k(t,x))-u(t,x)y\cdot D\eta_k(t,x)(\chi_{\sigma=1}\chi_{B_1}+\chi_{\sigma>1})\\
& =u(t,x+y)(\eta_k(t,x+y)-\eta_k(t,x))\quad\mbox{since }D\eta_k\equiv 0\mbox{ in } Q^{k+1}.
\end{align*}

We will apply Proposition \ref{prop:global1} to the equation of $v_k$ in $Q^{k+1}$ and obtain corresponding estimates for $v_k$ in $Q^k$.

Obviously, in $Q^{k+1}$ we have $\eta_k f_\beta\equiv f_\beta, b uD\eta_k\equiv 0$, and $\partial_t\eta_k u\equiv 0$. Thus, we only need to estimate the modulus of continuity of $h_{k\beta}$ in $Q^{k+1}$. Since
$$
\xi_k(t,x,y) := u(t,x+y)(\eta_k(t,x+y)-\eta_k(t,x)),
$$
which is the same as in the Theorem \ref{thm 1}, we also have \eqref{eq9.32} here. Therefore,
\begin{align*}
&[h_{k\beta}]_{\gamma/\sigma,\gamma;Q^{k+1}}\\
&\le C\Big(2^{k(\sigma+1) k}\|u\|_{L_\infty(Q_3)}+\sum_{j=0}^\infty2^{-j\sigma}\omega_u(2^j)\Big)
+C2^{k(\sigma+1)}[u]_{\gamma/\sigma,\gamma}+
C2^{k(\sigma+2)}\|u\|_{L_\infty(Q_3)}.\end{align*}
The rest is almost the same as (actually much simpler than) the proof of Theorem \ref{thm 1}, by using Proposition \ref{prop:global1} (recalling $\gamma<\alpha$), and we omit the details.
\end{proof}

In the following, we prove Theorem \ref{thm:linearschauder} using Theorem \ref{thm:schauder} and difference quotients.
\begin{proof}[Proof of Theorem \ref{thm:linearschauder}]
We only provide the proof for $\sigma+\gamma>2$.
We know from Theorem \ref{thm:schauder} that there exists $\gamma_0$ such that $\sigma+\gamma_0<2$ is not an integer, and the theorem holds for  $0<\gamma\le \gamma_0$. In the below we will prove the theorem for all $\gamma\in(\gamma_0,\sigma)$ using difference quotients.

We suppose the equation \eqref{eq:linear} holds in $Q_4$. We will consider the difference quotients in $x$ first. For $h\in (0,1/4)$, $e\in\mathbb{S}^{d-1}$, let
\[
u^h(t,x)=\frac{u(t, x+he)-u(t,x)}{h^{\gamma-\gamma_0}},
\quad f^h(t,x)=\frac{f(t, x+he)-f(t,x)}{h^{\gamma-\gamma_0}},
\]
and
\[
a^h(t,x,y)=\frac{a(t,x+he,y)-a(t,x,y)}{h^{\gamma-\gamma_0}},\quad b^h(t,x)=\frac{b(t, x+he)-b(t,x)}{h^{\gamma-\gamma_0}}.
\]
Then $u^h$ satisfies
\[
\partial_t u^h(t,x)= L_h u^h+b(t,x+he)Du^h+f^h+b^hDu+g\quad\mbox{in }Q_1,
\]
where
\[
\begin{split}
L_h u=\int_{\bR^d}\frac{\delta u(t,x,y)a(t,x+he,y)}{|y|^{d+\sigma}}\,dy,\quad
g=\int_{\bR^d}\frac{\delta u(t,x,y)a^h(t,x,y)}{|y|^{d+\sigma}}\,dy.
\end{split}
\]
Applying the result for $\gamma=\gamma_0$ gives
\[
\begin{split}
[u^h]_{1+\gamma_0/\sigma,\sigma+\gamma_0;Q_{3/4}}\le C\|u^h\|_{\gamma_0/\sigma,\gamma_0}+C[f^h+b^hDu+g]_{\gamma_0/\sigma,\gamma_0;Q_1}.
\end{split}
\]
It follows from direct calculations that
\begin{align*}
[g]_{\gamma_0/\sigma,\gamma_0;Q_1}\le C[u]_{1+\gamma_0/\sigma,\sigma+\gamma_0;Q_{5/4}}+C\|u\|_{\gamma_0/\sigma,\gamma_0}.
\end{align*}
Applying $C^{1+\gamma_0/\sigma,\sigma+\gamma_0}$ estimate as mentioned at the beginning of this proof, we have that
\[
[g]_{\gamma_0/\sigma,\gamma_0;Q_1}\le C\|u\|_{\gamma_0/\sigma,\gamma_0}+C[f]_{\gamma_0/\sigma,\gamma_0;Q_2}.
\]
Similarly, we have
\begin{align*}
[b^hDu]_{\gamma_0/\sigma,\gamma_0;Q_1}\le C\|Du\|_{\gamma_0/\sigma,\gamma_0;Q_1}\le C\|u\|_{\gamma_0/\sigma,\gamma_0}+C[f]_{\gamma_0/\sigma,\gamma_0;Q_2}.
\end{align*}
Therefore,
\[
\begin{split}
[u^h]_{1+\gamma_0/\sigma,\sigma+\gamma_0;Q_{3/4}}\le C\|u\|_{\gamma/\sigma,\gamma}+C[f]_{\gamma/\sigma,\gamma;Q_2}.
\end{split}
\]
 Note that we assumed that $\sigma+\gamma>2$ and thus, $\sigma>1$.
Also $1<\sigma+\gamma_0<2$.
Then we have
 \[
\begin{split}
[(Du)^h]^x_{\sigma+\gamma_0-1;Q_{3/4}}\le C\|u\|_{\gamma/\sigma,\gamma}+C[f]_{\gamma/\sigma,\gamma;Q_2},
\end{split}
\]
that is
\[
\frac{|Du(t,x+2he)-2Du(t,x+he)+Du(t,x)|}{h^{\sigma+\gamma-1}}\le C\|u\|_{\gamma/\sigma,\gamma}+C[f]_{\gamma/\sigma,\gamma;Q_2}
\]
for all $(t,x)\in Q_{1/2}$ and $h\le 1/20$. Making use of \eqref{eq 2.1} and sending $j\to\infty$ there, we have
\[
\begin{split}
|Du(t,x+he)-Du(t,x)-hD^2u(t,x)\cdot e|&\le Ch^{\sigma+\gamma-1}\sum_{k=1}^\infty2^{-k(\sigma+\gamma-2)}(\|u\|_{\gamma/\sigma,\gamma}+[f]_{\gamma/\sigma,\gamma;Q_2})\\
&\le C(\|u\|_{\gamma/\sigma,\gamma}+[f]_{\gamma/\sigma,\gamma;Q_2})h^{\sigma+\gamma-1},
\end{split}
\]
from which we have
\begin{equation}\label{eq:auxhigherorder}
[u]^x_{\sigma+\gamma;Q_{1/2}}\le C\|u\|_{\gamma/\sigma,\gamma}+C[f]_{\gamma/\sigma,\gamma;Q_2}.
\end{equation}

Similarly, we can use the difference quotients in $t$. For $s\in(0,1/10)$, let
\begin{equation}\label{eq:differencequotient}
u^s(t,x)=\frac{u(t, x)-u(t-s,x)}{s^{\frac{\gamma-\gamma_0}{\sigma}}}.
\end{equation}
By similar arguments, we have
\[
\begin{split}
[u^s]^t_{1+\gamma_0/\sigma;Q_{1/2}}\le C\|u\|_{\gamma/\sigma,\gamma}+C[f]_{\gamma/\sigma,\gamma;Q_2},
\end{split}
\]
that is
\[
[(u_t)^s]^t_{\gamma_0/\sigma;Q_{1/2}}\le C\|u\|_{\gamma/\sigma,\gamma}+C[f]_{\gamma/\sigma,\gamma;Q_2}.
\]
The same arguments in the above (noticing $\sigma>\gamma$) will lead to
\[
\begin{split}
[u]^t_{1+\gamma/\sigma;Q_{1/2}}\le C\|u\|_{\gamma/\sigma,\gamma}+C[f]_{\gamma/\sigma,\gamma;Q_2}.
\end{split}
\]
This estimate, together with \eqref{eq:auxhigherorder}, implies
\[
[u]_{1+\gamma/\sigma,\sigma+\gamma;Q_{1/2}}\le C\|u\|_{\gamma/\sigma,\gamma}+C[f]_{\gamma/\sigma,\gamma;Q_2}.
\]

We remark that actually the proof of the other situation $\sigma+\gamma \in(0,1)\cup(1,2)$ is exactly the same as above. 
\end{proof}

\section{Linear parabolic  equations with measurable coefficient in $t$}

In this section, we consider the linear equation \eqref{eq:linear}, where $K$, $b$, and $f$ are Dini continuous in $x$ but only measurable in the time variable $t$. We first need a proposition for the case that $K$ does not depend on $x$, and $b\equiv 0$.

\begin{proposition}\label{prop3.1-linear}
Let $\sigma\in(0,2)$ and $0< \lambda \le \Lambda$. Assume that $K$ does not depend on $x$, and $b\equiv 0$. Let $\alpha\in (0,1)$ such that $\sigma+\alpha$ is not an integer. Suppose $u\in C_x^{\sigma+\alpha}(Q_1)\cap C^{\alpha/\sigma,\alpha}((-2^\sigma,0)\times\bR^d)$ is a solution of \eqref{eq:linear} in $Q_1$.
Then,
\begin{equation*}
[u]^x_{\alpha+\sigma;Q_{1/2}}\le C\sum_{j=1}^\infty2^{-j\sigma}M_j+C[f]^x_{\alpha;Q_1},
\end{equation*}
where
\begin{equation*}
M_j = \sup_{^{(t,x),(t,x')\in (-1,0)\times B_{2^j},}_{0<|x-x'|<2}}\frac{|u(t,x)-u(t,x')|}{|x-x'|^\alpha}
\end{equation*}
and $C>0$ is a constant depending only on $d$, $\sigma$, $\lambda$, $\Lambda$, and $N_0$, and is uniformly bounded as $\sigma\to 2^-$.
\end{proposition}
\begin{proof}
We only prove the case that $\sigma+\alpha>2$ as before. Let $\eta$ be a cut-off function such that $\eta\in C^\infty_c(\widehat Q_{3/4})$ and $\eta\equiv 1$ in $Q_{1/2}$. Let $w(t,x)=u(t,x)-u(t,0)$, $\tilde f(t,x)=f(t,x)-f(t,0)$ and $v=\eta w$. Then $v$ satisfies
\[
v_t=Lv+h+\eta_tw+\eta \tilde f-\eta g(t)\quad\mbox{in }(-2^\sigma,0)\times\bR^d,
\]
where
\[
\begin{split}
h&=\eta Lw- L(\eta w)\\
&=\int_{\bR} \Big((\eta(t,x)-\eta(t,x+y))w(t,x+y)+y^TD\eta(t,x)w(t,x)\Big)K(t,y)\,d y
\end{split}
\]
and
\[
g(t)=(Lu)(t,0).
\]
By Theorem 4 in \cite{MP}, we have
\[
\|v\|^x_{\sigma+\alpha}\le C \|h+\eta_tw+\eta \tilde f-\eta g(t)\|^x_{\alpha}.
\]
From (3.18) in \cite{DZ162} and (3.23) in \cite{DZ162}, we have
\[
\begin{split}
\|h\|^x_{\alpha}&\le C (\|w\|^x_{\alpha;(-1,0)\times\bR^d}+[\nabla w]^x_{\alpha;Q_{15/16}})\\
&\le C (\|u\|^x_{\alpha;(-1,0)\times\bR^d}+[\nabla u]^x_{\alpha;Q_{15/16}}).
\end{split}
\]
It is clear that
\[
\begin{split}
\|\eta g(t)\|^x_{\alpha}&\le C(\|D^2u\|_{L_\infty(Q_{7/8})}
+\|D u\|_{L_\infty(Q_{7/8})}+\|u\|_{L_\infty((-1,0)\times\bR^d)}),\\
\|\eta \tilde f\|^x_{\alpha}&\le C[f]^x_{\alpha,Q_{3/4}}.
\end{split}
\]
Therefore, we have
\[
[u]^x_{\sigma+\alpha;Q_{1/2}}\le C (\|D^2u\|_{L_\infty(Q_{7/8})}
+\|Du\|_{L_\infty(Q_{7/8})}+[\nabla u]^x_{\alpha;Q_{15/16}}+\|u\|^x_{\alpha;(-1,0)\times\bR^d}+[f]^x_{\alpha;Q_{3/4}}).
\]
The same interpolation arguments of the proof of Theorem 1.1 in \cite{DZ16} lead to
\[
[u]^x_{\sigma+\alpha;Q_{1/2}}\le C (\|u\|^x_{\alpha;(-1,0)\times\bR^d}+[f]^x_{\alpha;Q_{3/4}}).
\]
Then as in the proof of Proposition \ref{prop3.1} (see also  \cite[Corollary 4.6]{DZ16}), applying this estimate to the equation of $\tilde v:=\tilde \eta(u(t,x)-u(t,0))$, where $\tilde\eta\in C_0^\infty(\widehat Q_{15/16})$ satisfying $\tilde \eta=1$ in $Q_{3/4}$, we have the desired estimates for $[u]^x_{\alpha+\sigma;Q_{1/2}}$.
\end{proof}

\begin{proposition}
	\label{prop:linear}
Suppose that \eqref{eq:linear} is satisfied in $Q_2$. Then under the conditions of Theorem \ref{thm 2}, we have
\begin{equation}
	\label{eq:linearglobal}
[u]_{\sigma;Q_{1/2}}^x\le C\|u\|_{\alpha}^x+C\sum_{k=1}^\infty\omega_f(2^{-k}),
\end{equation}
where $C>0$ is a constant depending only on $d$, $\lambda$, $\Lambda$, $\omega_a$, $\omega_b$, $N_0$ and $\sigma$. 
\end{proposition}
\begin{proof} We will consider three cases separately.

\medskip

\textbf{Case 1}: $\sigma\in (0,1)$.

\medskip

For $k\in \mathbb{N}$, let $v$ be the solution of
\begin{equation}\label{eq:existenceofapproximation}
\begin{cases}
\partial_t v = L(t,0)v+f(t,0)\quad& \text{for}\,\,x\in B_{2^{-k}}, \mbox{and almost every }t\in(-2^{-\sigma k},0]  \\
v = u\quad & \text{in}\,\, \big((-2^{-k\sigma},0)\times B_{2^{-k}}^c\big)\cup \big(\{t=-2^{-k\sigma}\}\times B_{2^{-k}}\big).
\end{cases}
\end{equation}
We sketch the proof of the existence of such $v$ as follows. Let $K^\varepsilon(t,0,y)$ and $f^\varepsilon(t,0)$ be the mollifications of $K(t,0,y)$ and $f(t,0)$ in $t$. Then there exists $v^\varepsilon$ satisfing
\begin{equation}\label{eq:auxexistence}
\begin{cases}
\partial_t v^\varepsilon = L^\varepsilon(t,0)v^\varepsilon+f^\varepsilon(t,0)\quad& \text{in}\,\,Q_{2^{-k}},\\
v^\varepsilon = u\quad & \text{in}\,\, \big((-2^{-k\sigma},0)\times B_{2^{-k}}^c\big)\cup \big(\{t=-2^{-k\sigma}\}\times B_{2^{-k}}\big).
\end{cases}
\end{equation}
Since this equation is uniformly elliptic, we have the global uniform H\"older estimate of $v^\varepsilon$ which is independent of $\varepsilon$. Thus, there exists a subsequence converging locally uniformly to a global H\"older continuous function $v$. On the other hand, by  Proposition \ref{prop3.1-linear}, we can reselect a subsequence such that for almost every time, they converge to $v$ locally uniformly in $C_x^{\sigma+\alpha}(B_{2^{-k}})$. Since we have from \eqref{eq:auxexistence} that for all $t\in(-2^{-k\sigma},0]$,
\[
\begin{cases}
v^\varepsilon(t,x)=u(-2^{-k\sigma},x)+\int_{-2^{-k\sigma}}^t L^\varepsilon(\tau,0)v^\varepsilon(\tau,x)\,d\tau+\int_{-2^{-k\sigma}}^t  f^\varepsilon(\tau,0)\,d\tau \quad \text{in}\,\,Q_{2^{-k}},\\
v^\varepsilon = u\quad\quad\quad\quad\quad  \text{in}\,\, \big((-2^{-k\sigma},0)\times B_{2^{-k}}^c\big)\cup \big(\{t=-2^{-k\sigma}\}\times B_{2^{-k}}\big),
\end{cases}
\]
we can send $\varepsilon\to 0$, using dominated convergence theorem, to obtain
\[
\begin{cases}
v(t,x)=u(-2^{-k\sigma},x)+\int_{-2^{-k\sigma}}^t L(\tau,0)v(\tau,x)\,d\tau+\int_{-2^{-k\sigma}}^t  f(\tau,0)\,d\tau \quad \text{in}\,\,Q_{2^{-k}},\\
v = u\quad\quad\quad\quad\quad  \text{in}\,\, \big((-2^{-k\sigma},0)\times B_{2^{-k}}^c\big)\cup \big(\{t=-2^{-k\sigma}\}\times B_{2^{-k}}\big).
\end{cases}
\]
This proves \eqref{eq:existenceofapproximation}. Moreover, we have from the estimates of $v^\varepsilon$ in Proposition \ref{prop3.1-linear} by sending $\varepsilon\to 0$, that
\begin{align}
[v]^x_{\alpha+\sigma;Q_{2^{-k-1}}} \le C\sum_{j=1}^\infty 2^{(k-j)\sigma}M_j+ C2^{k\sigma}[v]^x_{\alpha;Q_{2^{-k}}},
	\label{eq 1.121-2}
\end{align}
where  $\alpha\in(0,1)$ satisfying $\sigma+\alpha<1$ and
\begin{equation*}
M_j = \sup_{\substack{(t,x),(t,x')\in (-2^{-k\sigma},0)\times B_{2^{j-k}},\\0<|x-x'|<2^{-k+1}}}\frac{|u(t,x)-u(t,x')|}
{|x-x'|^\alpha}.
\end{equation*}

Let $k_0\ge 1$ be an integer to be specified. From \eqref{eq 1.121-2}, we have
\begin{equation}\label{eq 3.3-2}
[v]^x_{\alpha;Q_{2^{-k-k_0}}} \le C2^{-(k+k_0)\sigma}\sum_{j=1}^k 2^{(k-j)\sigma}M_j+ C 2^{-(k+k_0)\sigma}[u]^x_\alpha + C2^{-k_0\sigma}[v]^x_{\alpha;Q_{2^{-k}}}.
\end{equation}
Let $w: =u-v$ which satisfies
\begin{equation*}
\begin{cases}
\partial_tw=L(t,0)w +C_k\quad &\text{in}\,\, Q_{2^{-k}},\\
w = 0\quad &\text{in}\,\, \big((-2^{-k\sigma},0)\times B_{2^{-k}}^c\big)\cup \big(\{t=-2^{-k\sigma}\}\times B_{2^{-k}}\big),
\end{cases}
\end{equation*}
where
$$
C_k(t,x) = f(t,x)-f(t,0)+(L(t,x)-L(t,0))u.
$$
It is easily seen that
\begin{align*}
	\|C_k\|_{L_\infty(Q_{2^{-k}})}
	&\le \omega_f(2^{-k}) + C\omega_a(2^{-k})\Big(\sup_{(t_0,x_0)\in Q_{2^{-k}}}\sum_{j=0}^\infty 2^{j(\sigma-\alpha)}[u]^x_{\alpha;Q_{2^{-j}}(t_0,x_0)}+\|u\|_{L_\infty}\Big).
\end{align*}
Then by the H\"older estimate \cite[Lemma 2.5]{DZ16}, we have
\begin{align}
	\nonumber
&[w]_{\alpha/\sigma,\alpha;Q_{2^{-k}}}\le C2^{-k(\sigma-\alpha)}C_k
\\
                    	\label{eq 3.5-2}
&\le C2^{-k(\sigma-\alpha)}\Big[\omega_f(2^{-k})+\omega_a(2^{-k})\Big(\sup_{(t_0,x_0)\in Q_{2^{-k}}}\sum_{j=0}^\infty 2^{j(\sigma-\alpha)}[u]^x_{\alpha;Q_{2^{-j}}(t_0,x_0)}+\|u\|_{L_\infty}\Big)
\Big].
\end{align}
Combining \eqref{eq 3.3-2} and \eqref{eq 3.5-2} yields
\begin{align}
	\nonumber	
	&2^{(k+k_0)(\sigma-\alpha)}[u]^x_{\alpha;Q_{2^{-k-k_0}}}\\
	\nonumber
	&\le C2^{-(k+k_0)\alpha}\sum_{j=1}^k2^{(k-j)\sigma} [u]^x_{\alpha;(-2^{-k\sigma},0)\times B_{2^{j-k}}}\\
	\nonumber
&\,\,+C2^{-(k+k_0)\alpha}[u]^x_{\alpha}+C2^{-k_0\alpha+k(\sigma-\alpha)} [u]^x_{\alpha;(-2^{-k\sigma},0)\times B_{2^{-k}}}+C2^{k_0(\sigma-\alpha)}\omega_f(2^{-k})\nonumber\\
	\label{eq 3.7-2}	
&\,\,
+C2^{k_0(\sigma-\alpha)}\omega_a(2^{-k})\Big(\sup_{(t_0,x_0)\in Q_{2^{-k}}}\sum_{j=0}^\infty 2^{j(\sigma-\alpha)}[u]^x_{\alpha;Q_{2^{-j}(t_0,x_0)}}
+\|u\|_{L_\infty}\Big).
\end{align}
Let $Q^l$, $l=\ell_0,\ell_0+1,\cdots$ be those in the proof of Proposition \ref{prop3.21}. By translation of the coordinates, from \eqref{eq 3.7-2} we have for $l\ge \ell_0, k\ge l+1$,
\begin{align}
	\nonumber
&2^{(k+k_0)(\sigma-\alpha)}\sup_{(t_0,x_0)\in Q^l}[u]^x_{\alpha;Q_{2^{-k-k_0}}(t_0,x_0)}\\
	\nonumber
&\le  C2^{-(k+k_0)\alpha}\sup_{(t_0,x_0)\in Q^l}\sum_{j=0}^k2^{(k-j)\sigma}[u]^x_{\alpha;(t_0-2^{-k\sigma},t_0)\times B_{2^{j-k}}(x_0)}+C2^{-(k+k_0)\alpha}[u]^x_{\alpha}\\
   \label{eq 3.8-2}
&\quad+C2^{k_0(\sigma-\alpha)}\Big[ \omega_f(2^{-k})+\omega_a(2^{-k})\cdot\Big(\sup_{(t_0,x_0)\in Q^{l+1}}\sum_{j=0}^\infty 2^{j(\sigma-\alpha)}[u]^x_{\alpha;Q_{2^{-j}}(t_0,x_0)}	+\|u\|_{L_\infty}\Big)\Big].
\end{align}
Then we take the sum \eqref{eq 3.8-2} in $k = l+1,l+2, \ldots$ to obtain
\begin{align*}
	\nonumber
	&\sum_{k=l+1}^\infty 2^{(k+k_0)(\sigma-\alpha)}\sup_{(t_0,x_0)\in Q^l}[u]^x_{\alpha;Q_{2^{-k-k_0}}(t_0,x_0)}\\
	\nonumber
	&\le C\sum_{k=l+1}^\infty 2^{-(k+k_0)\alpha} \sup_{(t_0,x_0)\in Q^l}\sum_{j=0}^k2^{(k-j)\sigma}[u]^x_{\alpha;(t_0-2^{-k\sigma},t_0)\times B_{2^{j-k}}(x_0)} +C2^{-(l+k_0)\alpha}[u]^x_{\alpha} \\
         \nonumber
&\quad+C2^{k_0(\sigma-\alpha)}\sum_{k=l+1}^\infty\omega_f(2^{-k})
+C2^{k_0(\sigma-\alpha)} \\
	\nonumber &\qquad \cdot
\sum_{k=l+1}^\infty\omega_a(2^{-k})\Big(\sum_{j=0}^\infty 2^{j(\sigma-\alpha)}\sup_{(t_0,x_0)\in Q^{l+1}}[u]^x_{\alpha;Q_{2^{-j}}(t_0,x_0)}+ \|u\|_{L_\infty}\Big).
\end{align*}
As before, by switching the order of summations and then replacing $k$ by $k+j$, the first term on the right-hand side is bounded by
\begin{align*}
C2^{-k_0\alpha}\sum_{k=0}^\infty2^{k(\sigma-\alpha)}\sup_{(t_0,x_0)\in Q^l}[u]^x_{\alpha;Q_{2^{-k}}(t_0,x_0)}.
\end{align*}
With the above inequality, we have\begin{align*}
	\nonumber
	&\sum_{k=l+1}^\infty 2^{(k+k_0)(\sigma-\alpha)}\sup_{(t_0,x_0)\in Q^l}[u]^x_{\alpha;Q_{2^{-k-k_0}}(t_0,x_0)}\\
	\nonumber
	&\le C2^{-k_0\alpha}\sum_{k=0}^\infty2^{k(\sigma-\alpha)}\sup_{(t_0,x_0)\in Q^l}[u]^x_{\alpha;Q_{2^{-k}}(t_0,x_0)}\\
&\quad +C2^{-(l+k_0)\alpha}[u]^x_{\alpha} +C2^{k_0(\sigma-\alpha)}\sum_{k=l+1}^\infty\omega_f(2^{-k})
+C2^{k_0(\sigma-\alpha)} \\
	\nonumber &\qquad \cdot
\sum_{k=l+1}^\infty\omega_a(2^{-k})\Big(\sum_{j=0}^\infty 2^{j(\sigma-\alpha)}\sup_{(t_0,x_0)\in Q^{l+1}}[u]^x_{\alpha;Q_{2^{-j}}(t_0,x_0)}+ \|u\|_{L_\infty}\Big).
\end{align*}
The bound above together with the obvious inequality
\begin{equation*}
\sum_{j=0}^{l+k_0}2^{j(\sigma-\alpha)}\sup_{(t_0,x_0)\in Q^l}[u]^x_{\alpha;Q_{2^{-j}}(t_0,x_0)} \le C2^{(l+k_0)(\sigma-\alpha)}[u]^x_{\alpha},
\end{equation*}
implies that
\begin{align*}
	&\sum_{j=0}^\infty 2^{j(\sigma-\alpha)}\sup_{(t_0,x_0)\in Q^l}[u]^x_{\alpha;Q_{2^{-j}}(t_0,x_0)}\\
	&\le C2^{-k_0\alpha}\sum_{j=0}^\infty 2^{j(\sigma-\alpha)}\sup_{(t_0,x_0)\in Q^l}[u]^x_{\alpha;Q_{2^{-j}}(t_0,x_0)}\\
	&\,\,+C2^{-(l+k_0)\alpha}[u]^x_{\alpha}+C2^{(l+k_0)(\sigma-\alpha)}[u]^x_{\alpha}
+C2^{k_0(\sigma-\alpha)}\sum_{k=l+1}^\infty
\omega_f(2^{-k})\\
	&\,\,+C2^{k_0(\sigma-\alpha)}\sum_{k=l+1}^\infty \omega_a(2^{-k})\cdot\Big(\sum_{j=0}^\infty 2^{j(\sigma-\alpha)}\sup_{(t_0,x_0)\in Q^{l+1}}[u]^x_{\alpha;Q_{2^{-j}}(t_0,x_0)}
+\|u\|_{L_\infty}\Big).
\end{align*}
By first choosing $k_0$ sufficiently large and then $\ell_0$ sufficiently large (recalling $l\ge\ell_0$), we get
\begin{align*}
&\sum_{j=0}^\infty 2^{j(\sigma-\alpha)}\sup_{(t_0,x_0)\in Q^l}[u]^x_{\alpha;Q_{2^{-j}}(t_0,x_0)}\\
&\le \frac 14 \sum_{j=0}^\infty 2^{j(\sigma-\alpha)}\sup_{(t_0,x_0)\in Q^{l+1}}[u]^x_{\alpha;Q_{2^{-j}}(t_0,x_0)}+C2^{(l+k_0)(\sigma-\alpha)}\|u\|^x_{\alpha}+ C\sum_{k=1}^\infty \omega_f(2^{-k}).
\end{align*}
This implies
\[
\sum_{j=0}^\infty 2^{j(\sigma-\alpha)}\sup_{(t_0,x_0)\in Q^l}[u]^x_{\alpha;Q_{2^{-j}}(t_0,x_0)}\le C\|u\|^x_{\alpha}+ C\sum_{k=1}^\infty \omega_f(2^{-k}),
\]
which together with Lemma \ref{lem2.1} $(i)$  gives \eqref{eq:linearglobal}.

\medskip

\textbf{Case 2:} $\sigma\in (1,2)$.
\medskip

For $k\in \mathbb{N}$, let $v_M$ be the solution of
\begin{equation*}
\begin{cases}
	\partial_tv_{M} = L(t,0)v_M+f(t,0)+b(t,0)Du(t,0)-\partial_t p_0\quad \,\,\text{in}\,\,Q_{2^{-k}}\\
	v_M = g_M\qquad\qquad\qquad\qquad \,\,\text{in}\,\,\big((-2^{-k\sigma},0)\times B_{2^{-k}}^c\big)\cup \big(\{t=-2^{-k\sigma}\}\times B_{2^{-k}}\big),
	\end{cases}
\end{equation*}
where $M\ge 2\|u-p_0\|_{L_\infty(Q_{2^{-k}})}$ is a constant to be specified later,
\begin{equation*}
g_M = \max(\min(u-p_0,M),-M),
\end{equation*}
and
$$
p_0=p_0(t,x)=u^{(2^{-k})}(t,0)+x\cdot (Du^{(2^{-k})})(0,0),
$$ and $u^{(2^{-k})}$ is the mollification of $u$ in the $x$-variable only:
\begin{align*}
u^{(R)}(t,x) = \int_{\bR^{d}}u(t,x-Ry)\zeta(y)\,dy
\end{align*}
with $\zeta\in C_0^\infty(B_1)$ being a radial nonnegative function with unit integral.

Note that $[\partial _t p_0]_\alpha^x=0$. By Proposition \ref{prop3.1-linear}, we have
\begin{align*}
[v_M]^x_{\alpha+\sigma;Q_{2^{-k-1}}} \le C\sum_{j=1}^\infty 2^{(k-j)\sigma}M_j + C2^{k\sigma}[v_M]^x_{\alpha;Q_{2^{-k}}},
\end{align*}
where $\alpha\in(0,\min\{2-\sigma,(\sigma-1)/2\})$ and
\begin{align*}
M_j = \sup_{\substack{(t,x),(t,x')\in (-2^{-k\sigma},0)\times B_{2^{j-k}}\\ 0< |x-x'|<2^{-k+1}}}\frac{|u(t,x)-p_0(t,x)-u(t,x^\prime)
+p_0(t,x')|}{|x-x'|^\alpha}.
\end{align*}
From an estimate similar to Lemma \ref{lem2.7} with $\sigma\in (1, 2)$, it follows that for $j > k$, we have
\[
M_j \le C [u]^x_{\alpha;(-2^{-k\sigma},0)\times \bR^d},
\]
and thus,
\begin{align}
	\nonumber
	&[v_M]^x_{\alpha+\sigma;Q_{2^{-k-1}}} \le C\sum_{j=1}^\infty 2^{(k-j)\sigma}M_j + C2^{k\sigma}[v_M]^x_{\alpha;Q_{2^{-k}}}\\
	\label{eq 3.9-2}
	&\le C\sum_{j=1}^k 2^{(k-j)\sigma}M_j +C[u]^x_{\alpha;(-2^{-k\sigma},0)\times \bR^d}+ C2^{k\sigma}[v_M]^x_{\alpha;Q_{2^{-k}}}.
\end{align}
From the equation of $v_M$, we have for $h, \tau>0$ small, $(t,x)\in Q_{2^{-k-1}}$ and $|e|=1$ that
\begin{align*}
&\frac{v_M(t,x+he)-v_M(t,x)}{h}-\frac{v_M(t-\tau,x+he)-v_M(t-\tau,x)}{h}\\
&=\frac{1}{h} \int_{-\tau}^0 \big[(v_M)_s (t+s,x+he)-(v_M)_s (t+s,x)\big]\,ds\\
&=\frac{1}{h} \int_{-\tau}^0\int_{\bR^d} \big[\delta v_M (t+s,x+he,y)-\delta v_M (t+s,x,y)\big]K(t+s,0,y)\,dy ds.
\end{align*}
Using an argument similar to the proof of Lemma \ref{lem2.1}, and \eqref{eq 3.9-2}
we then have
\begin{align}
                                \label{eq3.03}
&[Dv_M]_{(\alpha+\sigma-1)/\sigma,\alpha+\sigma-1;Q_{2^{-k-2}}}
\le C[v_M]^x_{\alpha+\sigma;Q_{2^{-k-1}}}+C\sum_{j=1}^\infty 2^{(k-j)\sigma}M_j+ C2^{k\sigma}[v_M]^x_{\alpha;Q_{2^{-k}}}\nonumber\\
&\le C\sum_{j=1}^k 2^{(k-j)\sigma}M_j +C[u]^x_{\alpha;(-2^{-k\sigma},0)\times \bR^d}+ C2^{k\sigma}[v_M]^x_{\alpha;Q_{2^{-k}}}.
\end{align}
Let
\begin{equation}
\label{eq4.35}
p_1=p_1(t,x)=v_M(t,0)+x\cdot Dv_M(0,0).
\end{equation}
The above inequality and \eqref{eq 3.9-2} imply
\begin{align}
	\nonumber
	&[v_M-p_1]_{\alpha/\sigma,\alpha;Q_{2^{-k-k_0}}} \le C2^{-(k+k_0)\sigma}\sum_{j=1}^k2^{(k-j)\sigma}M_j\\
	\label{eq 3.10-2} &\quad +C2^{-(k+k_0)\sigma}[u]^x_{\alpha;(-2^{-k\sigma},0)\times \bR^d}+C2^{-k_0\sigma}[v_M]^x_{\alpha;Q_{2^{-k}}}.
\end{align}
Next $w_M: = g_M-v_M$, which equals to $u-p_0-v_M$ in $Q_{2^{-k}}$, satisfies
\begin{equation*}
\begin{cases}
\partial_t w_M= L(t,0)w_M + h_M +C_k\quad &\text{in}\,\,Q_{2^{-k}}\\
w_M = 0\quad &\text{in}\,\, \big((-2^{-k\sigma},0)\times B_{2^{-k}}^c\big)\cup \big(\{t=-2^{-k\sigma}\}\times B_{2^{-k}}\big),
\end{cases}
\end{equation*}
where
\begin{equation*}
h_M: = L(t,0)\big(u-p_0-g_M\big)
\end{equation*}
and
\begin{align*}
C_k = f-f(t,0)+b Du-b(t,0)Du(t,0)+(L-L(t,0))u.
\end{align*}
It follows easily that
\begin{align*}
	|C_k|&\le \omega_f(2^{-k})+\omega_b(2^{-k})\|Du\|_{L_\infty(Q_{2^{-k}})}
+\|b\|_{L_\infty}2^{-k\alpha}
[Du]^x_{\alpha;Q_{2^{-k}}}\\
	&\quad+C\omega_a(2^{-k})\sup_{(t_0,x_0)\in Q_{2^{-k}}}\sum_{j=0}^\infty 2^{j(\sigma-\alpha)}\sup_{t\in(t_0-2^{-j\sigma},t_0)}[u(t,\cdot)-p_{t,x_0}]^x_{\alpha;B_{2^{-j}}(x_0)}\\
&\quad +	C\omega_a(2^{-k})\Big(\|Du\|_{L_\infty(Q_{2^{-k}})}+\|u\|_{L_\infty}\Big),
\end{align*}
where $p_{t,x_0}=p_{t,x_0}(x)$ is the first-order Taylor's expansion of $u$ with respect to $x$ at $(t,x_0)$. From Lemma \ref{lem2.2}, we obtain
\begin{align*}
	|C_k|&\le \omega_f(2^{-k})+\omega_b(2^{-k})\|Du\|_{L_\infty(Q_{2^{-k}})}
+\|b\|_{L_\infty}2^{-k\alpha}
[Du]^x_{\alpha;Q_{2^{-k}}}\\
	&\quad +C\omega_a(2^{-k})\sup_{(t_0,x_0)\in Q_{2^{-k}}}\sum_{j=0}^\infty 2^{j(\sigma-\alpha)}\sup_{t\in(t_0-2^{-j\sigma},t_0)}\inf_{p\in \mathcal{P}_x}[u(t,\cdot)-p]^x_{\alpha;B_{2^{-j}}(x_0)}\\
&\quad +	C\omega_a(2^{-k})\Big(\|Du\|_{L_\infty(Q_{2^{-k}})}+\|u\|_{L_\infty}\Big).
\end{align*}
By the dominated convergence theorem, it is easy to see that
\begin{equation*}
\|h_M\|_{L_\infty(Q_{2^{-k}})}\rightarrow 0\quad \text{as}\quad M\rightarrow \infty.
\end{equation*}
Thus, similar to \eqref{eq 3.5}, choosing $M$ sufficiently large so that
\begin{equation*}
\|h_M\|_{L_\infty(Q_{2^{-k}})}\le C_k/2,
\end{equation*}
we have
\begin{align}
	\nonumber	\label{eq 3.11-2}
	&[w_M]_{\alpha/\sigma,\alpha;Q_{2^{-k}}}\\
	\nonumber
	&\le C2^{-k(\sigma-\alpha)}\Big(\omega_f(2^{-k})
+\big(\omega_b(2^{-k})+\omega_a(2^{-k})\big)
\|Du\|_{L_\infty(Q_{2^{-k}})}+2^{-k\alpha}[Du]^x_{\alpha;Q_{2^{-k}}}
	\\
	&\quad+\omega_a(2^{-k})\big(\sup_{(t_0,x_0)\in Q_{2^{-k}}}\sum_{j=0}^\infty 2^{j(\sigma-\alpha)}\sup_{t\in(t_0-2^{-j\sigma},t_0)}\inf_{p\in \mathcal{P}_x}[u(t,\cdot)-p]^x_{\alpha;B_{2^{-j}}(x_0)} +\|u\|_{L_\infty}\big)\Big).
\end{align}
Clearly,
\begin{equation}
                        \label{eq7.58}
\sup_{t\in(t_0-2^{-j\sigma},t_0)}\inf_{p\in \mathcal{P}_x}[u(t,\cdot)-p]^x_{\alpha;B_{2^{-j}}(x_0)}
\le \inf_{p\in \widetilde{\mathcal{P}}}[u-p]^x_{\alpha;Q_{2^{-j}}(t_0,x_0)}.
\end{equation}
By Lemma \ref{lem2.7} (more precisely, its proof)
\begin{equation}\label{eq 1.311-2}
M_j\le C\inf_{p\in \widetilde{\mathcal{P}}}[u-p]^x_{\alpha;Q_{2^{j-k}}}.
\end{equation}
From the triangle inequality and Lemma \ref{lem2.7} with $j=0$,
\begin{align*}
	[v_M]^x_{\alpha;Q_{2^{-k}}} &\le [w_M]^x_{\alpha;Q_{2^{-k}}} + [u-p_0]^x_{\alpha;Q_{2^{-k}}}\\
&\le[w_M]^x_{\alpha;Q_{2^{-k}}} + C\inf_{p\in \widetilde{\mathcal{P}}}[u-p]^x_{\alpha;Q_{2^{-k}}}.
\end{align*}
For $l=1,2,\cdots,$ let $Q^l=Q_{1-2^{-l}}$.  Combining \eqref{eq 3.10-2}, \eqref{eq 3.11-2}, and \eqref{eq 1.311-2}, similar to \eqref{eq 3.8-2}, we then get
\begin{align*}
	&2^{(k+k_0)(\sigma-\alpha)}\sup_{(t_0,x_0)\in Q^l}
\inf_{p\in \widetilde{\mathcal{P}}}[u-p]_{\alpha/\sigma,\alpha;Q_{2^{-(k_0+k)}}(t_0,x_0)} \\
		&\le 2^{(k+k_0)(\sigma-\alpha)}\sup_{(t_0,x_0)\in Q^l}[u-p_0-p_1]_{\alpha/\sigma,\alpha;Q_{2^{-(k_0+k)}}(t_0,x_0)} \\
	&\le C2^{-(k+k_0)\alpha}\sup_{(t_0,x_0)\in Q^l}\sum_{j=1}^k 2^{(k-j)\sigma}\inf_{p\in \widetilde{\mathcal{P}}}[u-p]^x_{\alpha;Q_{2^{j-k}(t_0,x_0)}} \\
	&\quad+ C2^{-(k+k_0)\alpha}[u]^x_{\alpha}+\sup_{(t_0,x_0)\in Q^l}2^{-k\alpha+k_0(\sigma-\alpha)}[Du]^x_{\alpha;Q_{2^{-k}}(t_0,x_0)}\\	
	&\quad + C2^{k_0(\sigma-\alpha)}\Big[\omega_f(2^{-k})+
\big(\omega_b(2^{-k})+\omega_a(2^{-k})\big)\|Du\|_{L_\infty(Q^{l+1})}\\
	&\quad+\omega_a(2^{-k})\Big(\sum_{j=0}^\infty 2^{j(\sigma-\alpha)}\sup_{(t_0,x_0)\in Q^{l+1}}\inf_{p\in \widetilde{\mathcal{P}}}[u-p]^x_{\alpha;Q_{2^{-j}}(t_0,x_0)}+\|u\|_{L_\infty}\Big)
\Big].
\end{align*}
Summing the above inequality in $k = l+1,l+2,\ldots$ as before, we obtain
\begin{align}
	\nonumber
	&\sum_{k=l+1}^\infty 2^{(k+k_0)(\sigma-\alpha)}\sup_{(t_0,x_0)\in Q^l} \inf_{p\in \widetilde{\mathcal{P}}}[u-p]_{\alpha/\sigma,\alpha;Q_{2^{-(k_0+k)}}(t_0,x_0)}\\
	\nonumber
	&\le C2^{-k_0\alpha}\sum_{j=0}^\infty 2^{j(\sigma-\alpha)}\sup_{(t_0,x_0)\in Q^{l+1}}
\inf_{p\in \widetilde{\mathcal{P}}}[u-p]^x_{\alpha;Q_{2^{-j}}(t_0,x_0)}\\
	\nonumber
	&\,\,+ C2^{-(k_0+l)\alpha}[u]^x_{\alpha}+C2^{k_0(\sigma-\alpha)}\sum_{k=l+1}^\infty2^{-k\alpha}\sup_{(t_0,x_0)\in Q^l}[Du]^x_{\alpha;Q_{2^{-k}}(t_0,x_0)}\\
	\nonumber &\,\,+C2^{k_0(\sigma-\alpha)}\Big[\sum_{k=l+1}^\infty
\Big(\omega_f(2^{-k})+\big(\omega_b(2^{-k})
+\omega_a(2^{-k})\big)\|Du\|_{L_\infty(Q^{l+1})}\Big)\\
	\label{eq 3.13-2}
	&\,\,+\sum_{k=l+1}^\infty\omega_a(2^{-k})\Big(\|u\|_{L_\infty}+ \sum_{j=0}^\infty 2^{j(\sigma-\alpha)}\sup_{(t_0,x_0)\in Q^{l+1}}
\inf_{p\in \widetilde{\mathcal{P}}}[u-p]^x_{\alpha;Q_{2^{-j}}(t_0,x_0)}\Big)\Big],
\end{align}
and
\begin{align*}
	&\sum_{j=0}^\infty 2^{j(\sigma-\alpha)}\sup_{(t_0,x_0)\in Q^l} \inf_{p\in \widetilde{\mathcal{P}}}[u-p]_{\alpha/\sigma,\alpha;Q_{2^{-j}}(t_0,x_0)}\\
	&\le C2^{-k_0\alpha}\sum_{j=0}^\infty 2^{j(\sigma-\alpha)}\sup_{(t_0,x_0)\in Q^{l+1}}
\inf_{p\in \widetilde{\mathcal{P}}}[u-p]^x_{\alpha;Q_{2^{-j}}(t_0,x_0)}\\ &\quad +C2^{k_0(\sigma-\alpha)}\|u\|_{L_\infty}+C2^{(l+k_0)(\sigma-\alpha)}[u]^x_{\alpha}+C2^{k_0(\sigma-\alpha)-l\alpha}[Du]^x_{\alpha;Q^{l+1}}\\
&\quad +C2^{k_0(\sigma-\alpha)}\sum_{k=l+1}^\infty\Big(\omega_f(2^{-k})
+(\omega_b(2^{-k})+\omega_a(2^{-k}))\|Du\|_{L_\infty(Q^{l+1})}\Big)\\
	&\quad +C2^{k_0(\sigma-\alpha)}\sum_{k=l+1}^\infty\omega_a(2^{-k}) \sum_{j=0}^\infty2^{j(\sigma-\alpha)}\sup_{(t_0,x_0)\in Q^{l+1}}
\inf_{p\in \widetilde{\mathcal{P}}}[u-p]^x_{\alpha;Q_{2^{-j}}(t_0,x_0)}.
\end{align*}
By choosing $k_0$ and $l$ sufficiently large, and using \eqref{eq 2.3a} and interpolation inequalities (recalling that $\alpha<(\sigma-1)/2$),  we obtain
\begin{align*}
	&\sum_{j=0}^\infty 2^{j(\sigma-\alpha)}\sup_{(t_0,x_0)\in Q^l} \inf_{p\in \widetilde{\mathcal{P}}}[u-p]_{\alpha/\sigma,\alpha;Q_{2^{-j}}(t_0,x_0)}\\
&\le 	\frac 14 \sum_{j=0}^\infty 2^{j(\sigma-\alpha)}\sup_{(t_0,x_0)\in Q^{l+1}}
\inf_{p\in \widetilde{\mathcal{P}}}[u-p]^x_{\alpha;Q_{2^{-j}}(t_0,x_0)}\\
&\quad\quad+ C2^{(k_0+l)(\sigma-\alpha)}\|u\|^x_{\alpha}+C\sum_{k=1}^\infty \omega_f(2^{-k}).
\end{align*}
Therefore,
\begin{align}\label{eq 3.14-2}
\sum_{j=0}^\infty 2^{j(\sigma-\alpha)}\sup_{(t_0,x_0)\in Q^l}
\inf_{p\in \widetilde{\mathcal{P}}}[u-p]_{\alpha/\sigma,\alpha;Q_{2^{-j}}(t_0,x_0)}\le C\|u\|^x_{\alpha}+C\sum_{k=1}^\infty \omega_f(2^{-k}),
\end{align}
which together with Lemma \ref{lem2.1} $(ii)$ (actually the proof of it) gives \eqref{eq:linearglobal}.
\medskip

\textbf{Case 3: $\sigma=1$.}
\medskip

Set $B_0(t) = \int_{t}^0b(s,0)ds$ and we define
\begin{align*}
\hat{u}(t,x) = u(t,x+B_0(t)),\quad \hat{f_\beta}(t,x) = f_\beta(t,x+B_0(t)),\quad \text{and}\quad \hat{b}(t,x) = b(t,x+B_0(t)).
\end{align*}
It is easy to see that in  $Q_\delta$ for some $\delta>0$,
\begin{align}
\partial_t\hat{u}(t,x) &= (\partial_tu)(t,x+B_0(t))-b(t,0)\nabla u(t,x+B_0(t))\nonumber\\
&=\hat L_\beta \hat{u} + \hat{f}_\beta+ (\hat{b}-b(t,0))\nabla \hat{u},\label{eq 2.221-2}
\end{align}
where $\hat L$ is the operator with kernel $a(t,x+B_0(t),y)|y|^{-d-\sigma}$.
For $(t,x)\in Q_{2^{-k}}$,
\begin{align*}
|\hat{f}_\beta(t,x)-\hat{f}_\beta(t,0)|\le \omega_f(2^{-k}),\\
|\hat b-b(t,0)|\le \omega_b((1+N_0)2^{-k}).
\end{align*}
Furthermore,
\begin{align*}
\|Du\|_{L_\infty} + \|\partial_t u\|_{L_\infty} \le (1+N_0)\big(\|D\hat{u}\|_{L_\infty} + \|\partial_t\hat{u}\|_{L_\infty}\big).
\end{align*}
Therefore, it is sufficient to bound $\hat{u}$. In the rest of the proof,  we estimate
the solution to \eqref{eq 2.221-2} and abuse the notation to use $u$ instead of $\hat{u}$ for simplicity.
By scaling, translation and covering arguments, we also assume $u$ satisfies the equation in  $Q_2$.

The proof is similar to the case $\sigma  \in(1,2)$ and we indeed proceed as in the previous case. Let $p_0$ and  $u^{(2^{-k})}$ be as in Case 2. 
We also assume that the solution $v$ to the following equation
\begin{equation*}
\begin{cases}
	\partial_tv =\hat L(t,0)v+\hat f(t,0)-\partial_t p_0\quad \,\,&\text{in}\,\,Q_{2^{-k}}\\
	v = u-p_0\quad \,\,&\text{in}\,\,\big((-2^{-k\sigma},0)\times B_{2^{-k}}^c\big)\cup \big(\{t=-2^{-k\sigma}\}\times B_{2^{-k}}\big),
\end{cases}
\end{equation*}
 exists without carrying out another approximation argument.
By Proposition \ref{prop3.1-linear} with $\sigma= 1$ and Lemma \ref{lem2.7} in \cite{DZ162},
\begin{align}
	\nonumber
	&[v]^x_{1+\alpha;Q_{2^{-k-1}}}\le C\sum_{j=1}^\infty 2^{k-j}M_j+C2^k[v]_{\alpha,\alpha;Q_{2^{-k}}}\\
	\nonumber
	&\le C\sum_{j=1}^\infty 2^{k-j}\sup_{t\in(-2^{-k\sigma},0)}\inf_{p\in \mathcal{P}_x}[u(t,\cdot)-p]^x_{\alpha;B_{2^{j-k}}} + C2^k[v]^x_{\alpha;Q_{2^{-k}}}\\
        \label{eq 3.16-2}
	&\le C\sum_{j=1}^k 2^{k-j}\sup_{t\in(-2^{-k\sigma},0)}\inf_{p\in \mathcal{P}_x}[u(t,\cdot)-p]^x_{\alpha;B_{2^{j-k}}} + C[u]^x_{\alpha}+C2^{k}[v]^x_{\alpha;Q_{2^{-k}}}.
\end{align}
From \eqref{eq 3.16-2} and using the equation, we obtain
\begin{align}
	\nonumber
	&[v-p_1]_{\alpha/\sigma,\alpha;Q_{2^{-k-k_0}}}\\
	\label{eq 3.16a-2}
	&\le C2^{-(k+k_0)}\sum_{j=1}^k 2^{k-j}
\inf_{p\in \widetilde{\mathcal{P}}}[u-p]^x_{\alpha;Q_{2^{j-k}}} +C2^{-k_0}[v]^x_{\alpha;Q_{2^{-k}}}+ C2^{-(k+k_0)}[u]^x_{\alpha},
\end{align}
where $p_1$ is defined as in \eqref{eq4.35}. Next $w :=u-p_0-v$ satisfies
\begin{equation*}
\begin{cases}
\partial_t w= \hat L(t,0)w + C_k\quad &\text{in}\,\,Q_{2^{-k}}\\
w= 0\quad &\text{in}\,\, \big((-2^{-k},0)\times B_{2^{-k}}^c\big)\cup \big(\{t=-2^{-k}\}\times B_{2^{-k}}\big),
\end{cases}
\end{equation*}
where by the cancellation property,
\[
C_k=\hat f-\hat f(t,0)+(\hat b-b(t,0))\nabla u+(\hat L-\hat L(t,0))u,
\]
so that
\begin{align*}
|C_k| \le &\omega_f(2^{-k})+\omega_{b}((1+N_0)2^{-k})\|Du\|_{L_\infty(Q_{2^{-k}})} + C\omega_a((1+N_0)2^{-k}) \\
&\cdot\Big(\sup_{(t_0,x_0)\in Q_{2^{-k}}}\sum_{j=0}^\infty 2^{j(1-\alpha)}\sup_{t\in(t_0-2^{-j\sigma},t_0)}\inf_{p\in \mathcal{P}_x}[u(t,\cdot)-p]^x_{\alpha;B_{2^{-j}}(x_0)}+\|u\|_{L_\infty}\Big).
\end{align*}
Clearly, for any $r\ge 0$,
$$
\omega_\bullet ((1+N_0)r)\le (2+N_0)\omega_\bullet (r).
$$
Therefore, similar to \eqref{eq 3.11-2},  we have
\begin{align}
	\nonumber
	&[w]_{\alpha,\alpha;Q_{2^{-k}}}\le C2^{-k(1-\alpha)}\big( \omega_f(2^{-k})+\omega_b(2^{-k})\|Du\|_{L_\infty(Q_{2^{-k}})}\\
	\label{eq 3.16b-2}
	&+\omega_a(2^{-k})\Big(\sup_{(t_0,x_0)\in Q_{2^{-k}}}\sum_{j=0}^\infty 2^{j(1-\alpha)}\sup_{t\in(t_0-2^{-j\sigma},t_0)}\inf_{p\in \mathcal{P}_x}[u(t,\cdot)-p]^x_{\alpha;B_{2^{-j}}(x_0)}+\|u\|_{L_\infty}\Big).
\end{align}
From the proof of \eqref{eq 2.19a} and the triangle inequality,
\begin{align*}
[v]^x_{\alpha;Q_{2^{-k}}}\le [w]^x_{\alpha;Q_{2^{-k}}}+[u-p_0]^x_{\alpha;Q_{2^{-k}}} \le [w]_{\alpha;Q_{2^{-k}}}+ C\inf_{p\in \widetilde{\mathcal{P}}}[u-p]^x_{\alpha;Q_{2^{-k}}}.
\end{align*}
For $l=1,2,\cdots,$ let $Q^l=Q_{1-2^{-l}}$.   Similar to \eqref{eq 3.8-2}, by combining \eqref{eq 3.16a-2} and \eqref{eq 3.16b-2}, shifting the coordinates, and using the above inequality, we obtain
for $l\ge 1$ and $k\ge l+1$,
\begin{equation}\label{eq:aux209}
\begin{split}
	&2^{(k+k_0)(1-\alpha)}\sup_{(t_0,x_0)\in Q^l}\inf_{p\in \widetilde{\mathcal{P}}}[u-p]_{\alpha/\sigma,\alpha;Q_{2^{-k-k_0}(t_0,x_0)}}\\
	&\le C2^{-(k+k_0)\alpha}\sum_{j=0}^k 2^{k-j}\sup_{(t_0,x_0)\in Q^l}
\inf_{p\in \tilde{\mathcal{P}}}[u-p]^x_{\alpha;Q_{2^{j-k}}(t_0,x_0)}\\
&\quad+C2^{k_0(1-\alpha)}\Big[\omega_f(2^{-k})
+\omega_b(2^{-k})\|Du\|_{L_\infty(Q^{l+1})}\\
	&\quad+\omega_a(2^{-k})\Big(\sum_{j=0}^\infty 2^{j(1-\alpha)}\sup_{(t_0,x_0)\in Q^{l+1}}\inf_{p\in \widetilde{\mathcal{P}}}[u-p]^x_{\alpha, Q_{2^{-j}}(t_0,x_0)}+\|u\|_{L_\infty}\Big)\Big]\\
	&\quad+ C2^{-(k+k_0)\alpha}[u]^x_{\alpha},
\end{split}
\end{equation}
which by summing in $k = l+1,l+2,\ldots$, implies that
\begin{align*}
	&\sum_{k=l+1}^\infty 2^{(k+k_0)(1-\alpha)}\sup_{(t_0,x_0)\in Q^l}
\inf_{p\in \widetilde{\mathcal{P}}}[u-p]_{\alpha/\sigma,\alpha;Q_{2^{-k-k_0}}(t_0,x_0)}\\
	&\le C2^{-k_0\alpha}\sum_{j=0}^\infty 2^{j(1-\alpha)}\sup_{(t_0,x_0)\in Q^{l+1}}
\inf_{p\in \widetilde{\mathcal{P}}}[u-p]^x_{\alpha;Q_{2^{-j}}(t_0,x_0)}\\
	&\quad+C2^{-(k_0+l)\alpha}[u]^x_{\alpha}+C2^{k_0(1-\alpha)}
\sum_{k=l+1}^\infty\omega_f(2^{-k}) \\
	&\quad+ C2^{k_0(1-\alpha)}\sum_{k=l+1}^\infty
\Big[\omega_b(2^{-k})\|Du\|_{L_\infty(Q^{l+1})}+\omega_a(2^{-k})\\
	&\qquad\cdot\Big(\sum_{j=0}^\infty 2^{j(1-\alpha)}\sup_{(t_0,x_0)\in Q^{l+1}}
\inf_{p\in \widetilde{\mathcal{P}}}[u-p]^x_{\alpha; Q_{2^{-j}}(t_0,x_0)}+\|u\|_{L_\infty}\Big)\Big],
\end{align*}
where for the first term on the right-hand side of \eqref{eq:aux209}, we replaced $j$ by $k-j$ and switched the order of the summation as before.
Therefore,
\begin{align*}
	&\sum_{j=0}^\infty 2^{j(1-\alpha)}\sup_{(t_0,x_0)\in Q^l}
\inf_{p\in \widetilde{\mathcal{P}}}[u-p]_{\alpha/\sigma,\alpha;Q_{2^{-j}}(t_0,x_0)}\\
	&\le C2^{-k_0\alpha}\sum_{j=0}^\infty 2^{j(1-\alpha)}\sup_{(t_0,x_0)\in Q^{l+1}}
\inf_{p\in \widetilde{\mathcal{P}}}[u-p]^x_{\alpha;Q_{2^{-j}}(t_0,x_0)}\\
	&\quad+C2^{(l+k_0)(1-\alpha)}[u]^x_{\alpha}
+C2^{k_0(1-\alpha)}\sum_{k=l+1}^\infty \omega_f(2^{-k})\\
	&\quad +C2^{k_0(1-\alpha)}\sum_{k=l+1}^\infty\omega_b(2^{-k})\|Du\|_{L_\infty(Q^{l+1})}
+C2^{k_0(1-\alpha)}\sum_{k=l+1}^\infty \omega_a(2^{-k})\\
	&\qquad\cdot\Big(\sum_{j=0}^\infty 2^{j(1-\alpha)}\sup_{(t_0,x_0)\in Q^{l+1}}
\inf_{p\in \widetilde{\mathcal{P}}}[u-p]^x_{\alpha; Q_{2^{-j}}(t_0,x_0)}+\|u\|_{L_\infty}\Big).
\end{align*}
Then we choose $k_0$ and then $l$ sufficiently large, and apply Lemma \ref{lem2.1} $(iii)$ (actually the proof it) to get
\begin{align*}
	&\sum_{j=0}^\infty 2^{j(1-\alpha)}\sup_{(t_0,x_0)\in Q^l}
\inf_{p\in \widetilde{\mathcal{P}}}[u-p]_{\alpha/\sigma,\alpha;Q_{2^{-j}}(t_0,x_0)}\\
	&\le \frac 14 \sum_{j=0}^\infty 2^{j(1-\alpha)}\sup_{(t_0,x_0)\in Q^{l+1}}
\inf_{p\in \widetilde{\mathcal{P}}}[u-p]^x_{\alpha; Q_{2^{-j}}(t_0,x_0)}\\
	&\quad\quad+C2^{(k_0+l)(1-\alpha)}\|u\|^x_{\alpha} +C\sum_{k=1}^\infty \omega_f(2^{-k}).
\end{align*}
This implies
\begin{align*}
\sum_{j=0}^\infty 2^{j(1-\alpha)}\sup_{(t_0,x_0)\in Q^l}
\inf_{p\in \widetilde{\mathcal{P}}}[u-p]_{\alpha/\sigma,\alpha;Q_{2^{-j}}(t_0,x_0)}\le C\|u\|^x_{\alpha} +C\sum_{k=1}^\infty \omega_f(2^{-k}),
\end{align*}
from which  \eqref{eq:linearglobal} follows.
\end{proof}

\begin{proof}[Proof of Theorem \ref{thm 2}]

As before, Theorem \ref{thm 2} follows from Proposition \ref{prop:linear} using the argument of freezing the coefficients. We only present the detailed proof of Theorem \ref{thm 2} for $\sigma\in (1,2)$. We omit the proof of the case $\sigma\in(0,1]$ since it is almost the same and actually is simpler.

Indeed, the proof here for $\sigma\in(1,2)$ is almost identical to that of Theorem \ref{thm 1}, so we just sketch it.

Without loss of generality, we assume the equation holds in $Q_3$.

{\em Step 1.} For $k=1,2,\ldots$, denote $Q^k := Q_{1-2^{-k}}$.
Let $\eta_k\in C_0^\infty(\widehat Q^{k+3})$ be a sequence of nonnegative smooth cutoff functions satisfying
$\eta\equiv 1$ in $Q^{k+2}$, $|\eta|\le 1$ in $Q^{k+3}$,  $\|\partial_t^jD^i\eta_k\|_{L_\infty}\le C2^{k(i+j)}$ for each $i,j\ge 0$.  Set $v_k := u\eta_k\in C^{1,\sigma+}$ and notice that in $Q^{k+1}$,
\begin{align*}
\partial_tv_k&=\eta_k \partial_tu+\partial_t\eta_k u
=\eta_k L u+\eta_kb Du+\eta_k f+\partial_t\eta_k u\\
&=L v_k+b Dv_k-b uD\eta_k+h_{k}+\eta_k f+\partial_t\eta_k u,
\end{align*}
where
\begin{align*}
h_{k}=\eta_k L u-L v_k=\int_{\bR^d}\frac{
\xi_k(t,x,y)a(t,x,y)}{|y|^{d+\sigma}}\,dy,
\end{align*}
and
\begin{align*}
\xi_k(t,x,y) &= u(t,x+y)(\eta_k(t,x+y)-\eta_k(t,x))-y\cdot D\eta_k(t,x)u(t,x)\\
& =u(t,x+y)(\eta_k(t,x+y)-\eta_k(t,x))\quad\mbox{since }D\eta_k\equiv 0\mbox{ in } Q^{k+1}.
\end{align*}

We will apply Proposition \ref{prop:linear} to the equation of $v_k$ in $Q^{k+1}$ and obtain corresponding estimates for $v_k$ in $Q^k$.

As before,  we have $\eta_kf\equiv f, \partial_t \eta_k u\equiv 0$, and $b uD\eta_k\equiv 0$ in $Q^{k+1}$. Thus, we only need to estimate the moduli of continuity of $h_k$ in $Q^{k+1}$ with respect to $x$. The same proof of \eqref{eq9.32} shows that
\begin{align}
                         \label{eq9.32-2}
\omega_h(r) := C\Big(2^{\sigma k}\|u\|_{L_\infty(Q_3)}+\sum_{j=0}^\infty2^{-j\sigma}\omega_u(2^j)\Big)\omega_a(r)+C2^{k\sigma}\omega_u(r)+C2^{k(\sigma+1)}\|u\|_{L_\infty(Q_3)}r.
\end{align}

As in the proof of Theorem \ref{thm 1}, by making use of Proposition \ref{prop:linear} to $v_k$ and interpolation inequalities, an iteration procedure will lead to
\begin{align}
\label{eq12.09-2}
[u]^x_{\sigma;Q^4}\le C\|u\|_{L_\infty(Q_3)}+C\sum_{j=0}^\infty\big(2^{-j\sigma}\omega_u(2^j) +\omega_u(2^{-j})+\omega_f(2^{-j})\big).
\end{align}
Applying this to the equation of $u(t,x)-u(t,0)$ gives to \eqref{eq12.17-2}.

Finally, since $\|v_1\|^x_{\alpha}$ is bounded by the right-hand side of \eqref{eq12.09-2}, from \eqref{eq 3.14-2}, we see that
\begin{align*}
\sum_{j=0}^\infty 2^{j(\sigma-\alpha)}\sup_{(t_0,x_0)\in Q^l}
\inf_{p\in \widetilde{\mathcal{P}}}[u-p]_{\alpha/\sigma,\alpha;Q_{2^{-j}}(t_0,x_0)}\le C
\end{align*}
for some large $l$.

This and \eqref{eq 3.13-2} with $u$ replaced by $v_1$ and $f$ replaced by $h_{1}+\eta_1 f+\partial_t\eta_1 u-b u D\eta_1$ give
\begin{align*}
&\sum_{j=k_1+1}^\infty 2^{(j+k_0)(\sigma-\alpha)}\sup_{(t_0,x_0)\in Q^{k_1}}
\inf_{p\in \widetilde{\cP}}[v_{1}-p]_{\alpha/\sigma,\alpha;Q_{2^{-j-k_0}}(t_0,x_0)}\\
&\le
C2^{-k_0\alpha}+C2^{k_0(\sigma-\alpha)}
\sum_{j=k_1}^\infty \big(\omega_f(2^{-j})+\omega_a(2^{-j})+\omega_u(2^{-j})
+\omega_b(2^{-j})+2^{-j\alpha}\big).
\end{align*}
Here we also used \eqref{eq9.32-2} with $k=1$.
Therefore, for any small $\varepsilon>0$, we can find $k_0$ sufficiently large then $k_1$ sufficiently large, depending only on $C$, $\sigma$, $N_0$, $\alpha$, $\omega_a$, $\omega_f$, $\omega_b$, and $\omega_u$, such that
$$
\sum_{j=k_1+1}^\infty 2^{(j+k_0)(\sigma-\alpha)}\sup_{(t_0,x_0)\in Q^{k_1}}
\inf_{p\in \widetilde{\cP}}[v_1-p]_{\alpha/\sigma,\alpha;Q_{2^{-j-k_0}}(t_0,x_0)}<\varepsilon.
$$
This, together with the fact that $v_1 = u$ in $Q_{1/2}$ and the proof of Lemma \ref{lem2.1} (ii), indicates that
$$
\sup_{(t_0,x_0)\in Q_{1/2}} [u]^x_{\sigma;Q_r(t_0,x_0)}\to 0 \quad\text{as}\quad r\to 0
$$
with a decay rate depending  only on $d$, $\lambda$, $N_0$, $\Lambda$, $\omega_a$, $\omega_f$, $\omega_b$, $\omega_u$, and $\sigma$.  Also, by evaluating the equation \eqref{eq:linear} on both sides and making use of \eqref{eq7.58}, \cite[Lemma 2.2]{DZ162}, and the dominated convergence theorem, we have that $\partial_t u$ is uniformly continuous in $x$ in $Q_{1/2}$ with a modulus of continuity controlled by $d$, $\sigma$, $\lambda$, $\Lambda$, $\omega_a$, $\omega_f$, $\omega_u$, $N_0$, $\omega_b$, and $\|u\|_{L_\infty}$.

Hence, the proof of the case when $\sigma\in (1,2)$ is completed.
\end{proof}

\begin{proof}[Proof of Theorem \ref{thm:linearschauder2}]
Given the proofs of Theorems \ref{thm:schauder} and \ref{thm:linearschauder}, Theorem \ref{thm:linearschauder2} can be similarly proved (actually simpler), and we omit the details.
\end{proof}

\bibliographystyle{abbrv}

\end{document}